\documentclass[12pt,a4paper]{amsart}
\makeatletter

\ifdefined\pdfpagewidth
\else
\fi
\calclayout
\makeatother

\usepackage{xcolor}

\usepackage{amsfonts}
\usepackage{amssymb}
\usepackage{graphicx}
\usepackage{mathrsfs}
\usepackage{amscd}
\usepackage{hyperref}
\usepackage{stmaryrd}
\usepackage{caption}
\usepackage{subcaption}
\usepackage{color}
\usepackage{xcolor}
\usepackage{comment}

\usepackage{tikz} 
\usepackage{tikz-cd}
\usepackage{mathtools}

\newtheorem{theorem}{Theorem}[section]
\newtheorem{lemma}[theorem]{Lemma}
\newtheorem{corollary}[theorem]{Corollary}
\newtheorem{setup}[theorem]{Setup}
\newtheorem{prop}[theorem]{Proposition}
\newtheorem{defn}[theorem]{Definition}
\newtheorem{example}[theorem]{Example}
\newtheorem{remark}[theorem]{Remark}
\newtheorem{conjecture}[theorem]{Conjecture}

\numberwithin{equation}{section}

\newcommand{\Z}{\mathbb{Z}}
\newcommand{\R}{\mathbb{R}}

\newcommand{\C}{\mathbb{C}}
\newcommand{\bS}{\mathbb{S}}

\newcommand{\bT}{\mathbf{T}}

\newcommand{\bL}{\mathbb{L}}

\newcommand{\bF}{\mathbb{F}}

\newcommand{\Hom}{\mathrm{Hom}}

\newcommand{\eff}{\mathrm{eff}}

\newcommand{\Span}{\mathrm{Span}}

\newcommand{\one}{\mathbf{1}}

\newcommand{\bP}{\mathbb{P}}

\newcommand{\Id}{\mathrm{Id}}

\title[Equivariant Lagrangian correspondence]{Equivariant Lagrangian correspondence and a conjecture of Teleman}

\author[Lau]{Siu-Cheong Lau}
\address{Department of Mathematics\\ Boston University}
\email{lau@math.bu.edu}
\author[Leung]{Nai-Chung Conan Leung}
\address{Institute of Mathematical Sciences\\ The Chinese University of Hong Kong}
\email{leung@math.cuhk.edu.hk}
\author[Li]{Yan-Lung Leon Li}
\address{Center of Geometry and Physics, Institute for Basic Science (IBS), Pohang 37673, Korea}
\email{ylli@ibs.re.kr}

\begin{document}
\maketitle

\begin{abstract}
    In this paper, we study the Floer theory of equivariant Lagrangian correspondences and apply it to derive precise relations between the disc potential of an invariant Lagrangian submanifold and that of its quotient, thereby addressing a conjecture of Teleman. Furthermore, we proved that their (equivariant) Lagrangian Floer cohomologies are isomorphic. In particular, the functor by equivariant Lagrangian correspondence induces a quasi-isomorphism between the equivariant derived Fukaya category at a regular moment-map level and the derived Fukaya category of the corresponding symplectic quotient.

    A key step is to extend Fukaya's construction of an $A_\infty$ tri-module for Lagrangian correspondences to Borel spaces.  We demonstrate that the equivariant obstruction of a Lagrangian correspondence plays an essential role, which leads to quantum corrections in the disc potentials of the quotients.  We computed the disc potential of the Lagrangian correspondence in the toric setup and relate it with mirror maps for compact semi-Fano toric manifolds.
\end{abstract}

\section{Introduction}
Let $(Y,\omega)$ be a Hamiltonian $G$-manifold, i.e. a symplectic manifold $Y$ equipped with a Hamiltonian $G$-action, where $G$ is a compact Lie group, and a moment map $\mu: Y \to \mathfrak{g}^*$. We consider a smooth symplectic quotient 
$$X=X_Q := Y\sslash_Q G = \mu^{-1}(Q) / G,$$
where $Q \subset \mathfrak{g}^*$ is a coadjoint orbit such that $G$ acts freely on $\mu^{-1}(Q)$.

This paper aims to investigate the relationship between the mirror complex geometry of a Hamiltonian $G$-manifold $Y$ and its symplectic quotient $X$.  In \cite{Teleman}, Teleman, based on toric mirror constructions by Givental \cite{Givental} and Hori-Vafa \cite{Hori-Vafa}, proposed the following conjecture when $G=T$ is abelian:

\begin{conjecture}
\label{introtelemanconj}
[Teleman \cite{Teleman}]
	\begin{enumerate}
		\item The mirror of a Hamiltonian $T$ action on a symplectic manifold $Y$ is a holomorphic fibration 
		$$ F: \check{Y} \to \check {T}_\mathbb{C} $$
		where $\check{Y}$ is the mirror of $Y$ and $\check{T}_\mathbb{C}$ is the complexified dual torus.
		\item For each $Q$ as above, the mirror of the symplectic quotient $X_Q$ is given by a fiber $F^{-1}\{\tilde{Q}\}$ for some $\tilde{Q} \in \check {T}_\mathbb{C}$.
	\end{enumerate}

 Moreover, under the Landau-Ginzburg (LG) mirror symmetry, if $(\check{Y}, W_Y)$ is an LG model of $Y$, then $(\check{X}, W_X) \coloneqq (F^{-1}(\tilde{Q}), W_Y|_{F^{-1}(\tilde{Q})})$ is an LG model of $X$.
\end{conjecture}

 \begin{remark}
In his ICM talk, Teleman further conjectured that for general $G$, the mirror of a Hamiltonian $G$ action on $Y$ is a holomorphic fibration 
$$ F: \check{Y} \to \check {G}_\mathbb{C}/\mathrm{Ad} $$
		where $\check{G}_\mathbb{C}$ is the complexified Langlands dual group and $\check{G}_\mathbb{C}/\mathrm{Ad}$ is its space of conjugacy classes, such that the mirror of $Y\sslash_0 G$ is related to a fiber of $F$ (See \cite{Telemanicmslide} for further details).
 \end{remark}

In the closed-string sector, Iritani and Sanda are constructing maps relating (equivariant) quantum D-modules $QDM_T(Y)$ and $QDM(X)$. In the present work, we prove an open-string and local version of this conjecture using equivariant Lagrangian Floer theory. We briefly describe our approach below, whose details are in Theorem \ref{telthm}.

From Floer-theoretic perspective, $\check{Y}$ is constructed by gluing local mirror charts given by $MC_{weak}(L)$, the weak Maurer-Cartan space of $L$ endowed with the disk potential $W_L$, via wall-crossing transformations \cite{CHL-toric}. When $L$ is $T$-invariant, $F$ is defined using the equivariant disk potential of $L$ due to Kim, the first-named author and Zheng \cite{KLZ}; a major portion of this paper is to justify Conjecture \ref{introtelemanconj} (2) by developing the theory of \textit{equivariant correspondence tri-modules}, as an equivariant extension of correspondence tri-modules by Fukaya \cite{Fukaya-corr}. 

To illustrate, consider the following example supporting Conjecture \ref{introtelemanconj}:

\begin{example} 
Let $Y=\C^{n+1}$ equipped with an $\bS^1$-action in the direction $(1, \dots ,1)$. Its symplectic quotient at any regular level is $\bP^n$ (see Figure \ref{fig:P2} when $n=2$).

    The Hori-Vafa mirror of $\bP^n$ (as a K\"ahler manifold) is the LG model 
    $$((\C^\times)^n, W_{\bP^n} = z_1 + \ldots z_n + \frac{q}{z_1\ldots z_n}),$$ where $q$ is the K\"ahler parameter encoding the symplectic area of the line class.  This LG model can be obtained from that of $\C^{n+1}$
    $$((\C^\times)^{n+1}, W_{\C^{n+1}} = z_1 + \ldots z_n + z_{n+1})$$ by restricting $W_{\C^{n+1}}$ on the fiber $F^{-1}\{q\}$ where $F: (\C^\times)^{n+1} \to \C^\times$ is defined as $F=z_1\ldots z_{n+1}$.  

\begin{figure}[h]
\begin{center}
\includegraphics[scale=0.25]{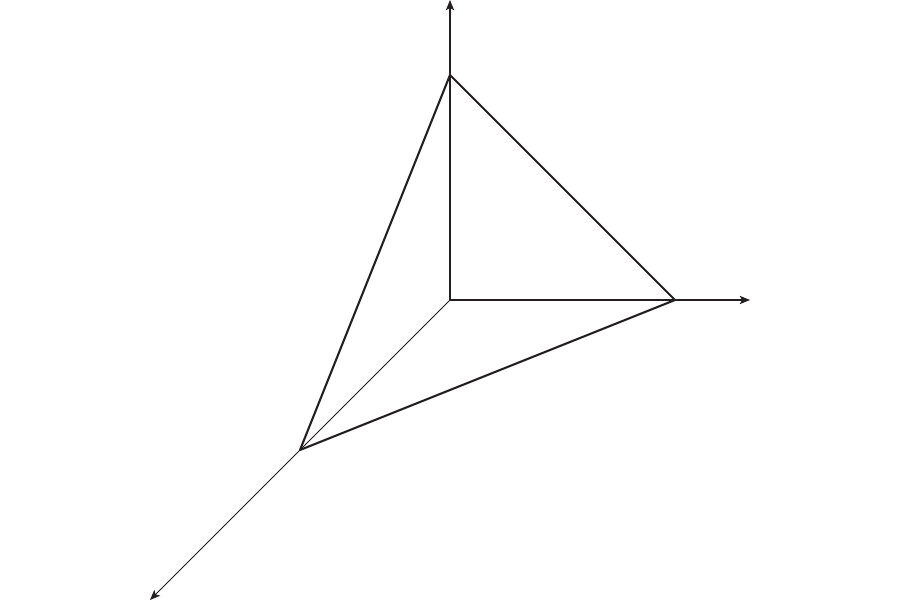}
\caption{\small{$\bP^2$ as a symplectic quotient of $\C^3$.}}
\label{fig:P2}
\end{center}
\end{figure}
    
\end{example}

However, even for compact toric Fano manifolds, non-trivial `quantum corrections' come up.  Let's consider the following example.

\begin{example} \label{ex:P1P1P1}
    Let $Y=(\bP^1)^3$ with an $\bS^1$-action in the direction $(1,1,1)$.  Its symplectic quotient equals $\bP^2$ (see Figure \ref{fig:P1P1P1}). The Hori-Vafa mirror of $Y$, setting the K\"ahler parameters $q_i=1$ for all $i$ for simplicity, is given by 
    $$((\C^\times)^{3}, W_{(\bP^1)^3} = z_1 + z_2 + z_3 + \frac{1}{z_1} + \frac{1}{z_2} + \frac{1}{z_3}).$$  Restricting to the fiber of $F=z_1z_2z_3$ at $c \in \C^\times$, we get 
    $$((\C^\times)^{2}, z_1 + z_2 + \frac{c}{z_1z_2} + \frac{1}{z_1} + \frac{1}{z_2} + \frac{z_1z_2}{c}),$$ which seems hard to be compared with the LG potential $W_{\bP^2} = z_1+z_2+\frac{q}{z_1z_2}$ of $\bP^2$. We will return to this in Example \ref{ex:P1P1P1cont} and \ref{ex:P1P1P1-2}.
\begin{figure}[h]
\begin{center}
\includegraphics[scale=0.25]{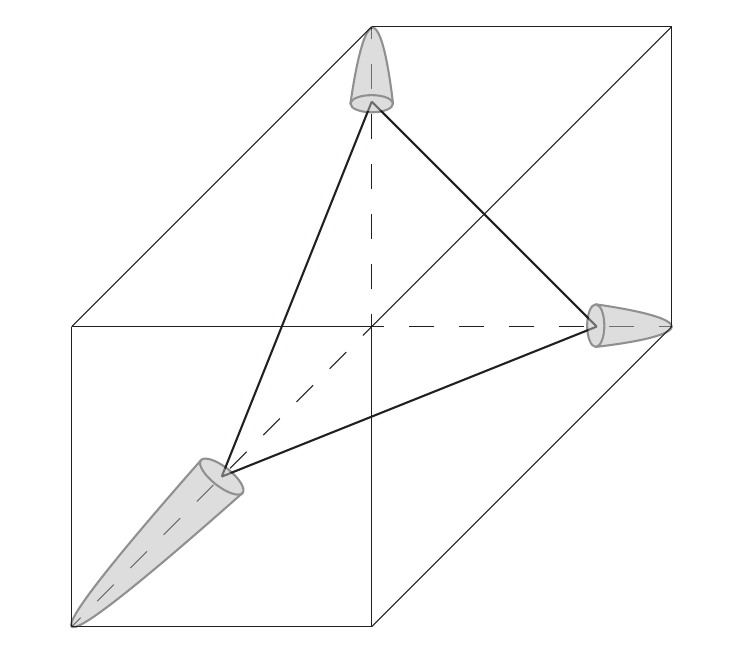}
\caption{\small{$\bP^2$ as a symplectic quotient of $(\bP^1)^3$.}}
\label{fig:P1P1P1}
\end{center}
\end{figure}
\end{example}

In this paper, we tackle the problem from the SYZ approach \cite{SYZ} and Lagrangian Floer theory \cite{FOOO10}.   By SYZ, the mirror $\check{Y}$ of the symplectic manifold $Y$ should be constructed as the complexified moduli space of (possibly degenerate) fibers of a Lagrangian torus fibration.  The construction receives quantum corrections coming from the Lagrangian deformation and obstruction theory of the fibers.  To compare the mirrors, we should find relations between the moduli space of Lagrangians in $Y$ and that in the symplectic quotient $X$.

Lagrangians in $Y$ and its symplectic quotient $X$ at $Q = \{0\}$ are related by a Lagrangian correspondence, the moment-level Lagrangian $$L^\pi := \{(y,\pi(y)) \in Y^- \times X: y \in \mu^{-1}(0)\}.$$ It relates a $G$-invariant Lagrangian $L \subset Y$ with its reduction $\bar{L} \subset X$.   Note that $L^\pi$ is diffeomorphic to $\mu^{-1}(0)$.  Moreover, $L^\pi$ is invariant under the diagonal $G$-action on $Y^- \times X$ (in which $G$ acts on $X$ trivially).

The Floer theory of Lagrangian correspondences was first studied by Wehrheim-Woodward \cite{wwquilted} in the exact/monotone setting.  More recently, Fukaya \cite{Fukaya-corr} developed a general theory and constructed an $A_\infty$ tri-module to encode the relations between the deformation-obstruction theory of $L, L^\pi$ and $\bar{L}$.  We would like to follow their constructions to understand Teleman's conjecture.

On the other hand, equivariant Floer theory is essential to understand how the fibration $F$ on $\check{Y}$ comes up.  Equivariant Lagrangian Floer theory is one of the essential ingredients in Daemi-Fukaya's approach of proving Atiyah-Floer conjecture \cite{Daemi-Fukaya}.  In \cite{KLZ}, Kim, the first-named author and Zheng developed equivariant extensions of the SYZ program and Lagrangian Floer theory.  

A key feature is that the \emph{equivariant Borel space $L_G = L\times_{G} EG$ of a Lagrangian $L$ can bound non-trivial stable discs, and hence captures equivariant quantum corrections}. Assuming $L$ has minimal Maslov index 0, the disc potential of $L_G$ takes the form 
$$\textstyle W(z) + \sum_i \lambda_i \log F_i(z) $$
where $W$ and $F_i$ are functions of the formal deformation space of $L$, and $\lambda_i$ are the equivariant parameters which form a basis of $H^2(BG)$ for the classifying space $BG$.  Thus, fibration $F$ arises from the first principle using equivariant Lagrangian Floer theory.

The goal of this paper is to develop the theory of equivariant Lagrangian correspondence and apply it to construct mirrors of symplectic quotients.  We find that it is rather common that the Lagrangian correspondence $L^\pi$ is obstructed in Floer theory, even in simple toric situations.  In general, one needs to use bulk deformations \cite[Theorem 3.8.41 and Corollary 3.8.43]{FOOO10} of $Y^-\times X$ in order to kill the obstructions.

Suppose $L^\pi$ is weakly unobstructed, possibly after bulk deformations.  A further ingredient is the equivariant disc potential of $L_G^\pi$.  Namely, the equivariant theory will give rise to non-trivial equivariant obstruction of $L_G^\pi$.  Such equivariant terms of $L_G^\pi$ will combine with the equivariant part of $L_G$, and produce further quantum corrections in the fibration $F$.  In general, the fibration $F$ involves a highly non-trivial mirror map, which is a central object that accounts for the powerful predictions of mirror symmetry in enumerative geometry.  A main idea of this paper is that the equivariant disc potential of the Lagrangian correspondence between $Y$ and $X$ contains information about the mirror map.

Here is the main theorem that we obtain for the Borel construction of the Lagrangian correspondence $L^\pi$. Let $L^\pi_G$ be the Borel space, which is a Lagrangian in $((Y^- \times X) \times T^*EG)\sslash_0 G$. Readers are referred to Section \ref{sec:equivCor} for our detailed setup of equivariant correspondence. 

\begin{theorem} \label{mainthm}
Under Setup \ref{hamlagsmred}, and hence the $G$-action on $L$ is free, so that $L_G$ is homotopic to $\bar{L}$. Assume further that $L,L^\pi,\bar{L}$ are weakly unobstructed.
 \begin{enumerate}
    \item (Proposition \ref{equivobstr}, simplified form) The $A_\infty$ tri-module \\
    $CF_{eq}(\bar{L}; L,L^\pi)$ has an equivariant obstruction (after boundary deformations) of the form
		\begin{equation}
		(n_{0,0,0})^2 =  (W_{L} + W_{L^\pi} - W_{\bar{L}}) \Id \pm (h_{L} + h_{L^\pi}) \cdot \lambda
		\label{eq:n^2}
		\end{equation}
		where $W_{L} + h_{L} \cdot \lambda$ and  $W_{L^\pi} + h_{L^\pi} \cdot \lambda$ are the equivariant disc potentials of $L$ and $L^\pi$ respectively, $\lambda = (\lambda_1,\ldots,\lambda_k)$ are the degree-two equivariant parameters of $G$ (and $k$ is the rank), and $W_{\bar{L}}$ is the disc potential of $\bar{L}$.  
  %The expression also holds after boundary deformations.
		\item(Corollary \ref{compequiud}) After fixing the canonical models for $L_G$ and $L^\pi_G$, there exists a map between the equivariant weak Maurer-Cartan spaces $$\circ: MC_{weak}(L_G) \times MC_{weak}(L^\pi_G) \to MC_{weak}(\bar{L})$$
        $$(b_{L_G}, b_{L^\pi_G}) \mapsto b_{\bar{L}} \coloneqq b_{L^\pi_G} \circ b_{L_G},$$
		 such that their equivariant disc potentials satisfy
		\begin{equation}
		W_{L_G}(b_{L_G}) + W_{L^\pi_G}(b_{L^\pi_G}) = W_{\bar{L}}(b_{\bar{L}}).
		\label{eq:W}
		\end{equation}
		\item (Corollary \ref{udisom})
        For any chosen $b_{L^\pi_G}$ and $(b_{L_G}, b_{\bar{L}})$ under \ref{updownmcspaceisom}, we have an algebra isomorphism between the deformed Floer cohomology rings 
		\begin{equation} \label{hfisom}
        HF(L_G,b_{L_G}) \cong HF(\bar{L},b_{\bar{L}}).
        \end{equation}
	\end{enumerate}
\end{theorem}

\begin{remark}
    Theorem \ref{mainthm} (3) (and its generalisations to pairs of Lagrangians in the sense of Remark \ref{severallag}) would imply %\footnote{Here we consider $G$-equivariant Floer theory via Borel spaces, instead of using $G$-equivariant Kuranishi spaces as originally proposed by Fukaya.} 
    an equivalence between (equivariant) derived Fukaya categories of $Y$ and $X$  \footnote{See also \cite{LS} for more precise statements on (equivariant) wrapped Fukaya categories.}:
    \begin{equation}\label{fukcatequiv}
        \mathrm{DFuk}_G(Y)_{\mu=0} \cong \mathrm{DFuk}(X),
    \end{equation}
with object level bijection given by
    $$(L,b_{L_G}) \leftrightarrow (\bar{L},b_{\bar{L}}),$$
    
 where $\mathrm{DFuk}_G(Y)_{\mu=0}$ is (formally) the derived Fukaya category of $Y_G$, as an (ordinary) category with objects $(L,b_{L_G})$, where $L\subseteq \mu^{-1}(0)$ is a $G$-invariant Lagrangian and $b_{L_G}\in MC_{weak}(L_G)$, 
 and morphism spaces being the deformed Floer cohomology:
 $$\Hom_{\mathrm{DFuk}_G(Y)}((L^{(1)},b_{L^{(1)}_G}), (L^{(2)},b_{L^{(2)}_G})) = HF_G(L^{(1)}, L^{(2)}; b_{L^{(1)}_G}, b_{L^{(2)}_G}).$$ 

%and proves \ref{fukcatequiv} at the endomorphism level:
%    $$\End_{\Fuk_G(Y)}(L,b_{L_G}) \cong \End_{\Fuk(X)}(\bar{L},b_{\bar{L}}).$$
\end{remark}

Further comments on Theorem \ref{mainthm}: for (2), we need to use the assumption that the $G$-action on $L$ is free, so that $H_G(L) \cong H(\bar{L})$ in classical cohomology.  In particular, $n_{0,0,1}(\one,-)$ gives an isomorphism between $H(L_G)$ and $H(\bar{L})$ which are taken as canonical models for the (quilted) Floer theory of $(\bar{L}, L_G, L^\pi_G)$ and $\bar{L}$ respectively.  Using this isomorphism and the inductive technique over the Novikov ring found by Fukaya \cite{Fukaya-corr}, the map $\circ: MC_{weak}(L_G) \times MC_{weak}(L^\pi_G) \to MC_{weak}(\bar{L})$ can be constructed by solving the equation $n_{0,0,0}(\one)=0$ under boundary deformations.

Under Equation \eqref{eq:W}, the deformed complex $(CF_{eq}(\bar{L}; L, L^\pi),n^{def}_{0,0,0})$ is unobstructed. Then both $n^{def}_{0,0,1}(\one;-)$ and $n^{def}_{1,0,0}(-;\one)$ are chain isomorphisms.  This gives (3) on the cohomology level, which turns out to be a ring isomorphism with respect to the deformed product structure.

In general, the obstruction of $L^\pi$ and the equivariant potential $W_{L^\pi} + \lambda 
\cdot h_{L^\pi_G}$ are highly non-trivial.  In Section \ref{sec:toric}, we find some toric geometries in which the obstruction vanishes and the equivariant potential can be computed.  In particular, when $Y=\C^n$ and $X=\C^n \sslash_c T^k$ is a semi-Fano toric manifold for some level $c$, we find that $h_{L^\pi_G}$ is essentially the mirror map.  That is,

\begin{theorem}[Theorem \ref{thm:equiv-semi-Fano}]
    Let $X$ be a compact semi-Fano toric manifold, and $Y=\C^n$ whose coordinate axes are in a one-to-one correspondence with the rays of the fan of $X$.  Then the Lagrangian correspondence $L^\pi$ is unobstructed.  Moreover, the equivariant disc potential of $L^\pi$ equals
    $$W^{Morse}_{L^\pi, T} = \sum_{j=1}^{n-d} \lambda_j (\log q_j - \log \check{q}_j(q))$$
    where $\check{q}_j(q)$ denotes the inverse mirror map for $X$.
\end{theorem}

The mirror map plays a central role in closed-string mirror symmetry for the enumerative geometry of holomorphic curves.  They are given by hypergeometric functions that are solutions to a certain Picard-Fuchs system of differential equations.  See Equation \eqref{eqn:funcn_g} for the expression in the toric case. 
Compared to our previous method of wall crossing and compactification \cite{CLL,CLLT,CCLT13}, 
the equivariant theory gives another approach to understanding mirror maps for toric Calabi-Yau manifolds via the Lagrangian correspondence for symplectic quotients.  It extends our understanding towards quantum corrections in SYZ mirror symmetry. 

In the above theorem, we take $Y=\C^n$ to ensure unobstructedness of $L^\pi$.  In general, if we take $Y$ to be a compact toric Fano manifold such as $\bP^1 \times \bP^1$, non-trivial (non-equivariant or equivariant) obstruction of $L^\pi$ can occur.  See Example \ref{ex:P1P1P1} and Example \ref{ex:P1P1} in Section \ref{sec:toric}.  

\begin{example}
\label{ex:P1P1P1cont}
    We continue to discuss Example \ref{ex:P1P1P1}.  Using the Maslov index formula by \cite{Cho-Kim} as explained in Proposition \ref{prop:Maslov}, we find that the Maslov indices of the depicted discs in Figure \ref{fig:P1P1P1} have Maslov index $(-2)$.  Thus, even in this simple situation, one needs to use a bulk deformation (of degree four) to kill these negative discs.  The bulk deformation will produce extra terms in the disc potential, which explains the discrepancy in the comparison of $W_L$ and $W_{\bar{L}}$.  See Example \ref{ex:P1P1P1-2}.
\end{example}

The relations between the (equivariant, wrapped) Fukaya categories of $Y$ and $X$ were conjectured in \cite{LS} for singular cases. Throughout the article, we have assumed that $G$ acts freely on $\mu^{-1}(0)$, therefore $0$ is a regular value of $\mu$. In some examples, we can check by hand that our statements on the relation between equivariant mirrors and mirrors of quotients still hold at singular moment levels.  We will illustrate an example in Subsection \ref{singlevel}.

\subsection*{Relation to other works} Since the pioneering work of Khovanov and Seidel \cite{KS-quiver} (for the exact case and $G=\Z_2$), there has been many developments of Lagrangian Floer Theory in presence of symmetry for both finite case (e.g. \cite{aurouxsmith,  baohonda, CH, CHL, chowleung,  HLS, HLS-corr, SS}) and continuous case (e.g. \cite{HLS-simplicial, zernik2015equivariant, zernik2017fixed, Woodward-Xu, Daemi-Fukaya, KLZ, HKLZ, futakisanda, Cazassus, xiao2023equivariant, LS}) with a wide range of applications, a noteworthy one being a formulation of the ``symplectic side" of the Atiyah-Floer conjecture \cite{atiyahconj} (e.g. \cite{Manolescu-Woodward, Daemi-Fukaya, Cazassus} \footnote{We refer the reader to \cite{Cazassus} for an overview on the role of equivariant Floer theory to Atiyah-Floer conjecture.}). See also \cite{ Se-lect, lekilipascaleff, EL}. We briefly describe some of them below, which developed a relation between a version of the equivariant Floer theory of $Y$ and the Floer theory of the quotient $X$.  A distinguished feature of our formulation is that it produces the fibration structure conjectured by Teleman, which enables a more direct comparison between the theories of $Y$, $Y_G$, and $X$.

In \cite{Daemi-Fukaya}, Daemi and Fukaya announced a construction of an $A_\infty$ homotopy equivalence from (a component of) the $G$-equivariant Fukaya category of $Y$ to the (bulk-deformed) Fukaya category of $X$ using a functor induced from $L^\pi$. They used an equivariant de Rham model that required $G$-equivariant Kuranishi structure on the disk moduli of $L$, and assumed minimal Maslov index greater than two.  
    
In our work, we make use of the disk moduli of (the approximation spaces of) the Borel spaces $L_G \subset Y_G$ as in \cite{KLZ}.  Moreover, since we do not restrict the minimal Maslov index to be greater than two, we need to take care of obstructions for the Lagrangian correspondence, which can also have an equivariant disc potential.
    
In \cite{Woodward-Xu}, Woodward and Xu constructed an open quantum Kirwan map from the gauged Floer theory of $Y$ to the Floer theory of $X$ by counting affine vortices.  The quasimap Floer theory for $Y$ in \cite{Woodward} is the key ingredient in their formulation.

In this paper, a main observation is that the equivariant Lagrangian correspondence encodes the discrepancies caused by discs emanated from unstable locus in $Y$ for the $G_\C$ action.  Moreover, the usual Floer theory of $(Y,L)$ is obtained as the non-equivariant limit of our formulation of equivariant Lagrangian Floer theory.
    
In \cite{Cazassus}, Cazassus constructed the equivariant Floer complex $CF_G(L, L')$ and Kirwan morphisms between $CF_G(L, L')$ and $CF(\bar{L}, \bar{L}')$ for a pair of $G$-Lagrangians $(L, L')$ in a different way using quilted Floer theory together with a telescope construction.  %On the other hand, we define $CF_G(L)$ as (a canonical model of) an inverse limit of de Rham models as $A_\infty$ algebras (Definition \ref{equivfloercpx}).
%\item (continuous symmetry) Xiao Yao, Futaki-Sanda, Zernik
%\item (finite group equivariant Floer theory) Seidel-Smith, Bao-Honda,  
%\item (finite group action on Floer theory) Cho-Hong, Wu, Chow-Leung
%\item  Viterbo, Hendricks-Lipshitz-Sarkar,  Lekili-Pascaleff, Seidel (Categorical dynamics

Equivariant theory in the closed-string sector has a much longer history which is one of the crucial ingredients for the discovery of mirror symmetry.  See for instance the works of Givental \cite{givental_imrn96} in terms of Gromov-Witten invariants and Viterbo \cite{viterbo} in terms of Floer cohomology. Recently, González, Mak and Pomerleano \cite{GMP1,GMP2} studied equivariant quantum cohomology and symplectic cohomology in relation with equivariant Seidel representations and 3D mirror symmetry. A common feature (compared to our construction here and also to \cite{KLZ,Cazassus}) is the use of Borel spaces in constructing equivariant Floer theory.

\subsection*{Updates} After the arXiv preprint of this work, important related works appear and let us make some updates here.

In \cite{PT1}, Pomerleano and Teleman studied the $LG$-equivariant quantum cohomology $QH_{LG}(Y)$ and the relations to quantum cohomology of the symplectic quotient $QH(X)$. Monotonicity condition plays an important role for the theory.  In this paper, we stress on the equivariance obstruction theory for Lagrangian correspondence in general situations.

In \cite{xiao2}, assuming equivariant transversality, Xiao used the bulk deformation technique in \cite{FOOO} to solve the obstruction of the the moment level correspondence $L^\pi \subseteq Y^- \times X$ by deforming the ambient space $Y^- \times X$.
Surjectivity of the restriction map between basic form complexes
\begin{align}\label{bascpxsurj}
\Omega^\bullet_{bas}(Y \times X) \rightarrow \Omega^\bullet_{bas}(L^\pi)
\end{align}
is one of the key ingredients.

In our formulation, we consider weakly unobstructed deformations, which are slightly more flexible and is particularly useful in Fano and general-type geometries, in which the disc potential plays an important role. Furthermore, we take the Borel construction of $Y \times X$ and do not assume equivariant transversality. %We plan to use Kirwan surjection between equivariant cohomologies \cite{Kirwansurj}
By considering canonical models and observing that $H^\bullet_{G}(L^\pi) \cong H^\bullet(X)$, we have the following surjection
\begin{align}
H^\bullet_{G}(Y \times X) \rightarrow H^\bullet_{G}(L^\pi)
\end{align}
which allows us to solve equivariant obstructedness by equivariant bulk deformation of $L^\pi$ in a forthcoming work.

In \cite{sampietro}, under the relatively exact setting, Sampietro constructed equivariant Lagrangian Floer homologies $HF^G_*(L_0, L_1; \Lambda_0)$ for a pair of cleanly-intersecting $G$-invariant Lagrangians $L_0, L_1 \subseteq \mu^{-1}(0)$, by lifting Cazassus's construction \cite{Cazassus} to the Novikov ring $\Lambda_0$. Moreover, he proved an isomorphism between (equivariant) Lagrangian Floer homologies of $(L_0, L_1)$ and their quotients $(\overline{L_0}, \overline{L_1})$:
$$HF^G_*(L_0, L_1; \Lambda_0) \cong HF_*(\overline{L_0}, \overline{L_1}; \Lambda_0).$$

Compared to \ref{hfisom} in Theorem \ref{mainthm} (3), while we restricted ourselves to the case of a single Lagrangian $L$, it works for more general symplectic manifolds as in \cite{FOOO}. In such settings, due to the existence of obstructions to defining (equivariant) Lagrangian Floer cohomology, deformations by (equivariant) weak Maurer-Cartan elements are needed. To obtain \ref{hfisom} in full generality, an identification of their (equivariant) weak Maurer-Cartan elements is necessary, for which we did so carefully via equivariant correspondence tri-modules.

The structure of this paper is as follows.  We review the Floer theory of Lagrangian correspondences (developed by Fukaya \cite{Fukaya-corr}) in Section \ref{sec:corr} and equivariant Lagrangian Floer theory in Section \ref{sec:equiv}.  In Section \ref{sec:equivCor}, we develop a Floer theory for equivariant Lagrangian correspondences and tackle Teleman's conjecture.  In Section \ref{sec:toric}, we solve the obstructions in the toric setup and find a relation with the mirror map for toric semi-Fano manifolds.

\section*{Acknowledgements}
The first-named author expresses his gratitude to Yoosik Kim and Xiao Zheng for enlightening discussions on various related topics.  The third-named author thanks Denis Auroux, Kwokwai Chan, Cheol-Hyun Cho, Dongwook Choa, Hiroshi Iritani, Yu-Shen Lin, Ziming Ma, Yong-Geun Oh, Hiroshi Ohta, Kaoru Ono,  Paul Seidel and Weiwei Wu for valuable discussions on various stages of this project, and Ki-Fung Chan for a careful reading on the draft. He also thanks the National Center for Theoretical Sciences for hospitality in which part of this work was done and presented. Special thanks to E. Shelukhin for pointing out a missing assumption in Proposition \ref{prop:semi-Fano} in earlier versions. We also thank the anonymous referee for pointing out some mistakes in Proposition \ref{prop:unobs} in earlier versions. 

N. C. Leung was supported by grants of the Hong Kong Research Grants Council (Project No. CUHK14301721 \& CUHK14306720 \& CUHK14306322) and direct grants from CUHK. This work was supported by the Institute for Basic Science (IBS-R003-D1).

\section{Weakly-unobstructed Lagrangian correspondences} \label{sec:corr}
In this section, we recall and develop the Lagrangian Floer theory of (weakly-unobstructed) Lagrangian correspondences, following Fukaya \cite{Fukaya-corr}. 
%Readers
%who are more interested in their equivariant extensions and applications may directly jump to the subsequent sections in the first reading.

The first three subsections focus on the algebraic aspects: in subsection \ref{ainftyalg}, we review the notion of $A_\infty$ algebras and tri-modules; in subsection \ref{cyclicprop}, we recall the concept of cyclic property for $A_\infty$ tri-modules and introduce a stronger notion of bi-cyclic property. Also, we strengthen Fukaya's composition result to weak bounding cochains in Proposition \ref{algcompweak}; in Subsection \ref{hpt}, we review the homological perturbation theory of filtered $A_\infty$ algebras and develop that of filtered $A_\infty$ tri-modules.

The latter three subsections are more geometric in nature: In Subsection \ref{lft}, we recall the de Rham model of Lagrangian Floer theory; in Subsection \ref{subsec-lagcorrcomp}, we review the concept of Lagrangian correspondences and their geometric compositions; finally, their Floer theory via correspondence tri-modules will be recalled in Subsection \ref{subsec-corrtrimod}. We also extend Fukaya's unobstructedness result to weakly-unobstructed cases in Corollary \ref{weakunobcomp}.

\subsection{$A_\infty$ algebras and tri-modules} \label{ainftyalg}

In this subsection, we first recall the notion of $A_\infty$ algebras and $A_\infty$ tri-modules over them, building on Fukaya's foundational work \cite{Fukaya-corr}. See also \cite{mau}.

\subsubsection{Novikov coefficients}
We fix the notation for the Novikov coefficients.\\
Given a (commutative, unital, ungraded) ground ring $R$, the (universal) Novikov ring over $R$ is a $T$-adic completion of $R[T]$ defined by
$$\Lambda_0 = \Lambda_0(R) = \left\{\sum^\infty_{i=0} a_i T^{\lambda_{i}}| a_i\in R; 0 = \lambda_{0}<\cdots < \lambda_i <\cdots; \lim_i \lambda_{i} = \infty\right\}$$
as a valuation ring with (unique) maximal ideal $\Lambda_+$ and fraction field $\Lambda$.\\
For each discrete submonoid $$\mathbb{G} = \{0 = \beta_{0}<\beta_{1}<\cdots < \beta_i <\cdots\}\subseteq (\mathbb{R}_{\geq 0}, +, 0)$$
the subring of $\mathbb{G}$-gapped elements 
$\Lambda^{\mathbb{G}}_0 \subseteq \Lambda_0$ is defined by 
$$\Lambda^{\mathbb{G}}_0 = \left\{\sum^\infty_{i=0} a_i T^{\beta_{i}} \in \Lambda_0\right\}$$
as a valuation subring with the maximal ideal $\Lambda^{\mathbb{G}}_+$ and fraction field $\Lambda^{\mathbb{G}}$.\\

For any graded $R$-module $\overline{C}$, the completed tensor product $C \coloneqq \overline{C}\hat{\otimes}\Lambda_0$ is a graded complete $\Lambda_0$-module with $\deg T = 0$. Similarly, define $C_+ \coloneqq \overline{C}\hat{\otimes}\Lambda_+$; given any discrete submonoid $\mathbb{G}\subseteq (\mathbb{R}_{\geq 0}, +, 0)$, denote the submodule of $\mathbb{G}$-gapped elements as $C^\mathbb{G}\coloneqq \overline{C}\hat{\otimes}\Lambda^\mathbb{G}_0$; similarly $C^\mathbb{G}_+\coloneqq \overline{C}\hat{\otimes}\Lambda^\mathbb{G}_+$.

\begin{remark}
For later purposes, we will also consider $R$ being a $2\mathbb{Z}_{\geq 0}$-graded commutative algebra, i.e. a $\mathbb{Z}$-graded commutative algebra (over some ring $S$) concentrated in nonnegative even degrees $R = \displaystyle{\bigoplus_{2m\in\mathbb{Z}_{\geq 0} }}R^{2m}$. A typical example is the rational cohomology ring $H^*(BG; \mathbb{Q})$ of the classifying space $BG$ for a compact connected Lie group $G$. In such situation, the Novikov ring $\Lambda_0(R)$ will also be $2\mathbb{Z}_{\geq 0}$-graded with $\deg{T}=0$. Hence the grading in the completed tensor product $C \coloneqq \overline{C}\hat{\otimes}\Lambda_0$ will be the total grading of $\overline{C}$ and $\Lambda_0(R)$.
\end{remark}

\subsubsection{$A_\infty$ Algebras}
We refer the readers to \cite[Chapter 3]{Ohfukcat} for the precise definitions of a (filtered, gapped) $A_\infty$ algebra $C = (C^\bullet = \overline{C}^\bullet\hat{\otimes}\Lambda_0, \{m_k\}_{k\geq 0})$, as well as its (strict) unitality, weak Maurer-Cartan set (resp. space) \\ $\widehat{MC}_{weak}(C; \Lambda^{\mathbb{G}}_+)$ (resp. $MC_{weak}(C; \Lambda^{\mathbb{G}}_+)$) and potential function $W$.

Moreover, we define the restriction of scalars of $A_\infty$ algebras as follows:

\begin{defn}
    Given a filtered  $A_\infty$ algebra $C = (C^\bullet, \{m_k\})$ over $\Lambda_0(R)$, for any ring morphism $S\xrightarrow{\varphi}R$, the restriction of scalars of $C$ (along $\varphi$), denoted as $C_S$, is a filtered  $A_\infty$ algebra $ (C^\bullet, \{m^{S}_k\}_{k\geq 0})$ over $\Lambda_0(S)$, where 
    \begin{itemize}
        \item $C^\bullet = \overline{C}^\bullet\hat{\otimes}\Lambda_0(R)$ is the $\Lambda_0(S)$-module obtained from the restriction of scalars of $C$ along the ring morphism $\Lambda_0(S)\xrightarrow{\varphi}\Lambda_0(R)$.
        \item $m^S_k: (C_S[1])^{\hat{\otimes}k} \rightarrow C_S[1]$ is defined by $m_k$ as a $\Lambda_0(S)$-multi-linear map.
    \end{itemize}
\end{defn}
It follows immediately that if $C$ is gapped (resp. unital), so as $C_S$.

Besides, the $A_\infty$ algebras $C$ and $C_S$ are related by identity map as follows:
\begin{corollary}
    The identity morphism $ Id: C_S \rightarrow C$ is a strict $A_\infty$ morphism over $\varphi$. It is gapped (resp. unital) if $C$ is.
\end{corollary}

This implies the following corollary on their weak Maurer-Cartan sets: 

\begin{corollary}
\label{basechangemc}
 $Id$ induces $Id: C^{odd}_{S, +} \rightarrow C^{odd}_{+}$, which restricts to a map 
 \begin{equation}
 \label{idmcmap}
 Id: \widehat{MC}_{weak}(C_S; \Lambda_+(S))\rightarrow \widehat{MC}_{weak}(C; \Lambda_+(R))
  \end{equation}
 between weak Maurer-Cartan sets such that $\varphi(W_S(b)) = W(b)$ for any $b\in\widehat{MC}_{weak}(C_S; \Lambda_+(S))$.
\end{corollary} 

Moreover, consider the following fiber product
$$\widehat{MC}_{weak}(C; \Lambda_+(R)) \times_{\Lambda_0(R)} \Lambda_0(S) = \{(b, a) | W(b) = \varphi(a)\}.$$

Then it follows from Corollary \ref{basechangemc} that (\ref{idmcmap}) factors through a map
$$f: \widehat{MC}_{weak}(C_S; \Lambda_+(S))\rightarrow \widehat{MC}_{weak}(C; \Lambda_+(R)) \times_{\Lambda_0(R)} \Lambda_0(S)$$ defined as $f(b) = (b, W_S(b))$.

\begin{prop}\label{restrscamcisom}
    $f$ is a bijection with the inverse    
    $$g:\widehat{MC}_{weak}(C; \Lambda_+(R)) \times_{\Lambda_0(R)} \Lambda_0(S) \rightarrow \widehat{MC}_{weak}(C_S; \Lambda_+(S))$$
   defined as $g(b, a) = b$. Moreover, $g$ intertwines the natural projection to $\Lambda_0(S)$ and $W_S$, i.e. $W_S(g(b, a)) = a$.
\end{prop}

\begin{proof}
    Note that for any $(b,a) \in \widehat{MC}_{weak}(C; \Lambda_+(R)) \times_{\Lambda_0(R)} \Lambda_0(S)$,
    $$\sum_k m_k(b^{\otimes k}) = \varphi(a) \cdot_R e $$
   which implies  
    $$\sum_k m^S_k(b^{\otimes k}) = a \cdot_S e $$
    hence $b \in \widehat{MC}_{weak}(C; \Lambda_+(S))$ with $W_S(b) = a$. This implies $g$ maps into \\
    $\widehat{MC}_{weak}(C_S; \Lambda_+(S))$ and  satisfies both $W_S(g(b, a)) = a$ and $f\circ g = Id$. 
    
    The remaining identity $g\circ f = Id$ follows directly from definition.
\end{proof}

\subsubsection{$A_\infty$ tri-modules}
We refer the readers to \cite[Section 5.1]{Fukaya-corr} for the precise definitions of (filtered, gapped, unital) $A_{\infty}$ tri-modules $D = (D^\bullet = \overline{D}^\bullet\hat{\otimes}\Lambda_0, \{n_{k'',k',k}\})$ and $A_{\infty}$ tri-module morphisms $f: D_1 \rightarrow D_2$. Besides, we recall the notion of pullbacks of $A_\infty$ tri-modules and morphisms below:

\begin{defn}
\label{pullbacktrimod}
Given three (unital) filtered $A_{\infty}$ algebra morphisms $g: C_1 \rightarrow C_2$, $g': C'_1 \rightarrow C'_2$, $g'': C''_1 \rightarrow C''_2$ and (unital) filtered left $C''_i$, right $(C'_i, C_i)$-$A_{\infty}$ tri-modules $D_i$ for $i = 1, 2$, the pullback $A_\infty$ tri-module of $D_2$ by $(g'', g', g)$, denoted as 
    $(g'', g', g)^*D_2$, is a (unital) filtered left $C''_1$, right $(C'_1, C_1)$-$A_{\infty}$ tri-module defined as tri-module version of the pullback $A_{\infty}$ bi-module defined in \cite[Definition 3.5.2]{Ohfukcat}.
    %\cite[Definition 5.2.8]{FOOO}
    %If in addition $C_i$ (resp. $C'_i, C''_i$) has a strict unit $e_i$ (resp. $e'_i, e''_i$) for $i = 1, 2$ such that $g$ (resp. $g', g''$) is unital, and $D_2$ is unital with respect to $(e''_2, e'_2, e_2)$, then $(g'', g', g)^*D_2$ is unital with respect to $(e''_1, e'_1, e_1)$.
    
 A (unital) filtered $A_\infty$ tri-module morphism $f: D_1 \rightarrow D_2$ over $(g'', g', g)$ is a (unital) filtered left $C''_1$, right $(C'_1, C_1)$-$A_{\infty}$ tri-module morphism $f: D_1 \rightarrow (g'', g', g)^*D_2$.
\end{defn}
The details are the same as those in the bimodule case and hence omitted.

\subsection{Cyclic Property} \label{cyclicprop}
In this subsection, we discuss the notion of cyclic property for $A_{\infty}$ tri-modules due to Fukaya in \cite{Fukaya-corr}, and introduce the stronger notion of bi-cyclic property. Both of them are crucial in relating the deformation-obstruction theory of $A_\infty$ algebras via their $A_\infty$ modules.
\subsubsection{Cyclic Property}
We recall the notion of cyclic elements in $A_\infty$ tri-modules, introduced by Fukaya in \cite[Definition 6.5]{Fukaya-corr}, as follows:

\begin{defn}
\label{leftcyclic}
    Given unital, $\mathbb{G}$-gapped filtered $A_{\infty}$ algebras $C'', C', C$ and a unital, $\mathbb{G}$-gapped filtered left $C''$, right $(C', C)$ - $A_{\infty}$ tri-module $D$, a $\mathbb{G}$-gapped element $\mathbf{1}\in D^{0,\mathbb{G}}$ is called left $C''$-cyclic (or simply left cyclic) if
\begin{enumerate}
    \item $\overline{n}_{0,0,0}(\overline{\mathbf{1}}) = 0$.
    \item $n_{1,0,0}(-; \mathbf{1}): C'' \rightarrow D$ is an isomorphism of $\Z /2$-graded gapped $\Lambda_0$-modules.
\end{enumerate}
Similarly, we call $\mathbf{1}$ right $C'$-cyclic (resp. right $C$-cyclic) if (1) and (2) are satisfied with $n_{1,0,0}(-; \mathbf{1})$ replaced by $n_{0,1,0}(\mathbf{1}; -): C' \rightarrow D$ (resp. $n_{0,0,1}(\mathbf{1}; -): C \rightarrow D$).

\end{defn}
\begin{remark}
    It follows from $\mathbb{G}$-gappedness that (2) above is equivalent to the following condition:
    
    \begin{enumerate}
    \item [(2)'] $\overline{n}_{1,0,0}(-; \overline{\mathbf{1}}): \overline{C}'' \rightarrow \overline{D}$ is an isomorphism of $\Z$-graded $R$-modules.
\end{enumerate}
\end{remark}

We also recall the following important result of Fukaya \cite{Fukaya-corr} on compositions of (strict) bounding cochains of $C$ and $C'$ using cyclic property:

\begin{prop} \label{algcompdegoneprop}
    Fix a left cyclic element $\mathbf{1}\in D^{0, \mathbb{G}}$, then there exists a map, called the composition map 
\begin{equation} \label{algcompdegone}
    C^{odd}_{+, \mathbb{G}} \times C'^{odd}_{+, \mathbb{G}} \xrightarrow{\circ} C''^{odd}_{+, \mathbb{G}},   \end{equation}  
    $$(b,b') \mapsto b'' \coloneqq b \circ b',$$

     where $b''$ is uniquely determined by the equation $n^{b'',b',b}_{0,0,0}(\mathbf{1})=0$.
    
    Moreover, it restricts to a map between their (strict) Maurer-Cartan sets
    \begin{equation} \label{algcompstricteqn}
    \widehat{MC}(C) \times \widehat{MC}(C') \xrightarrow{\circ} \widehat{MC}(C'')
    \end{equation} 

which respects their gauge equivalence relations. Therefore, it descends to a map between their (strict) Maurer-Cartan spaces 
    \begin{equation} 
    MC(C) \times MC(C') \xrightarrow{\circ} MC(C'').
    \end{equation}
\end{prop}

\begin{proof}
    The proof is identical to that of \cite[Proposition 6.6, 6.16]{Fukaya-corr} (where $D$ is a left $(C, C')$, right $C''$-$A_{\infty}$ tri-module and $\mathbf{1}$ is right cyclic).
\end{proof}

We generalise Proposition \ref{algcompdegoneprop} to weak Maurer-Cartan sets/spaces below:
\begin{prop} \label{algcompweak}
 (\ref{algcompdegone})
    induces a map between weak Maurer-Cartan sets
    \begin{equation} \label{algcompweakeqn} \widehat{MC}_{weak}(C) \times \widehat{MC}_{weak}(C') \xrightarrow{\circ} \widehat{MC}_{weak}(C''),
    \end{equation}
    $$(b,b') \mapsto b'' \coloneqq b \circ b',$$

    in which their potential functions satisfy
    \begin{equation} \label{algpotentialeqn}
    W_C(b) + W_{C'}(b') = W_{C''}(b'').
     \end{equation}
Moreover, (\ref{algcompweakeqn}) descends to a map between their weak Maurer-Cartan spaces 
    \begin{equation} 
    MC_{weak}(C) \times MC_{weak}(C') \xrightarrow{\circ} MC_{weak}(C'').
    \end{equation}
\end{prop}

\begin{remark}
    Analogous statement holds for $\mathbf{1}$ being right $C'$-cyclic.
\end{remark}

\begin{proof}
    Let $b,b'$ be weak bounding cochains as stated and $b''\in C''^{odd}_+$ their composition. Consider the deformed $A_{\infty}$ algebras $(C'', \{m''^{b''}_{k''}\}, e'')$, \\
    $(C', \{m'^{b'}_{k'}\}, e')$, $(C, \{m^{b}_{k}\}, e)$, then $D$ also admits an $A_\infty$ deformation \\
    $(D, \{n^{b'',b',b}_{k'',k',k}\})$ as a unital, $\mathbb{G}$-gapped filtered left $(C'', b'')$, \\
    right $((C', b'), (C, b))$-$A_{\infty}$ tri-module.\\
    
Consider the following $A_\infty$ relation applied to $\mathbf{1}\in D$,
\begin{align} \label{trimodeqn1}
    &n^{b'',b',b}_{0,0,0}(n^{b'',b',b}_{0,0,0}(\mathbf{1})) + n^{b'',b',b}_{1,0,0}(m''^{b''}_0(1); \mathbf{1}) + (-1)^{|\mathbf{1}|}n^{b'',b',b}_{0,1,0}( \mathbf{1};m'^{b'}_0(1)) \\
    &+ (-1)^{|\mathbf{1}|}n^{b'',b',b}_{0,0,1}(\mathbf{1}; m^{b}_0(1))=0. \nonumber
    \end{align}
Note that the first term vanishes by definition of $b''=b\circ b'$; the third term equals $n^{b'',b',b}_{0,1,0}( \mathbf{1}; W_{C'}(b')\cdot e') = (-1)^{||\mathbf{1}||} W_{C'}(b') \cdot \mathbf{1}$ by unitality; similarly, the fourth term equals $n^{b'',b',b}_{0,0,1}(\mathbf{1} ; W_C(b)\cdot e ) = (-1)^{||\mathbf{1}||}W_C(b) \cdot \mathbf{1}$. Therefore, 
\begin{align*}
    &n^{b'',b',b}_{0,1,0}( \mathbf{1};m'^{b'}_0(1))+ n^{b'',b',b}_{0,0,1}(\mathbf{1}; m^{b}_0(1))= (-1)^{||\mathbf{1}||}(W_C(b) +  W_{C'}(b')) \cdot \mathbf{1}\\
    &= -n^{b'',b',b}_{1,0,0}((W_C(b) +  W_{C'}(b')) \cdot e'';\mathbf{1}) 
\end{align*}
by unitality again. Therefore, (\ref{trimodeqn1}) becomes 
$$n^{b'',b',b}_{1,0,0}(  m''^{b''}_0(1)- ((W_C(b) +  W_{C'}(b')) \cdot e''); \mathbf{1}) = 0.$$
Since $n^{b'',b',b}_{1,0,0}(-; \mathbf{1}): C'' \rightarrow D$ is a gapped isomorphism, it implies $$m''^{b''}_0(1) = ((W_C(b) +  W_{C'}(b')) \cdot e'',$$ i.e. $b''\in \widehat{MC}_{weak}(C'')$ with $W_{C''}(b'') = W_C(b) + W_{C'}(b')$.\\

    The last assertion on gauge equivalence follows directly from Proposition \ref{algcompdegoneprop}.
\end{proof}

Recall that each $b\in C^{odd}_{+, \mathbb{G}}$ induces a $b$-deformed gapped $A_\infty$ algebra $C_b = (C, \{m^b_k\})$ which is unobstructed (i.e. $(m^b_1)^2 = 0$) if $b \in \widehat{MC}_{weak}(C)$; similarly, each triple $(b'', b', b)\in C''^{odd}_{+, \mathbb{G}} \times C'^{odd}_{+, \mathbb{G}} \times C^{odd}_{+, \mathbb{G}}$ induces a $(b'', b', b)$-deformed gapped left $C''_{b''}$, right $(C'_{b'}, C_b)$-$A_{\infty}$ tri-module $(D, \{n^{b'', b', b}_{k'',k',k}\})$, which is unobstructed (i.e. $(n^{b'', b', b}_{k'',k',k})^2 = 0$) if they are weak bounding cochains satisfying $W_{C''}(b'') = W_C(b) + W_{C'}(b')$. When $b''\coloneqq b \circ b'$, $m^{b''}_1$ and $n^{b'', b', b}_{0,0,0}$ can be related as follows:

\begin{prop} \label{defisom}
    Given $b\in C^{odd}_{+, \mathbb{G}}, b'\in C'^{odd}_{+, \mathbb{G}}$ with $b''\coloneqq b \circ b'\in C''^{odd}_{+, \mathbb{G}}$, then the map 
    $$\phi''\coloneqq n^{b'', b', b}_{1,0,0}(-; \mathbf{1}): (C'', m^{b''}_1) \rightarrow (D, n^{b'', b', b}_{0,0,0})$$ is a pre-chain isomorphism (up to a sign), i.e. a bijection such that for any $x''\in C''$,
    \begin{equation}
    \label{defbij}
        \phi''(m^{b''}_1(x'')) = -n^{b'', b', b}_{0,0,0}(\phi''(x'')).
        \end{equation}
\end{prop}

\begin{proof}
    That $\phi''$ is bijective follows from the bijectivity of its $R$-reduction $\overline{\phi}'' = \overline{n}_{1,0,0}(-; \overline{\mathbf{1}})$; to show (\ref{defbij}), recall the following $A_\infty$ relation applied to $x''$ and $\mathbf{1}$:
    \begin{align*} &n^{b'', b', b}_{1,0,0}(m^{b''}_1(x''); \mathbf{1}) + n^{b'', b', b}_{0,0,0}(n^{b'', b', b}_{1,0,0}(x''; \mathbf{1})) \\
    &+ (-1)^{||x||} n^{b'',b',b}_{1,0,0}(x''; n^{b'',b',b}_{0,0,0}(\mathbf{1})) =0.
   \end{align*}
The result follows by observing that the last term vanishes by assumption.
\end{proof}

In case of weak bounding cochains, we obtain the following isomorphism of deformed chain complexes and cohomologies:

\begin{corollary} \label{defchisom}
If in addition $b\in \widehat{MC}_{weak}(C)$ and $b'\in \widehat{MC}_{weak}(C')$, then $\phi'' $ is a chain isomorphism of  gapped $\Lambda_0$-modules (up to a sign), inducing the following isomorphism of cohomologies as gapped $\Lambda_0$-modules:
$$[\phi'']: H(C'', m^{b''}_1) \rightarrow H(D, n^{b'', b', b}_{0,0,0}), $$
$$[x'']\mapsto [n^{b'', b', b}_{1,0,0}(x''; \mathbf{1})].$$
\end{corollary}

\subsubsection{Bi-cyclic property} We introduce the stronger notion of bi-cyclic property below:
\begin{defn}
Under the same setting of Definition \ref{leftcyclic}, a $\mathbb{G}$-gapped element $\mathbf{1}\in D^{0,\mathbb{G}}$ is called $(C'', C')$-bi-cyclic (or simply bi-cyclic) if it is both left $C''$ cyclic and right $C'$ cyclic.
\end{defn}

Bi-cyclic property provides an identification of deformation cochains:

\begin{prop}
\label{bicycliciden}
    Given a bi-cyclic element $\mathbf{1}\in D^{0,\mathbb{G}}$, then for any $b\in C^{odd}_{+, \mathbb{G}}$, we have the following mutually inverse isomorphisms
    \begin{equation}
    \label{defspaceisom}
    \begin{tikzcd}
    C'^{odd}_{+, \mathbb{G}}
    \arrow[r,rightarrow, yshift=0.3ex, "b \circ (-) "] \arrow[r,leftarrow, yshift=-0.3ex, "(-)\circ b"']
& C''^{odd}_{+, \mathbb{G}},
\end{tikzcd}
\end{equation}
$$ b' \leftrightarrow b''$$
characterised by the equation $n^{b'',b',b}_{0,0,0}(\mathbf{1})=0$. %the following maps are inverse to each other:
    %$$C'^{odd}_{+, \mathbb{G}} \xrightarrow{b \circ (-)} C''^{odd}_{+, \mathbb{G}}; C''^{odd}_{+, \mathbb{G}} \xrightarrow{(-) \circ b} C'^{odd}_{+, \mathbb{G}}.$$ 
\end{prop}

\begin{proof}
    Given $b'\in C'^{odd}_{+, \mathbb{G}}$, apply the left cyclic property to define $b'' \coloneqq b \circ b'$ via $n^{b'',b',b}_{0,0,0}(\mathbf{1})=0$. Then apply the right cyclic property to define $b'' \circ b \in C'^{odd}_{+, \mathbb{G}}$ via $n^{b'',b'' \circ b, b}_{0,0,0}(\mathbf{1})=0$. Note that $b'$ solves $n^{b'',b',b}_{0,0,0}(\mathbf{1})=0$ by assumption. Therefore, by the uniqueness of solution to $n^{b'',b'' \circ b, b}_{0,0,0}(\mathbf{1})=0$, we have $b'' \circ b = b'$, showing one of the inverse equalities. The proof of the other one is analogous.
\end{proof}

Combining Proposition \ref{defisom} and \ref{bicycliciden} yields the following identification of deformed pre-chain complexes:

\begin{corollary}
\label{bicycisom}
(\ref{defspaceisom}) induce pre-chain isomorphisms (up to a sign)
\begin{equation}
\label{algdefchisom}
    (C', m'^{b'}_1) \xrightarrow[\sim]{\phi'\coloneqq n^{b'', b', b}_{0,1,0}(\mathbf{1}; - )} (D, n^{b'',b',b}_{0,0,0}) \xleftarrow[\sim]{\phi''\coloneqq n^{b'', b', b}_{1,0,0}(-; \mathbf{1})} (C'', m''^{b''}_1)
    \end{equation}
i.e. bijections such that for any $x'\in C'$, $x''\in C''$,
    $$\phi'(m'^{b'}_1(x')) = -n^{b'', b', b}_{0,0,0}(\phi'(x')); \phi''(m^{b''}_1(x'')) = -n^{b'', b', b}_{0,0,0}(\phi''(x'')).$$ 
Hence, $\varphi\coloneqq (\phi'')^{-1} \circ \phi' : C' \rightarrow C''$ is a (genuine) pre-chain isomorphism.
\end{corollary}

Combining further with Proposition \ref{algcompweak} and Corollary \ref{defchisom} yields the following identification of deformed chain complexes and cohomologies:

\begin{corollary}
\label{bicycchisom}
    If in addition $b\in \widehat{MC}_{weak}(C)$, then (\ref{defspaceisom})
restricts to 
    \begin{equation}
    \label{bicycmcisom}
    \begin{tikzcd}
    \widehat{MC}_{weak}(C')
    \arrow[r,rightarrow, yshift=0.3ex, "b \circ (-) "] \arrow[r,leftarrow, yshift=-0.3ex, "(-)\circ b"']
& \widehat{MC}_{weak}(C''),
\end{tikzcd}
\end{equation}
$$ b' \leftrightarrow b''$$
satisfying $W_C(b) + W_{C'}(b') = W_{C''}(b'')$, which further descends to 
\begin{equation}
    \begin{tikzcd}
    MC_{weak}(C')
    \arrow[r,rightarrow, yshift=0.3ex, "b \circ (-) "] \arrow[r,leftarrow, yshift=-0.3ex, "(-)\circ b"']
& MC_{weak}(C''),
\end{tikzcd}
\end{equation}
$$ [b'] \leftrightarrow [b''],$$
which depends only on the gauge equivalence class $[b]\in MC_{weak}(C)$.\\

Moreover, (\ref{algdefchisom}) are chain isomorphisms (up to a sign) and $\varphi$ is a (genuine) chain isomorphism, which induces the following isomorphisms of gapped $\Lambda_0$-modules
$$H(C', m'^{b'}_1) \xrightarrow[\sim]{[\phi']} H(D, n^{b'',b',b}_{0,0,0}) \xleftarrow[\sim]{[\phi'']} H(C'', m''^{b''}_1).$$
\end{corollary}

Recall that $H(C', m'^{b'}_1)$ is an associative algebra (as an $A_\infty$ algebra, i.e. associativity holds up to signs) $(H(C', m'^{b'}_1), [m'^{b'}_2], [e'])$. It turns out that $[\varphi]$ respects the product structure (up to a sign) as follows:
\begin{prop}
\label{ringisom}
    $[\varphi]: (H(C', m'^{b'}_1), [m'^{b'}_2]) \rightarrow (H(C'', m''^{b''}_1), [m''^{b''}_2])$ is a unital algebra isomorphism up to a sign, i.e. for any $[x'_1], [x'_2]\in H(C', m'^{b'}_1)$, 
    \begin{equation}
    \label{preservesm2}
    [m''^{b''}_2(\varphi(x'_1), \varphi(x'_2))] = - [\varphi(m'^{b'}_2(x'_1, x'_2))];
    \end{equation}
    \begin{equation}
    \label{preservesunit}[\varphi(e')] = -[e''].
    \end{equation}
\end{prop}
\begin{proof}
    Given $[x'_1], [x'_2]\in H(C', m'^{b'}_1)$, denote $x''_i \coloneqq \varphi(x'_i) \in C''$ for $i=1, 2$. Applying $[\phi'']$ on both sides of (\ref{preservesm2}), it suffices to show that 
    \begin{equation}
    \label{a}
    [n^{b'', b', b}_{1,0,0}(m''^{b''}_2(x''_1, x''_2); \mathbf{1})] = - [n^{b'', b', b}_{0,1,0}(\mathbf{1}; m'^{b'}_2(x'_1, x'_2) )].
    \end{equation}
    For the LHS, consider the $A_\infty$ tri-module relation applied to $x''_1, x''_2 \in C''$, $\mathbf{1}\in D$, which descends to the following equation in $H(D)$:
    \begin{equation}
    \label{aa}
    [n^{b'', b', b}_{1,0,0}(m''^{b''}_2(x''_1, x''_2); \mathbf{1})] + (-1)^{||x''_1||}[n^{b'', b', b}_{1,0,0}(x''_1; n^{b'', b', b}_{1,0,0}(x''_2; \mathbf{1}))] = 0.
    \end{equation}
    Observe that the second term equals $(-1)^{||x''_1||}[n^{b'', b', b}_{1,0,0}(x''_1; n^{b'', b', b}_{0,1,0}(\mathbf{1}; x'_2))]$ by the $A_\infty$ relation applied to $x''_1 \in C''$, $\mathbf{1}\in D$, $x'_2 \in C'$. 
    Therefore, (\ref{aa}) becomes 
    \begin{equation}
    [n^{b'', b', b}_{1,0,0}(m''^{b''}_2(x''_1, x''_2); \mathbf{1})] + (-1)^{||x''_1||}[n^{b'', b', b}_{1,0,0}(x''_1; n^{b'', b', b}_{0,1,0}(\mathbf{1}; x'_2))] = 0.
    \end{equation}
    Similarly for the RHS, consider instead the $A_\infty$ tri-module relation applied to $x'_1, x'_2 \in C'$, $\mathbf{1}\in D$, which descends to 
    \begin{equation}
    \label{aaa}
    [n^{b'', b', b}_{0,1,0}(\mathbf{1}; m'^{b'}_2(x'_1, x'_2))] + [n^{b'', b', b}_{0,1,0}(n^{b'', b', b}_{0,1,0}(\mathbf{1} ; x'_1); x'_2)] = 0,
    \end{equation}
    where the second term equals $[n^{b'', b', b}_{0,1,0}(n^{b'', b', b}_{1,0,0}(x''_1; \mathbf{1}); x'_2)]$, and hence
    (\ref{aaa}) becomes 
    \begin{equation}
    [n^{b'', b', b}_{0,1,0}(\mathbf{1}; m'^{b'}_2(x'_1, x'_2))] + [n^{b'', b', b}_{0,1,0}(n^{b'', b', b}_{1,0,0}(x''_1; \mathbf{1}); x'_2)] = 0.
    \end{equation}
    Therefore, (\ref{a}) is equivalent to the following equation
    \begin{equation}
  (-1)^{||x''_1||}[n^{b'', b', b}_{1,0,0}(x''_1; n^{b'', b', b}_{0,1,0}(\mathbf{1}; x'_2))] = -[n^{b'', b', b}_{0,1,0}(n^{b'', b', b}_{1,0,0}(x''_1; \mathbf{1}); x'_2)], 
    \end{equation}
which follows from the (induced equation in $H(D)$ of the) $A_\infty$ tri-module relation applied to $x''_1 \in C''$, $x'_2 \in C'$ and $\mathbf{1}\in D$.\\
(\ref{preservesunit}) follows immediately from the unitality relations of $D$.
\end{proof}

\subsection{Homological Perturbation Theory} \label{hpt}
In this subsection, we review the homological perturbation theory of filtered $A_\infty$ algebras, pioneered by Fukaya, Oh, Ohta and Ono in \cite{FOOO}, which transfers $A_\infty$ algebra structure of $C$ to its canonical model via a (strong) contraction. Then we develop so for filtered $A_\infty$ tri-modules. Our treatment below will be closer to that of \cite{yuannonarchi}. %Moreover, we will establish some properties which will be relevant when we apply the theory to the inverse limits of them. 

\subsubsection{Strong Contractions}
In this subsection, we first recall the notion of (strong) contraction, followed by examples from Witten-Morse theory.
\begin{defn}
    Given two graded (co)chain complexes of $R$-modules $\overline{C} = (\overline{C}^\bullet, d), \overline{H} = (\overline{H}^\bullet, \delta)$, a contraction of $(\overline{C}, \overline{H})$ consists of a triple of linear maps $(i, p, h)$, where
    \begin{itemize}
        \item $i: \overline{H}^\bullet \rightarrow \overline{C}^\bullet, p: \overline{C}^\bullet \rightarrow \overline{H}^\bullet$ are degree 0 (co)chain maps.
        \item $h: \overline{C}^\bullet \rightarrow \overline{C}^{\bullet-1}$ is a chain homotopy between $i \circ p$ and $Id_{\overline{C}}$, i.e.
         \begin{equation}
         \label{chhmtyeqn}
            Id_{\overline{C}} -i \circ p = d\circ h + h \circ d.
         \end{equation}
    \end{itemize}
    A strong contraction of $(\overline{C}, \overline{H})$ is a contraction $(i, p, h)$ satisfying the following:
    \begin{equation*}
        p \circ i = Id_{\overline{H}}; h \circ h = 0; h \circ i = 0; p \circ h = 0.
    \end{equation*}
\end{defn}

%We will denote a contraction symbolically as 
%\begin{tikzcd}
%h&\arrow[loop left](\overline{C}^\bullet, d) \arrow[r,rightarrow, yshift=0.3ex, "p"] \arrow[r,leftarrow, yshift=-0.3ex, "i"']
%& (\overline{H}^\bullet, \delta) 
%\end{tikzcd}.\\
We recall the following construction of a strong contraction when $\overline{H} = (H^\bullet(C, d), 0)$:

\begin{prop} \label{strongcontractionconstr}
    Given a graded cochain complex of vector spaces $\overline{C}$ over a field $R$. Then there exists a strong contraction $(i, p, h)$ between $\overline{C}$ and $\overline{H} \coloneqq (H^\bullet(\overline{C}, d), 0)$. 
\end{prop}

\begin{proof}
    We choose a direct sum decomposition of graded vector spaces $\overline{C}^\bullet = \overline{F}^\bullet \oplus \mathrm{Ker} \, d|_{\overline{C}^\bullet}$ (hence $d|_{\overline{F}^\bullet}: \overline{F}^\bullet \xrightarrow{\sim} \mathrm{Im} \, d|_{\overline{C}^\bullet}$). We further choose a direct sum decomposition $\mathrm{Ker} \, d|_{\overline{C}^\bullet} = \mathrm{Im} \, d|_{\overline{C}^{\bullet-1}} \oplus \overline{\mathcal{H}}^\bullet$ , which induces an isomorphism $\overline{\mathcal{H}}^\bullet \xrightarrow{\phi} \overline{H}^\bullet$. Therefore, we have the following ``Hodge decomposition" of $\overline{C}$:
    $$\overline{C}^\bullet = \overline{F}^\bullet \oplus \mathrm{Im} \, d|_{\overline{C}^{\bullet-1}} \oplus \overline{\mathcal{H}}^\bullet.$$

We then define the contraction $(i, p, h)$ as follows:
\begin{itemize}
    \item 
$i: \overline{H}^\bullet \rightarrow \overline{C}^\bullet$ as the composition of the inclusion of $\overline{\mathcal{H}}^\bullet$ and $\phi^{-1}$. 
 \item 
$p: \overline{C}^\bullet \rightarrow \overline{H}^\bullet$ as the composition of  $\phi$ and the projection onto $\overline{\mathcal{H}}^\bullet$.
\item $h: \overline{C}^\bullet \rightarrow \overline{C}^{\bullet-1}$ as $h|_{\mathrm{Im}\, d|_{\overline{C}^{\bullet-1}}}\coloneqq (d|_{\overline{F}^{\bullet-1}})^{-1}$ and zero on other summands.
\end{itemize}

It follows from definition that under the ``Hodge decomposition", every element $x\in \overline{C}^\bullet$ can be decomposed as $x = h(dx) + d(h(x)) + i(p(x))$, which implies (\ref{chhmtyeqn}). Other properties follow directly from definitions. 
\end{proof}

\paragraph{Harmonic Contractions}
A class of geometric examples of strong contractions is called harmonic contractions, whose origin comes from the (Riemannian) Hodge Decomposition of the de Rham complexes of closed oriented Riemannian manifolds $(L, g)$. We briefly recall its construction, and refer the reader to e.g. \cite[Section 7]{yuannonarchi} for further details.

\begin{defn} \label{harmcontr}
    Given a closed oriented Riemannian manifold $(L, g)$, the associated harmonic contraction is a contraction of the de Rham complex $(\Omega^\bullet(L), d)$ and its (de Rham) cohomology $(H^\bullet(L; \mathbb{R}), 0)$. It consists of a triple $(i, p, h)$ which is defined as in the proof of Proposition \ref{strongcontractionconstr}, where 
    \begin{itemize}
        \item $\overline{F}^\bullet = d^*\Omega^{\bullet+1}(L)$ is the subspace of co-exact forms.
        \item $\overline{\mathcal{H}}^\bullet = \mathcal{H}^\bullet_\Delta(L) \coloneqq Ker(\Delta|_{{\Omega^{\bullet}(L)}})$ is the subspace of harmonic forms.
    \end{itemize}
\end{defn}

\begin{corollary}
    $(i, p, h)$ is a strong contraction. Moreover, the constant-1 function $\mathbf{1}$ satisfies $i(p(\mathbf{1})) = \mathbf{1}$.
\end{corollary}

\begin{remark}\label{harmcontrrmk}
Actually, more is known from Hodge theory: from the proof of Proposition \ref{strongcontractionconstr}, the ``Hodge decomposition" of $\Omega^{\bullet}(L)$ takes the form
\begin{align*}
\Omega^{\bullet}(L) &= d^*\Omega^{\bullet+1}(L) \oplus d\Omega^{\bullet-1}(L) \oplus \mathcal{H}^\bullet_\Delta(L)\\
&= (\bigoplus_{\lambda>0}\Omega^{\bullet}(L)_{\lambda})\oplus \mathcal{H}^\bullet_\Delta(L),
\end{align*} 
    which is the (real) Hodge decomposition and the eigenform decomposition.
    \\ Moreover, one may check that $h=d^* \circ Gr_{\Delta}$ satisfies the assumption of $h$ in the proof, where $Gr_{\Delta}: \Omega^{\bullet}(L) \rightarrow \Omega^{\bullet}(L) $ is the Green's operator, i.e. 
    $$Gr_{\Delta}(\alpha) = \Delta^{-1}(\alpha-i(p(\alpha))) = \sum_{\lambda>0}\lambda^{-1}\alpha_\lambda,$$ where $\alpha = \sum_{\lambda\geq0}\alpha_\lambda$.\\
    Furthermore, the homotopy equation $x = h(dx) + d(h(x)) + i(p(x))$ becomes 
   \begin{align*}      
   \alpha &= d^* Gr_{\Delta}(d\alpha) + d(d^* Gr_{\Delta}\alpha) + i(p(\alpha))\\
   &= \Delta(Gr_{\Delta}(\alpha)) + i(p(\alpha)),
   \end{align*}
   i.e. the defining equation of $Gr_{\Delta}$.
\end{remark}

\paragraph{$\lambda_0$-harmonic Contractions}
`The above formulae of harmonic contractions suggests the following generalisation, called $\lambda_0$-harmonic contractions for fixed $\lambda_0 \geq 0$, defined as follows:

\begin{defn}\label{lambdaharmcontr}
Given a closed oriented Riemannian manifold $(L, g)$ and fixed $\lambda_0 \geq 0$, a $\lambda_0$-harmonic contraction is a contraction of the de Rham complex $(\Omega^\bullet(L), d)$ and the direct sum of its eigenform summands supported on $[0, \lambda_0]$, $\displaystyle{\bigoplus_{0\leq \lambda\leq \lambda_0}\Omega^{\bullet}(L)_{\lambda}}$,  as a subcomplex of $(\Omega^\bullet(L), d)$.  It is a triple $(i_{\lambda_0}, p_{\lambda_0}, h_{\lambda_0})$, where 
\begin{itemize}
    \item $i_{\lambda_0}, p_{\lambda_0}$ are the inclusion and projection with respect to the decomposition 
    \begin{align*}
\Omega^{\bullet}(L)
&= (\bigoplus_{\lambda_0< \lambda}\Omega^{\bullet}(L)_{\lambda})\oplus (\bigoplus_{0\leq \lambda\leq \lambda_0}\Omega^{\bullet}(L)_{\lambda}),
\end{align*} 
which are degree 0 chain maps by definition. 
\item $h_{\lambda_0}\coloneqq d^* \circ Gr_{\lambda_0}$, where $Gr_{\lambda_0}: \Omega^{\bullet}(L) \rightarrow \Omega^{\bullet}(L) $ is defined as
    $$Gr_{\lambda_0}(\alpha) = \Delta^{-1}(\alpha-i_{\lambda_0}(p_{\lambda_0}(\alpha))) = \sum_{\lambda>\lambda_0}\lambda^{-1}\alpha_\lambda,$$ where $\alpha = \sum_{\lambda\geq0}\alpha_\lambda$. 
\end{itemize}
The chain homotopy equation follows from the definition of $Gr_{\lambda_0}$ as demonstrated in the Remark \ref{harmcontrrmk}.
\end{defn}

In particular, when $\lambda_0=0$, it reduces to the usual harmonic contraction (after identifying $\mathcal{H}^\bullet_\Delta(L)$ with $H^\bullet(L; \mathbb{R})$).

\begin{corollary}
    $(i_{\lambda_0}, p_{\lambda_0}, h_{\lambda_0})$ is a strong contraction. Moreover, the constant-1 function $\mathbf{1}$ satisfies $i_{\lambda_0}(p_{\lambda_0}(\mathbf{1})) = \mathbf{1}$.
\end{corollary}

\paragraph{($\lambda_0$-)Harmonic Contractions for Witten Laplacians}
Actually, the above constructions of ($\lambda_0$-)harmonic contractions generalises to Witten deformation of $\Omega^\bullet(L)$, introduced by Witten in \cite{wittenmorse}, for which we briefly recall:

\begin{defn}
    Given a smooth manifold $L$ and a smooth function $f:L \rightarrow \mathbb{R}$, for each $t\in\mathbb{R}$, the Witten deformation of $\Omega^\bullet(L)$ (by $tf$) is the cochain complex $\Omega^\bullet_{t}(L) = (\Omega^\bullet(L), d_{t})$, where 
$$d_{t} = e^{-tf}de^{tf} =   d + tdf\wedge : \Omega^\bullet(L) \rightarrow \Omega^{\bullet+1}(L).$$
If $(L,g)$ is a closed oriented Riemannian manifold, the Witten deformation of the codifferential $d^*$ and Laplacian $\Delta$ are defined as 
$$d^*_{t} = e^{tf}d^*e^{-tf},$$
$$\Delta_{t} = d_{t}d^*_{t}+ d^*_{t}d_{t}.$$
\end{defn}

The construction of $\lambda_0$-contraction in Definition \ref{lambdaharmcontr} carries through to the Witten deformation $\Omega^\bullet_{t}(L)$ and the subcomplex $\displaystyle{\bigoplus_{0\leq \lambda\leq \lambda_0}\Omega^{\bullet}_t(L)_{\lambda}}$ (a.k.a. Witten's instanton complex), and is denoted as $(i'_{t, \lambda_0}, p'_{t, \lambda_0}, h'_{t, \lambda_0})$. 
\begin{corollary}
    $(i'_{t, \lambda_0}, p'_{t, \lambda_0}, h'_{t, \lambda_0})$ is a strong contraction. Moreover, the function $e^{-tf}$ satisfies $i'_{t, \lambda_0}(p'_{t, \lambda_0}(e^{-tf})) = e^{-tf}$.
\end{corollary}

Moreover, recall that there are canonical chain isomorphisms  
$$\begin{tikzcd}
(\Omega^\bullet(L), d) \arrow[r,rightarrow, yshift=0.3ex, "e^{-tf}\cdot"] \arrow[r,leftarrow, yshift=-0.3ex, "e^{tf}\cdot"']
& (\Omega^\bullet(L), d_t). 
\end{tikzcd}$$
Therefore, the contraction $(i'_{t, \lambda_0}, p'_{t, \lambda_0}, h'_{t, \lambda_0})$ on $(\Omega^\bullet_t(L), \displaystyle{\bigoplus_{0\leq \lambda\leq \lambda_0}\Omega^{\bullet}_t(L)_{\lambda}})$ pulls back to $(i_{t, \lambda_0}, p_{t, \lambda_0}, h_{t, \lambda_0})$ on $(\Omega^\bullet(L), \displaystyle{\bigoplus_{0\leq \lambda\leq \lambda_0}e^{tf}\Omega^{\bullet}_t(L)_{\lambda}})$.

\begin{corollary}\label{wittendefcontr}
    $(i_{t, \lambda_0}, p_{t, \lambda_0}, h_{t, \lambda_0})$ is a strong contraction. Moreover, the constant-1 function $\mathbf{1}$ satisfies $i_{t, \lambda_0}(p_{t, \lambda_0}(\mathbf{1})) = \mathbf{1}$.
\end{corollary}

\paragraph{Witten-Morse Contraction}
An important property of the Witten complex is that for a Morse-Smale pair $(f, g)$ and sufficiently large $t$, \\
$(\displaystyle{\bigoplus_{0\leq \lambda\leq \lambda_0}\Omega^{\bullet}_t(L)_{\lambda}}, d_t)$ is  chain isomorphic to the Morse complex $CM^\bullet(f)$. This allows one to use the $\lambda_0$-harmonic contractions for Witten Laplacians $(i_{t, \lambda_0}, p_{t, \lambda_0}, h_{t, \lambda_0})$ to induce a contraction $(i, p, h)$  from de Rham complex $(\Omega^\bullet(L), d)$ to Morse complex $(CM^\bullet(f), d_{Morse})$ via Witten  complex, called a Witten-Morse Contraction. We summarise the results from Witten-Morse theory needed for constructing such a contraction below, and reader are referred to \cite{zhangwittendef} and the reference therein for further details.

\begin{prop}
    Given a closed oriented Riemannian manifold $(L, g)$ and a Morse function $f$ such that $(g,f)$ is a Morse-Smale pair,

    \begin{enumerate}
        \item \cite[Theorem 6.4]{zhangwittendef} There is a graded quasi-isomorphism (de Rham map) $$p: (\Omega^\bullet(L), d) \rightarrow (CM^\bullet(f), d_{Morse})$$
        defined by integrating differential forms along unstable submanifolds associated to critical points of $f$.

    \item \cite[Theorem 6.9]{zhangwittendef} For each fixed $\lambda_0 \geq 0$, there exists $t_0>0$ such that for each $t\geq t_0$, the composition of the following chain maps
$$(\displaystyle{\bigoplus_{0\leq \lambda\leq \lambda_0}\Omega^{\bullet}_t(L)_{\lambda}}, d_t) \subseteq (\Omega^{\bullet}_t(L), d_t) \xrightarrow{e^{tf}\cdot}
    (\Omega^{\bullet}(L), d) \xrightarrow{p} (CM^\bullet(f), d_{Morse})$$
    is a chain isomorphism.
        
    \end{enumerate}

\end{prop}

It follows that $p$ factors through 
$p_{t, \lambda_0}: \Omega^{\bullet}(L) \rightarrow \displaystyle{\bigoplus_{0\leq \lambda\leq \lambda_0}e^{tf}\Omega^{\bullet}_t(L)_{\lambda}}$, inducing a chain isomorphism $(\displaystyle{\bigoplus_{0\leq \lambda\leq \lambda_0}e^{tf}\Omega^{\bullet}_t(L)_{\lambda}}, d) \xrightarrow{\sim} (CM^\bullet(f), d_{Morse})$. Denote its inverse by $i: CM^\bullet(f) \xrightarrow{\sim} \displaystyle{\bigoplus_{0\leq \lambda\leq \lambda_0}e^{tf}\Omega^{\bullet}_t(L)_{\lambda}}\subseteq \Omega^{\bullet}(L)$, which factors through $i_{t, \lambda_0}$. Define $h\coloneqq h_{t, \lambda_0}: \Omega^{\bullet}(L) \rightarrow \Omega^{\bullet-1}(L)$. Hence, together with Corollary \ref{wittendefcontr}, we have shown the following:
\begin{corollary}
\label{wittenmorsecontr}
    For each fixed $\lambda_0 \geq 0$, there exists $t_0>0$ such that for each $t\geq t_0$, there exists a strong contraction (a Witten-Morse contraction) $(i, p, h)$ of $((\Omega^\bullet(L), d), (CM^\bullet(f), d_{Morse}))$. Moreover, the constant-1 function $\mathbf{1}$ satisfies $i(p(\mathbf{1})) = \mathbf{1}$.
\end{corollary}

\subsubsection{Transfer of $A_\infty$ algebra structures}
In this subsection, we recall the homological perturbation lemma for (unital) filtered $A_\infty$ algebras, following the version stated in \cite[Theorem 4.4, Proposition 4.7]{yuannonarchi}.

\begin{prop}
\label{algebratransfer}
    Given a contraction $(i, p, h)$ of $\overline{C}, \overline{H}$, for any $\mathbb{G}$-gapped filtered $A_\infty$ algebra structure $(C^\bullet, \{m_k = \sum_{\beta\in\mathbb{G}}m_{k, \beta}T^{E(\beta)}\})$ on $C$ with $m_{1,\beta_0} = d$, there exists a natural  $\mathbb{G}$-gapped filtered $A_\infty$ algebra structure $(H^\bullet, \{m^H_k = \sum_{\beta\in\mathbb{G}}m^H_{k, \beta}T^{E(\beta)}\})$ on $H$ with $m^H_{1,\beta_0} = \delta$ and a natural  $\mathbb{G}$-gapped filtered $A_\infty$ morphism
    $$\Tilde{i} = \{\Tilde{i}_k = \sum_{\beta\in\mathbb{G}}\Tilde{i}_{k, \beta}T^{E(\beta)}\}_{k\geq 1}: (H^\bullet, \{m^H_k\}) \rightarrow (C^\bullet, \{m_k\})$$
    such that $\Tilde{i}_{k, \beta_0} = i$.
    
    If in addition $(C^\bullet, \{m_k\})$ admits a strict unit $e \in \overline{C}^0$ such that $i(p(e)) = e$, then $e_H \coloneqq p(e)\in \overline{H}^0$ is a strict unit for $(H^\bullet, \{m^H_k\})$ such that $\Tilde{i}$ is unital.
\end{prop}

\begin{remark} \label{hplmcspace}
    $\tilde{i}$ induces a map between their weak Maurer-Cartan sets
     \begin{equation}
    exp(\tilde{i}): \widehat{MC}_{weak}(H) \rightarrow \widehat{MC}_{weak}(C)
    \end{equation}
     respecting their potential functions, i.e. $W_C(exp(\tilde{i})(b)) = W_H(b)$.\\
     Moreover, it descends to a map between their weak Maurer-Cartan spaces:
    \begin{equation} \label{hplmcisom}
    exp(\tilde{i}): MC_{weak}(H) \rightarrow MC_{weak}(C).
    \end{equation}
    
    In case of strong contraction, $i$ is a chain homotopy equivalence,  and hence   $\tilde{i}$ is an $A_\infty$ homotopy equivalence by $A_\infty$ Whitehead Theorem (see e.g. \cite[Theorem 4.2.45]{FOOO}). Therefore, (\ref{hplmcisom}) is a bijection.
   
\end{remark}

As a corollary, one can construct canonical models for any $A_\infty$ algebras:

\begin{corollary} \label{canonicalalg}
    Given any $\mathbb{G}$-gapped, unital filtered $A_\infty$ algebra $C$, there exists a $\mathbb{G}$-gapped, unital filtered $A_\infty$ algebra structure $(H^\bullet(C) \coloneqq H^\bullet(\overline{C}, m_{1, \beta_0}; \Lambda_0), \{m^H_k = \sum_{\beta\in\mathbb{G}}m^H_{k, \beta}T^{E(\beta)}\}, [e])$ on $H(C)$, called a canonical model of $C$. 
\end{corollary}

\begin{proof}
   Applying Proposition \ref{algebratransfer} to the strong contraction constructed from Proposition \ref{strongcontractionconstr} associated to the chain complex $(\overline{C}^\bullet, m_{1, \beta_0})$ yields the desired $A_\infty$ algebra structure on $H(C)$. Unitality follows from the fact that $e\in \mathcal{H}$ (as long as $H(C) \neq 0$, which WLOG can be assumed).
\end{proof}

\begin{remark}\label{alglowdeg}
    From the definition it follows that $\overline{m}^H_{2} = \cdot$ is the cohomological product induced by $\overline{m}_{2}$, which is therefore independent of the choice of a strong contraction. In particular, the unital algebra $(H^\bullet(\overline{C}), \overline{m}^H_{2}, [e])$ is the usual cohomology  ring induced from the classical $A_\infty$ algebra $(\overline{C}, \{\overline{m}_{k}\}, e)$.
\end{remark}

\subsubsection{Transfer of $A_\infty$ tri-module structures}

In this subsection, we prove the following transfer theorem for $A_\infty$ tri-modules: 

\begin{prop} \label{hpltrimod}
Given a contraction $(i_D, p_D, h_D)$ of $(\overline{D}, d_D), (\overline{H}_D, \delta_D)$, for any $\mathbb{G}$-gapped filtered left $C''$, right $(C', C)$ $A_\infty$ tri-module $D$  with $\overline{n}_{k'', k', k} = d_D$, there exists a natural $\mathbb{G}$-gapped filtered left $C''$, right $(C', C)$ $A_\infty$ tri-module $(H_D^\bullet, \{n^H_{k'', k', k}\})$  with $\overline{n}^H_{k'', k', k, \beta_0} = \delta_D$, and a natural  $\mathbb{G}$-gapped filtered $A_\infty$ tri-module morphism
    $$\Tilde{i}^D = \{\Tilde{i}^D_{k'',k', k} = \sum_{\beta\in\mathbb{G}}\Tilde{i}^D_{k'',k', k, \beta}T^{E(\beta)}: (H_D^\bullet, \{n^H_{k'', k', k}\}) \rightarrow (D^\bullet, \{n_{k'', k', k}\})\}$$
    such that $\Tilde{i}^D_{0, 0, 0, \beta_0} = i_D$.

    If in addition  $C$ (resp. $C', C''$) has a strict unit $e$ (resp. $e', e''$) such that $D$ is unital with respect to $(e'', e', e)$, and $(i_D, p_D, h_D)$ is a strong contraction,  then $H_D$ is also unital with respect to $(e'', e', e)$.
\end{prop}

\begin{proof}
    The construction of $\Tilde{i}^D$ and $\{n^H_{k'', k', k}\}$ are analogous to those for $A_\infty$ bi-modules as constructed in the proof of \cite[Theorem 5.4.18]{FOOO} (in which although they assumed that $C, C'$ are canonical $A_\infty$ algebras and the bimodule contraction is for canonical model, the same formulae hold without these assumptions). Therefore, we will just provide the following inductive formulae for  $\Tilde{i}^D_{k'',k', k, \beta}$ and $n^H_{k'', k', k, \beta}$ for $(k, k', k'', \beta)\neq (0, 0, 0, \beta_0)$ respectively:
 \begin{align*}\sum_{\substack{k_i, k'_i, k''_i \geq 0\\k_1 + k_2=k\\k'_1 + k'_2=k'\\k''_1 + k''_2=k''}}
    \sum_{\substack{\beta_1+\beta_2 = \beta\\(k_1, k'_1, k''_1, \beta_1) \\ \neq (0, 0, 0, \beta_0)}}h_D \circ n_{k''_1, k'_1, k_1, \beta_1} (Id_{C''}^{\otimes k''_1} \otimes \Tilde{i}^D_{k''_2, k'_2, k_2, \beta_2} \otimes Id_{C'}^{\otimes k'_1} \otimes Id_{C}^{\otimes k_1}),\end{align*}
    %\begin{align*}
    %&\Tilde{i}^D_{k'',k', k, \beta} \coloneqq \sum_{\substack{k_i, k'_i, k''_i \geq 0\\k_1 + k_2=k\\k'_1 + k'_2=k'\\k''_1 + k''_2=k''}} \\
    %&\sum_{\substack{\beta_1+\beta_2 = \beta\\(k_1, k'_1, k''_1, \beta_1)\neq (0, 0, 0, \beta_0)}}h_D \circ n_{k''_1, k'_1, k_1, \beta_1} (Id_{C''}^{\otimes k''_1} \otimes \Tilde{i}^D_{k''_2, k'_2, k_2, \beta_2} \otimes Id_{C'}^{\otimes k'_1} \otimes Id_{C}^{\otimes k_1}),
    %\end{align*}
\begin{align*}\sum_{\substack{k_i, k'_i, k''_i \geq 0\\k_1 + k_2=k\\k'_1 + k'_2=k'\\k''_1 + k''_2=k''}} 
    \sum_{\substack{\beta_1+\beta_2 = \beta\\(k_1, k'_1, k''_1, \beta_1) \\ \neq (0, 0, 0, \beta_0)}}p_D \circ n_{k''_1, k'_1, k_1, \beta_1} (Id_{C''}^{\otimes k''_1} \otimes \Tilde{i}^D_{k''_2, k'_2, k_2, \beta_2} \otimes Id_{C'}^{\otimes k'_1} \otimes Id_{C}^{\otimes k_1}),\end{align*}
    %\begin{align*}
    %&n^H_{k'',k', k, \beta} \coloneqq \sum_{\substack{k_i, k'_i, k''_i \geq 0\\k_1 + k_2=k\\k'_1 + k'_2=k'\\k''_1 + k''_2=k''}} \\
    %&\sum_{\substack{\beta_1+\beta_2 = \beta\\(k_1, k'_1, k''_1, \beta_1)\neq (0, 0, 0, \beta_0)}}p_D \circ n_{k''_1, k'_1, k_1, \beta_1} (Id_{C''}^{\otimes k''_1} \otimes \Tilde{i}^D_{k''_2, k'_2, k_2, \beta_2} \otimes Id_{C'}^{\otimes k'_1} \otimes Id_{C}^{\otimes k_1}),
    %\end{align*}
and $\Tilde{i}^D_{0, 0, 0, \beta_0} \coloneqq i_D; n^H_{0, 0, 0, \beta_0} \coloneqq \delta_D$. Compare \cite[Formula 5.4.5, 5.4.6]{FOOO} for the case of $A_\infty$ bimodules. Unitality follows from the fact that $(i_D, p_D, h_D)$ is a strong contraction (See e.g. \cite[Proposition 4.7 (iii)]{yuannonarchi} for the case of $A_\infty$ algebras, which also uses strong contraction properties and inductive formulae).
\end{proof}

Combining Propositions \ref{hpltrimod},\ref{algebratransfer} and Definition \ref{pullbacktrimod} yields the following: 

\begin{corollary}
\label{hpltrimodbasechange}
    If in addition, we are given a contraction
    $(i, p, h)$ (resp. $(i', p', h')$, $ (i'', p'', h'')$) of $(\overline{C}, \overline{H})$ (resp. $(\overline{C}', \overline{H'})$, $(\overline{C}'', \overline{H''})$), 
inducing $A_\infty$ algebras $(H^\bullet, \{m^H_k\})$ (resp. $(H'^\bullet, \{m'^H_k\})$, $(H''^\bullet, \{m''^H_k\})$) and $A_\infty$ morphisms $\Tilde{i}$ (resp. $\Tilde{i'}$, $\Tilde{i''}$)
as in Proposition \ref{algebratransfer}, then the pullback $A_\infty$ tri-module $(\Tilde{i''}, \Tilde{i'}, \Tilde{i})^*H_D$ is a $\mathbb{G}$-gapped filtered left $H''$, right $(H', H)$ $A_\infty$ tri-module, and the pullback $A_\infty$ tri-module morphism $(\Tilde{i''}, \Tilde{i'}, \Tilde{i})^* \Tilde{i}_D$ is a  $\mathbb{G}$-gapped filtered $A_\infty$ tri-module morphism over $(\Tilde{i''}, \Tilde{i'}, \Tilde{i})$.

    If in addition $C$ (resp. $C', C''$) has a strict unit $e$ (resp. $e', e''$) such that $D$ is unital, and all the contractions are strong contractions such that 
    $$i(p(e)) = e; i'(p'(e')) = e'; i''(p''(e'')) = e'',$$
    then $(\Tilde{i''}, \Tilde{i'}, \Tilde{i})^*H_D$ is unital with respect to $(e''_H, e'_H, e_H)$.
\end{corollary}

Similarly, one can construct canonical models for any $A_\infty$ tri-modules:

\begin{corollary} \label{canonicaltrimod}
    For any $\mathbb{G}$-gapped, unital filtered left $C''$, right $(C', C)$ $A_\infty$ tri-module $(D^\bullet, \{n_{k'', k', k}= \sum_{\beta\in\mathbb{G}}n_{k'', k', k, \beta}T^{E(\beta)} \})$, there exists a $\mathbb{G}$-gapped, unital filtered left $H(C'')$, right $(H(C'), H(C))$ $A_\infty$ tri-module on $H(D)$
    $$(H^\bullet(D)\coloneqq H^\bullet(\overline{D}, \overline{n}_{0, 0, 0}; \Lambda_0), \{n^H_{k'', k', k}= \sum_{\beta\in\mathbb{G}}n^H_{k'', k', k, \beta}T^{E(\beta)} \})$$ 
    called a canonical model of $D$, where $H(C'')$ (resp. $H(C'), H(C)$) is a canonical model of $C''$ (resp. $C', C$) defined in Corollary \ref{canonicalalg}. 
\end{corollary}

\begin{remark}\label{trimodlowdeg}
    Observe that $\overline{n}^H_{1,0,0}$ is the cohomological left $H(\overline{C}'')$-module action on $H(\overline{D})$ induced from $\overline{n}_{1,0,0}$, which is therefore independent of the choice of strong contractions. Similarly for $\overline{n}^H_{0,1,0}$ and $\overline{n}^H_{0,0,1}$.
\end{remark}

\subsection{Lagrangian Floer Theory} \label{lft}

In this section, we recall the de Rham model of the (Lagrangian) Floer complex $CF(L)$ associated to a Lagrangian $L$ in a symplectic manifold $X$. It was first due to Fukaya in \cite{fukcyclic}, with further details on the Kuranishi structures and virtural fundamental chains (with an application to constructing the de Rham model of the Floer complex) in \cite{foookuranishi}. It is further generalised to unobstructed immersed Lagrangian correspondences in \cite{Fukaya-corr}, for which we will mainly follow. The main theorem for this section is as follows.

\begin{theorem} \cite[Theorem 3.14]{Fukaya-corr} \label{floercpx}
Given a closed (or tame at infinity) symplectic manifold $(X,\omega)$ and a $V$-relatively spin \footnote{i.e. an  oriented real vector bundle $V$ over $X$ such that $V|_L \oplus TL$ is spin.}, closed, connected, embedded Lagrangian submanifold $L$ of $(X,\omega)$. The completed de Rham complex $CF(L)\coloneqq \Omega(L;\Lambda_0)$, admits a (strictly) unital, $\mathbb{G}_L$-gapped filtered $A_{\infty}$ algebra structure \\
$(CF(L),\{m_k\}, e)$ for some discrete submonoid $\mathbb{G}_L\subseteq \mathbb{R}_{\geq 0}$.
\end{theorem}

\begin{remark}
    Actually, in loc. cit. the above statement holds for immersed Lagrangian $\widetilde{L}\rightarrow L\subseteq X$ with clean self-intersections (i.e. $\widetilde{L} \times_X \widetilde{L}$ is a clean fiber product). We focus on embedded Lagrangians for simplicity. 
\end{remark}

More precisely, for each $\beta\in \mathfrak{G}\coloneqq H_2(X,L;\mathbb{Z})$, an $\omega$-compatible almost complex structure $J$ and $k\in \mathbb{Z}_{\geq 0}$, there exists an oriented Kuranishi structure $\widehat{\mathcal{U}}$ on the compactified moduli space of pseudo-holomorphic disks $\mathcal{M}_{k+1}(L; \beta) \coloneqq \mathcal{M}_{k+1}(X,L; J, \beta)$. By \cite{fooodisk2}, these K-spaces form a tree-like K-system 
$$(\{\mathcal{M}_{k+1}(L, \beta), ev=(ev_0, \cdots, ev_k): \mathcal{M}_{k+1}(L; \beta) \rightarrow L^{k+1}, E: \mathfrak{G} \rightarrow \mathbb{R} \}_{k\geq0;\beta\in \mathfrak{G}}$$
which, after choosing a compatible system of CF-perturbations \\
$\{\mathfrak{S}_{k+1}(L;\beta)\}_{k\geq0;\beta\in \mathfrak{G}}$, gives rise to a $\mathbb{G}_L$-gapped filtered $A_\infty$ algebra structure $\{m_k\}_{k\geq0}$ on $CF(L)$ by \cite{foookuranishi}, where $\mathbb{G}_L \subseteq \mathbb{R}_{\geq 0} $ is the submonoid generated by
$$\mathbb{G}^0_L\coloneqq\{E(\beta)| \beta \in \mathfrak{G}; \mathcal{M}_{k+1}(L, \beta) \neq \phi\}$$
which is discrete by Gromov's Compactness Theorem.

Moreover, by \cite[Proposition 3.35]{Fukaya-corr}, the constant one function $1\in\Omega^0(L)$ defines a (strict) unit $e$ of the $A_\infty$ algebra $(CF(L), \{m_k\})$.

Furthermore, after choosing a harmonic contraction $(i, p, h)$ of \\
$(\Omega^\bullet(L; \R), H^\bullet(L; \mathbb{R}))$ by Definition \ref{harmcontr}, we apply Proposition \ref{algebratransfer} to $CF(L)$ to obtain a canonical model $CF_{can}(L)$ as a unital, $\mathbb{G}_L$-gapped filtered $A_{\infty}$ algebra.

\begin{remark}
    Note that from the Remark \ref{hplmcspace}, the weak Maurer-Cartan spaces of $CF(L)$ and $CF_{can}(L)$ are (canonically) isomorphic. Therefore, by an abuse of notations, we will denote both of them as  $MC_{weak}(L)$.
\end{remark}

\subsection{Lagrangian correspondences and their compositions} \label{subsec-lagcorrcomp}
In this subsection, we review the concept of Lagrangian correspondences and their compositions, especially the notion of clean compositions which appears naturally later in our theory of equivariant correspondence tri-modules.

Throughout this subsection, $(M, \omega_M)$ denotes a smooth manifold $M$ with a symplectic form $\omega_M$. Also, we denote by $M^- := (M, -\omega_M)$ the symplectic manifold endowed with the negative symplectic form $-\omega_M$.
\begin{defn}
A Lagrangian correspondence $L$ from $(M, \omega_M)$ to $(N, \omega_N)$, denoted as $M \xrightarrow{L} N$, is a Lagrangian submanifold in $(M^- \times N, -\omega_M\oplus\omega_N)$.
\end{defn}

\begin{defn}
Given two Lagrangian correspondences $P \xrightarrow{L'} M$ and $M \xrightarrow{L} N$, their geometric composition $ L \circ L'$ is a subset of $P \times N$ defined as
\begin{equation}
L \circ L' = pr_{PN}((L' \times L)\cap (P \times \Delta_{M} \times N)) = pr_{PN}(L' \times_{M} L),
\end{equation}
where $pr_{PN} : P \times M \times M \times N \rightarrow P \times N$ is the natural projection.\\
If $L \circ L' \subseteq P^- \times N$ is a Lagrangian submanifold, we say the pair $(L', L)$ is composable, and regard $L \circ L'$ as a Lagrangian correspondence $P \xrightarrow{L \circ L'} N$.
\end{defn}

%\begin{remark}
%For simplicity, we will abbreviate above as ``the geometric composition of $P \xrightarrow{L'} M \xrightarrow{L} N$".
%\end{remark}

\begin{remark}
    Alternatively, one can define $L \circ L'$ as
\begin{equation}
L \circ L' = pr_{12}((P\times N)\times L' \times L)\cap \Delta_{PMN}),
\end{equation}
where $pr_{12} : P\times N \times P \times M \times M \times N \rightarrow P \times N$ is the projection to first two factors.\\

The equivalence of these two definitions follows from the fact that under the canonical bijection $P \times \Delta_{M} \times N \cong \Delta_{PMN}$, $L' \times_{M} L$ is identified with $((P\times N)\times L' \times L)\cap \Delta_{PMN}$. While the former definition is cleaner, the latter definition has the advantage of having the following equality:
$$((P\times N)\times L' \times L)\cap \Delta_{PMN} = ((L \circ L') \times L' \times L)\cap \Delta_{PMN},$$
which is more consistent with the construction of correspondence tri-modules in later sections. We will use both definitions interchangeably.
\end{remark}

A priori, $L' \times_{M} L$ needs not be smooth, and even if $L' \times_{M} L$ is smooth, its projection $L \circ L'$ needs not be smooth. We recall the following notion of clean composition as follows:

\begin{defn} \label{cleancomp}
    We say $L \circ L'$ is a clean composition (or $(L', L)$ is cleanly composable)  if the following are satisfied:
    \begin{enumerate}
        \item $L' \times L$ intersects cleanly with $P \times \Delta_{M} \times N$ in $P \times M \times M \times N$, i.e. 
        $$L' \times_{M} L = (L' \times L)\cap (P \times \Delta_{M} \times N) \subseteq P \times M \times M \times N$$ is a smooth submanifold with 
        $$T(L' \times_{M} L) = T(L' \times L) \cap T (P \times \Delta_{M} \times N).$$ 

        \item $L \circ L' \subseteq P^- \times N$ is a smooth submanifold.

        \item $pr_{PN}$ restricts to a smooth fibration $\pi_{L\circ L'}: L' \times_{M} L \rightarrow L \circ L'$.
    \end{enumerate}

\end{defn} 

\begin{remark}
    It turns out once $L \circ L'$ is a clean composition, the Lagrangian property of $L \circ L' \subseteq P^- \times N$ is automatically satisfied. See e.g. \cite[Lemma 2.1.7, (ii)]{mcduff-salamon} in the context of linear coisotropic reduction.
\end{remark}

Two important special cases in which the above hold are as follows:

\begin{defn}
      We say $L \circ L'$ is 
      \begin{enumerate}
          \item a transversal composition if it is a clean composition with (1) strengthened to \\
      (1)' $L' \times L$ intersects transversely with $P \times \Delta_{M} \times N$ in $P \times M \times M \times N$.
      
\item an embedded composition if (1)' is satisfied and $pr_{PN}|_{L' \times_{M} L}: L' \times_{M} L \rightarrow P \times N$ is a smooth embedding.
\end{enumerate}

\end{defn}

\begin{remark}
    Note that for an embedded composition $L \circ L'$, it is smooth with $\pi_{L \circ L'}$ being a diffeomorphism, and so (2) and (3) of Definition \ref{cleancomp} 
    are satisfied.
\end{remark}

\subsection{Correspondence tri-module} \label{subsec-corrtrimod}
In this section, we review a generalisation of Lagrangian Floer theory to Lagrangian correspondences, pioneered by Wehrheim-Woodward in their study of quilted Floer theory (see e.g. \cite{wwquilted} and the reference therein). We will follow Fukaya's construction of the correspondence tri-module $CF(L''; L', L)$ associated to a triple of Lagrangian correspondences $(L, L', L'')$ in \cite{Fukaya-corr} stated below:

\begin{theorem} \cite[Proposition 8.7]{Fukaya-corr} \label{corrtrimod}
    Given three symplectic manifolds \\
    $P, M, N$ which are closed or tame at infinity, and three relatively spin, closed, connected, embedded Lagrangian correspondences $P \xrightarrow{L''}N, P \xrightarrow{L'}M, M \xrightarrow{L}N $ such that the following intersection is clean:
    $$\mathcal{I}_{L'', L', L} \coloneqq (L'' \times L' \times L)\cap \Delta_{PMN} \subseteq P \times N \times P\times M \times M \times N.$$ 
    Then the completed de Rham complex $CF(L''; L', L)= \Omega(\mathcal{I}_{L'', L', L}; \Lambda_0)$ admits a strictly unital, $\mathbb{G}_{L'', L', L}$-gapped filtered left $CF(L'')$, right \\
    $(CF(L'), CF(L))$ $A_{\infty}$ tri-module structure $\{n_{k'',k',k}\}$ for some $\mathbb{G}_{L'', L', L}$.

\end{theorem}

\begin{remark}
    Actually, in \cite{Fukaya-corr}, Fukaya proved the above for immersed Lagrangian correspondences $\widetilde{L}, \widetilde{L'}, \widetilde{L''}$ with clean self-intersections such that      
    $$\widetilde{\mathcal{I}}_{L'', L', L}\coloneqq (\widetilde{L''} \times \widetilde{L'} \times \widetilde{L})\times_{(P\times M \times N)^2} \Delta_{PMN}$$ is clean.
    We focus on embedded Lagrangian correspondences for simplicity. 
\end{remark}

\begin{remark}
    Following \cite{Fukaya-corr}, the relative spin structure on $L$ (resp. $L'$, $L''$) is with respect to
 $pr^*_M(V_M \oplus TM) \oplus pr^*_NV_N$ (resp.  $pr^*_P(V_P \oplus TP) \oplus pr^*_MV_M; pr^*_P(V_P \oplus TP) \oplus pr^*_NV_N)$
 for some oriented real vector bundles $V_P\rightarrow P; V_M\rightarrow M; V_N\rightarrow N$.
    %the relative spin structure of $L$ is with respect to $pr^*_M(V_M \oplus TM) \oplus pr^*_NV_N$ for some oriented real vector bundles $V_M \rightarrow M$ and $V_N \rightarrow N$. Similarly for $L'$ and $L''$ (with common $V_P$, $V_M$ and $V_N$).
\end{remark}

\begin{remark}
    Remark of conventions in the case where $P = pt$: In \cite[Section 5]{Fukaya-corr}, Fukaya wrote the tri-module as $CF(L', L; L'')$, and treated it as a left $(CF(L'), CF(L))$, right $CF(L'')$ $A_{\infty}$ tri-module, which is opposite to the above. It is for the purpose of showing the compatibility of compositions via ``Y-diagrams". See \cite[Section 9]{Fukaya-corr} for details. 
    
    Meanwhile, for the simplicity of the exposition, we will not make such a distinction, i.e. we always consider $CF(L''; L', L)$ as a left $CF(L'')$, right $(CF(L'), CF(L))$ $A_{\infty}$ tri-module, regardless of whether $P = pt$ or not. Alternatively, one can adopt Fukaya's convention and prove analogous statements by replacing left $CF(L'')$, right $(CF(L'), CF(L))$ tri-module by left $(CF(L'), CF(L))$, right $CF(L'')$ tri-module when $P = pt$.      
\end{remark}

%$\Omega((\widetilde{L''} \times \widetilde{L'} \times \widetilde{L})\times_{(P\times M \times N)^2} \Delta_{PMN}; \Lambda_0)$

In what follows, we briefly recall Fukaya's construction in \cite[Section 8.2]{Fukaya-corr}, and refer the reader to there for details. We follow his convention that quilts are $-J$ holomorphic (and hence have nonpositive energies $-E$) to ensure that the quilts are compatible with the anti-holomorphic maps that appeared in the moduli space of Y-diagrams. See \cite[Remark 9.4]{Fukaya-corr} for details. 

Roughly speaking, $\{n_{k'',k',k}\}$ counts the moduli space of quilted drums $\{\mathcal{M}_{k'',k',k}(L''; L',  L; E)\}_{E\geq 0}$. Its interior $\mathring{\mathcal{M}}_{k'',k',k}(L''; L',  L; E)$ consists of quilted drums $u = (u_P, u_M, u_N): \Sigma_1 \times \Sigma_2 \times \Sigma_3 \rightarrow P \times M \times N$, where 
\begin{itemize}
    \item The quilted drum $W \coloneqq \mathbb{S}^1\times \mathbb{R} \cong ([0,3]/\sim) \times \mathbb{R}$
    %$$= (([0,1] \times \mathbb{R}) \cup ([1,2] \times \mathbb{R}) \cup([2,3] \times \mathbb{R}))/\sim$$ 
    is a quilted cylinder with three patches $W_i \coloneqq [i-1,i] \times \mathbb{R}$  and three seams $\sigma_{i} \coloneqq \{i\} \times \mathbb{R} = W_i \cap W_{i+1} $ for $1\leq i \leq 3$ (with convention $W_4 = W_1$).
    \item $\Sigma$ is a bordered Riemann surface as the union of $W$ with trees of sphere components whose roots are not on the seams $\{\sigma_{i}\}$. Similarly for $\Sigma_i$ by replacing $W$ with $W_i$ above.
    \item $z^{(i)} = (z^{(i)}_1, \cdots z^{(i)}_{k_i})$ are marked points on $\sigma_i$, where $1 \leq i \leq 3$ and $k_1 = k', k_2 = k, k_3 = k''$.
    \item $u_P: \Sigma_1 \rightarrow P, u_M: \Sigma_2 \rightarrow M, u_N: \Sigma_3 \rightarrow N$ are  $-J_P$ (resp. $-J_M$, $-J_N$)-holomorphic maps satisfying the following seam conditions, asymptotic conditions, an energy condition and stability conditions.
    \item{[Seam conditions]} $$(u_P, u_M)|_{\sigma_1}: \sigma_1 \rightarrow L' \subseteq P \times M,$$
    $$(u_M, u_N)|_{\sigma_2}: \sigma_2 \rightarrow L \subseteq M \times N,$$
    $$(u_P, u_N)|_{\sigma_3}: \sigma_3 \rightarrow L'' \subseteq P \times N.$$
    \item {[Asymptotic conditions]} For any $t\in [0,1]$, the limits\\
    $\lim_{\tau\rightarrow \pm \infty}u_P(t, \tau)$ exist and are independent of $t$. Denote the limits as $p_{-\infty},  p_{+\infty} \in P$ respectively. Similarly, assume the following limits exist: $$\lim_{\tau\rightarrow \pm \infty}u_M(t, \tau) = m_{\pm \infty},$$
    $$\lim_{\tau\rightarrow \pm \infty}u_N(t, \tau) = n_{\pm \infty}.$$
    It follows from the seam conditions that %$$(p_{+\infty}, m_{+\infty})\in L'; (m_{+\infty}, n_{+\infty})\in L; (p_{+\infty}, n_{+ \infty})\in L''$$
    %$$(p_{-\infty}, m_{-\infty})\in L'; (m_{-\infty}, n_{-\infty})\in L; (p_{-\infty}, n_{-\infty})\in L''$$
    %and hence
    $$(p_{+\infty}, n_{+\infty}, p_{+\infty}, m_{+\infty}, m_{+\infty}, n_{+\infty})\in \mathcal{I}_{L'', L', L},$$
    $$(p_{-\infty}, n_{-\infty}, p_{-\infty}, m_{-\infty}, m_{-\infty}, n_{-\infty})\in \mathcal{I}_{L'', L', L}.$$
    
    \item{[Energy condition]} $$E(u)\coloneqq \int_{\Sigma_1}u^*_P\omega_P + \int_{\Sigma_2}u^*_M\omega_M + \int_{\Sigma_3}u^*_N\omega_N = -E.$$

    \item{[Stability conditions]} 
    The automorphism group $Aut(u)$ (in the sense as in \cite[Definition 8.18]{Fukaya-corr}) is finite.
\end{itemize}

The evaluation map at the marked points $\{z^{(2)}_j\}$ on the seam $\sigma_2$ with target $L$, $ev_L = (ev_1, \cdots, ev_k): \mathring{\mathcal{M}}_{k'',k',k}(L''; L',  L; E) \rightarrow L^k$, is defined as  
$$ev_L(u) = ((u_M(z^{(2)}_1), u_N(z^{(2)}_1)), \cdots, ((u_M(z^{(2)}_k), u_N(z^{(2)}_k)).$$ 
Similarly for the other evaluation maps $ev_{L'}$, $ev_{L''}$.

Moreover, the asymptotic conditions induce the evaluation maps at the infinity ends $ev_{\pm\infty}: \mathring{\mathcal{M}}_{k'',k',k}(L''; L',  L; E) \rightarrow \mathcal{I}_{L'', L', L}$, defined as 
$$ev_{+\infty}(u) \coloneqq (p_{+\infty}, n_{+\infty}, p_{+\infty}, m_{+\infty}, m_{+\infty}, n_{+\infty}),$$
$$ev_{-\infty}(u) \coloneqq (p_{-\infty}, n_{-\infty}, p_{-\infty}, m_{-\infty}, m_{-\infty}, n_{-\infty}).$$

By \cite[Proposition 8.19]{Fukaya-corr}, $\mathring{\mathcal{M}}_{k'',k',k}(L''; L',  L; E)$ can be compactified to a Kuranishi space with corners $\mathcal{M}_{k'',k',k}(L''; L',  L; E)$ such that
$ev_L, ev_L', ev_L'',$ \\
$ev_{\pm\infty}$ extend to strongly smooth maps with $ev_{+\infty}$ being weakly submersive. Moreover, by \cite[Proposition 8.20]{Fukaya-corr}, for each fixed $E_0>0$, for all $E<E_0$, $\mathcal{M}_{k'',k',k}(L''; L',  L; E)$ admits a system of CF-perturbations $\widehat{\mathfrak{S}}$ such that 
\begin{itemize}
    \item $\widehat{\mathfrak{S}}$ are outer collaring of the thickenings of $\mathfrak{S}$.
\item $\widehat{\mathfrak{S}}$ is transversal to 0.
\item $ev_{\pm\infty}$ are strongly submersive.
\end{itemize}
After these setup, for each $E_0$, we define $n^{E_0, \epsilon}_{k'',k',k}$ as a map $$CF(L'')^{\otimes k''} \otimes CF(L''; L', L) \otimes CF(L')^{\otimes k'}\otimes CF(L)^{\otimes k} \rightarrow CF(L''; L', L)$$
by $n^{E_0, \epsilon}_{k'',k',k} = \sum_{E<E_0}T^E n^{E, \epsilon}_{k'',k',k}$, where
\begin{itemize}
\item $n^{E, \epsilon}_{k'',k',k}(x''; y; x', x) \coloneqq ev_{+\infty !}(ev_{L''}^*x'' \wedge ev^*_{-\infty}y \wedge ev_{L'}^*x' \wedge ev_L^*x; \widehat{\mathfrak{S}}^\epsilon).$
    \item $x = (x_1, \cdots, x_k) \in \Omega(L)^{\otimes k}$. Similarly for $x'$ and $x''$.
 %   \item $y\in CF(L''; L', L).$
 %   \item $ev_L = (ev_1, \cdots, ev_k): \mathcal{M}(L''; L',  L) \rightarrow L^k$ is the evaluation map at the marked points on the seam with target $L$. Similarly for $ev_{L'}$, $ev_{L''}$.
   \item $ev_L^*x \coloneqq (ev_1)^*x_1 \wedge \cdots (ev_k)^*x_k$. Similarly for $ev_{L'}^*x', ev_{L''}^*x''$.
   \item $\widehat{\mathfrak{S}}^\epsilon$ is the restriction of $\widehat{\mathfrak{S}}$ to a particular $\epsilon > 0$.
 %   \item $ev_{+\infty}: \mathcal{M}(L''; L',  L) \rightarrow (\widetilde{L''} \times \widetilde{L'} \times \widetilde{L})\times_{(P\times M \times N)^2} \Delta_{PMN}$ is the  asymptotic boundary map when approaching to $+\infty$. Similarly for $ev_{-\infty}$.
 %   \item $\widehat{\mathfrak{S}}$ is a collared CF perturbation of $\mathcal{M}(L''; L',  L)$ such that, among other properties, it is transversal to 0 and $ev_{\pm\infty}$ are strongly submersive.
    
\end{itemize}

It follows from \cite[Lemma 8.21]{Fukaya-corr} that $(CF(L''; L', L), \{n^{E_0, \epsilon}_{k'',k',k}\}) $ defines a filtered unital $A_\infty$ tri-module structure modulo  $T^{E_0}$. After an algebraic argument involving pseudo-isotopy between $A_\infty$ tri-modules modulo various $E_0$ and taking their limits as in the last step of proof of \cite[Theorem 5.25]{Fukaya-corr}, one obtains a filtered unital $A_\infty$ tri-module $(CF(L''; L', L), \{n_{k'',k',k}\}_{k'', k', k\geq 0})$. The gapping monoid $\mathbb{G}_{L'', L', L}\subseteq \mathbb{R}_{\geq 0}$ is the submonoid generated by $G^0_{L''}, G^0_{L'},$ \\
$ G^0_{L}$ and 
$G^0_{L'', L', L} \coloneqq \{E \in \mathbb{R}_{\geq 0}|  \mathcal{M}_{k'',k',k}(L''; L',  L; E) \neq \phi \},$ 
which is discrete by Gromov's compactness theorem.

Moreover, in the case when $L'' = L \circ L'$ is an embedded composition, $CF(L \circ L'; L', L)$ admits a canonical cyclic element as follows:

\begin{theorem} \cite[Lemma 8.10]{Fukaya-corr}\label{embcorrtrimod} In the context of Theorem \ref{corrtrimod}, assume $L'' = L \circ L'$ is an embedded composition, then the constant one function $const_1: L \circ L'\cong ((L \circ L')\times L' \times L)\cap \Delta \rightarrow \mathbb{R}$ induces a left cyclic element $\mathbf{1} \in CF^0(L \circ L'; L', L)$. 
\end{theorem}

\begin{remark}
    More precisely, all the Lagrangian correspondences are endowed with their relative spin structures, and (2) is an equality of such. We refer the reader to the original paper \cite{Fukaya-corr} for the precise treatment.
\end{remark}

\begin{corollary} \cite[Theorem 8.2]{Fukaya-corr} \label{compstrict}
    If $L'' = L \circ L'$, then there exists a map 
\begin{equation} \label{compdegone}    CF^{odd}_+(L) \times CF^{odd}_+(L') \xrightarrow{\circ} CF^{odd}_+(L''),   \end{equation}  
    $$(b,b') \mapsto b'' \coloneqq b \circ b',$$
    characterised by the equation $n^{b'',b',b}_{0,0,0}(\mathbf{1})=0$. \\
    
    Moreover, (\ref{compdegone})   restricts to a map between their (strict) Maurer-Cartan sets
    \begin{equation} \label{compstricteqn}
    \widehat{MC}(L) \times \widehat{MC}(L') \xrightarrow{\circ} \widehat{MC}(L''),
    \end{equation} 
which respects gauge equivalence relation. Therefore, (\ref{compstricteqn}) further descends to a map between their (strict) Maurer-Cartan spaces 
    \begin{equation} 
    MC(L) \times MC(L') \xrightarrow{\circ} MC(L'').
    \end{equation}
\end{corollary}

It is natural to ask whether (\ref{compdegone}) also descends to weak Maurer-Cartan sets, and how their disk potentials are related. Applying Proposition \ref{algcompweak} to the correspondence tri-module $CF(L''; L', L)$, we obtain the following:

\begin{corollary} \label{weakunobcomp}
    (\ref{compdegone}) descends to a map between weak Maurer-Cartan sets
    \begin{equation} \label{compweakeqn} \widehat{MC}_{weak}(L) \times \widehat{MC}_{weak}(L') \xrightarrow{\circ} \widehat{MC}_{weak}(L''),
    \end{equation}
    $$(b,b') \mapsto b'' \coloneqq b \circ b',$$
    in which their disk potentials satisfy
    \begin{equation} \label{potentialeqn}
    W_L(b) + W_{L'}(b') = W_{L''}(b'').
     \end{equation}

     Moreover, (\ref{compweakeqn}) descends to a map between their weak Maurer-Cartan spaces 
    \begin{equation} 
    MC_{weak}(L) \times MC_{weak}(L') \xrightarrow{\circ} MC_{weak}(L'').
    \end{equation}
\end{corollary}

 In particular, (\ref{compweakeqn}) restricts to (\ref{compstricteqn}) for $b,b'$ with $W_L(b) = 0 = W_{L'}(b')$.

\section{Equivariant de Rham model} \label{sec:equiv}

In this section, we construct an equivariant extension of the Floer complex $CF_G(L)$, called the equivariant de Rham model.  In \cite{KLZ}, based on classical Borel construction, Kim, the first-named author and Zheng constructed the equivariant Floer theory and the disc potential of $L_G \subset Y_G$ for a symplectic $G$-action on $Y$ and a $G$-invariant Lagrangian $L$.  When the $G$-action is Hamiltonian, $Y_G$ can be taken as a symplectic quotient of $Y \times T^*EG$.  Cazassus \cite{Cazassus} studied equivariant Floer homology in this case later.  

The first four subsections study Borel spaces and their symplectic analogues: in Subsection \ref{classifyingspace}, we recall the notion of classifying spaces; in Subsection \ref{symplborelspace} (resp. \ref{symplborelspace}), we review the notions of symplectic (resp. Lagrangian) Borel spaces;  in Subsection \ref{subsec-lagcorrborelspace}, we generalise the Borel construction to Lagrangian correspondences.

The latter four subsections investigate equivariant Lagrangian Floer theory in details: in Subsection \ref{subsec-equivfloercpx}, we define $CF_G(L)$ as a canonical model of an inverse limit of de Rham models as $A_\infty$ algebras, whose algebraic counterparts are developed in Subsection \ref{subsec-invlimit} and \ref{subsec-hptinvlimit}. Lastly, in Subsection \ref{sectionequivmorsemodel}, we will also recall the original construction of equivariant Morse model $CF^{Morse}_{G}(L)$ in \cite{KLZ}.

\subsection{Classifying spaces}
\label{classifyingspace}
%and their finite-dimensional approximations%
In this subsection, we revisit the notion of classifying spaces of  a compact Lie group $G$  along with their approximation spaces.

Let $EG$ be the universal principal $G$-bundle over the classifying  space $BG$. Formally, it induces a  Hamiltonian space $((T^*EG, \omega_{can}), G, \mu_{G})$ for which $G$ acts freely on $\mu_{G}^{-1}(0)$ with symplectic quotient $$(T^*EG \sslash G = \mu_{G}^{-1}(0)/G, \omega_{red}) \cong (T^*BG,\omega_{can}).$$
 Moreover, $EG$ embedds as the zero section $0_{EG} \subseteq T^*EG$, which is a $G$-invariant Lagrangian lying inside $\mu_{G}^{-1}(0)$ with quotient $\overline{0}_{EG} \coloneqq 0_{EG}/G \cong 0_{BG} \subseteq T^*BG$.

   In practice, we approximate $EG$ and $BG$ by finite dimensional smooth closed manifolds (see e.g. \cite[Appendix A.10]{tuequiv}):
%   $$G \curvearrowright G = EG_0 \subseteq EG_1 \subseteq \cdots \subseteq EG_l \subseteq \cdots$$
%    $$pt = BG_0 \subseteq BG_1 \subseteq \cdots \subseteq BG_l \subseteq \cdots$$
\begin{equation} \label{egemb}
\begin{tikzcd} 
  G = EG_0 \arrow[r, hook] \arrow[d]
    & EG_1 \arrow[r, hook] \arrow[d]
    &\cdots
    &\arrow[r, hook]
    &EG_l \arrow[r, hook] \arrow[d]
    &\cdots\\
pt = BG_0 \arrow[r, hook] 
    & BG_1 \arrow[r, hook] 
    &\cdots
    &\arrow[r, hook]
    &BG_l \arrow[r, hook] &\cdots 
    \end{tikzcd}
\end{equation}
For each $l$,  $EG_l$ is an $(l-1)$-connected principal $G$-bundle over $BG_l$ satisfying
\begin{equation} \label{egbgexactsq}
	EG_{l-1} \cong EG_{l} \times_{BG_{l} } BG_{l-1}.
\end{equation}

 Similarly, it induces a Hamiltonian space $((T^*EG_l, \omega_{can}), G, \mu_{l})$ with symplectic quotient $(T^*BG_l,\omega_{can})$. Similarly, $EG_l$ embeds as a $G$-invariant Lagrangian $0_{EG_l} \subseteq \mu_{l}^{-1}(0) \subseteq T^*EG_l$ with quotient $\overline{0}_{EG_l} \cong 0_{BG_l} \subseteq T^*BG_l$.\\

    For future purpose, for each $l\in \mathbb{Z}_{\geq 0}$, we choose a $G$-invariant metric on $EG_l$, inducing the quotient metric  on $BG_l$, such that embeddings in (\ref{egemb}) are isometric embeddings. These metrics lift to $\omega_{can}$-compatible metrics (known as Sasaki metrics) on $T^*EG_l$ and $T^*BG_l$, inducing almost K\"{a}hler structures $(T^*EG_l, \omega_{can}, g_{E_l}, J_{E_l})$ and $(T^*BG_l, \omega_{can}, g_{B_l}, J_{B_l})$ with the following canonical isomorphisms:
    $$(T^*EG_l \sslash G = \mu_{l}^{-1}(0)/G, \omega_{red}, {g}_{red}, {J}_{red} ) \cong (T^*BG_l,\omega_{can}, g_{B_l}, J_{B_l}).$$
    
    Observe that the metrics split $T^*EG_l$ and $T^*BG_l$ compatibly, inducing the following sequences of almost K\"{a}hler embeddings
    %$$T^*G = T^*EG_0 \subseteq T^*EG_1 \subseteq \cdots \subseteq T^*EG_l \subseteq \cdots$$
    %$$pt = T^*BG_0 \subseteq T^*BG_1 \subseteq \cdots \subseteq T^*BG_l \subseteq \cdots$$
\begin{equation} \label{tegemb}
    \begin{tikzcd}
  T^*G = T^*EG_0 \arrow[r, hook] \arrow[d]
    & T^*EG_1 \arrow[r, hook] \arrow[d]
    &\cdots
    &\arrow[r, hook]
    &T^*EG_l \arrow[r, hook] \arrow[d]
    &\cdots\\
pt = T^*BG_0 \arrow[r, hook] 
    & T^*BG_1 \arrow[r, hook] 
    &\cdots
    &\arrow[r, hook]
    &T^*BG_l \arrow[r, hook] &\cdots 
    \end{tikzcd}
\end{equation}

Note that $0_{EG_l}$ and $0_{BG_l}$ are compatible with the embeddings in (\ref{tegemb}). 

\subsection{Symplectic Borel spaces}
\label{symplborelspace}

%The following setup will be applied frequently:

%\begin{setup} \label{ham}
%Let $(Y, \omega_Y)$ be a Hamiltonian $G$-manifold with a proper moment map $\mu_Y: Y \rightarrow \mathfrak{g}^*$. 
%\end{setup}
In this subsection, we recall the notion of a symplectic Borel space associated to a Hamiltonian space $((Y, \omega_Y), G, \mu_Y)$.

Formally we consider the diagonal Hamiltonian $G$ action on $(Y \times T^*EG, \omega_Y \oplus \omega_G)$ with moment map $\mu_Y + \mu_G$. Since $G$ acts freely on $(\mu_Y + \mu_G)^{-1}(0)$, it has a free symplectic quotient called the symplectic Borel space, i.e.
\begin{equation} \label{yg}
(Y_G \coloneqq (Y \times T^*EG)\sslash G = (\mu_Y + \mu_G)^{-1}(0)/G, \omega_G \coloneqq \omega_{red}).
\end{equation}
Again in practice, we will approximate $(Y_G, \omega_G)$ using $EG_l$. Namely, for each $l\in \mathbb{Z}_{\geq 0}$, we replace $(T^*EG, \omega_{can})$ above by $(T^*EG_l, \omega_{can})$ and define 
\begin{equation} \label{yn}
(Y_l \coloneqq (Y \times T^*EG_l)\sslash G = (\mu_Y + \mu_l)^{-1}(0)/G, \omega_l \coloneqq \omega_{red}).
\end{equation}
Note that (\ref{tegemb}) induces a sequence of $G$-equivariant symplectic embeddings among $Y \times T^*EG_l$ preserving the moment maps $\mu_Y + \mu_l$, and hence gives rise to the following sequence of symplectic embeddings
\begin{equation} \label{ygemb}
\begin{tikzcd}
  &Y = Y_0 \arrow[r, hook] 
     &Y_1 \arrow[r, hook] 
    &\cdots
    &\arrow[r, hook]
    &Y_l \arrow[r, hook] 
    &\cdots
    \end{tikzcd}
\end{equation}
Recall from \cite[Proposition 4.7]{Cazassus} that for each $l\in\mathbb{Z}_{\geq 0}$, after choosing metrics $g_{E_l}$ as above, there is a canonical symplectic fibration
\begin{equation} \label{ynbgnfib}
(Y,\omega) \xrightarrow{\iota_l} (Y_l, \omega_l) \xrightarrow{\pi_l} (T^*BG_l,\omega_{can}) 
\end{equation}
such that the following commutative diagrams of fibrations with fiber $Y$ hold:
\begin{equation} 
    \begin{tikzcd}
  Y = Y_0 \arrow[r, hook] \arrow[d, "\pi_0"]
    & Y_1 \arrow[r, hook] \arrow[d, "\pi_1"]
    &\cdots
    &\arrow[r, hook]
    &Y_l \arrow[r, hook] \arrow[d, "\pi_l"]
    &\cdots\\
pt = T^*BG_0 \arrow[r, hook] 
    & T^*BG_1 \arrow[r, hook] 
    &\cdots
    &\arrow[r, hook]
    &T^*BG_l \arrow[r, hook] &\cdots 
    \end{tikzcd}
\end{equation}

Moreover, each $G$-invariant almost K\"{a}hler structure $(Y, \omega_Y, J_Y, g_Y)$ on $Y$ induces such structure $(Y \times T^*EG_l , \omega_Y \oplus \omega_{can}, J_Y\oplus J_{E_l}, g_Y\oplus g_{E_l})$ on $Y \times T^*EG_l$, which descends to an almost K\"{a}hler structure $(Y_l, \omega_l, g_l, J_l)$ on $Y_l$ such that both $\iota_l$ and $\pi_l$ are $J$-holomorphic, and (\ref{ygemb}) is a sequence of almost K\"{a}hler embeddings.
\subsection{Lagrangian Borel spaces}
\label{lagborelspace}
%We first consider Lagrangians which are $G$-invariant.
In this subsection, we recall the notion of Lagrangian Borel space associated to a $G$-invariant Lagrangian $L \subseteq Y$.

By the Hamiltonian equations of $\mu_Y$, a connected Lagrangian $L \subseteq Y$ is $G$-invariant if and only if $L$ lies in $\mu_Y^{-1}(c)$ for some (unique) $c \in\mathfrak{g}^*$. In such a case, the $G$-equivariance of $\mu_Y$ implies that $c$ is a central element.  Without loss of generality, we assume $c=0$ (by shifting $\mu_Y$ by $c$).  $L$ gives rise to a $G$-invariant Lagrangian $L \times 0_{EG} \subseteq Y \times T^*EG$. Its reduction $$L_G \coloneqq L \times_G 0_{EG} \subseteq (Y \times T^*EG) \sslash G = Y_G$$
is called the Lagrangian Borel space.
%This was called a $G$-Lagrangian in \cite{Cazassus}.

We have a finite dimensional approximation of $L_G$.  Namely, for each $l \in \mathbb{Z}_{\geq 0}$, consider $G$-Lagrangian $L \times 0_{EG_l} \subseteq Y \times T^*EG_l$ and its reduction $L_l \coloneqq L \times_G 0_{EG_l} \subseteq (Y \times T^*EG_l) \sslash G = Y_l$. This gives rise to a sequence of Lagrangians $\{L_l\}$ approximating $L_G$, compatible under the embedding (\ref{ygemb}), i.e.
\begin{equation} \label{yglgemb}
\begin{tikzcd}
  &Y = Y_0 \arrow[r, hook]
     &Y_1 \arrow[r, hook]
    &\cdots
    &\arrow[r, hook]
    &Y_l \arrow[r, hook] 
    &\cdots\\
    &L = L_0 \arrow[r, hook] \arrow[u,hook] 
    & L_1 \arrow[r, hook] \arrow[u,hook]  
    &\cdots
    &\arrow[r, hook]
    &L_l \arrow[r, hook] \arrow[u,hook] &\cdots
    \end{tikzcd}
\end{equation}

Besides, observe that under the symplectic fibration (\ref{ynbgnfib}), $L_l$ is a fibered Lagrangian over the base Lagrangian $ 0_{BG_l}$ with the fiber Lagrangian $L$, i.e.

\begin{equation} \label{lnfib}
\begin{tikzcd}
&(Y,\omega) \arrow{r}{\iota _l} &(Y_l, \omega_l) \arrow{r}{\pi _l} &(T^*BG_l,\omega_{can})\\
&L \arrow[u,hook] \arrow[r] &L_l \arrow[u,hook] \arrow[r] &0_{BG_l}\arrow[u,hook] 
\end{tikzcd}
\end{equation}

Combining  (\ref{lgbgemb}) and (\ref{ynbgnfib}), we obtain fiber bundles with fiber $L$:
\begin{equation} 
\label{lgbgemb}
	\begin{tikzcd}
		&L = L_0 \arrow[r, hook] \arrow[d]
		& L_1 \arrow[r, hook] \arrow[d]  
		&\cdots
		&\arrow[r, hook]
		&L_l \arrow[r, hook] \arrow[d] 
		&\cdots\\
		&pt = 0_{BG_0} \arrow[r, hook]
		& 0_{BG_1} \arrow[r, hook]
		&\cdots
		&\arrow[r, hook]
		&0_{BG_l}\arrow[r, hook] 
		&\cdots
	\end{tikzcd}
\end{equation}

Moreover, it follows from (\ref{egbgexactsq}) that  for each $l \in \mathbb{Z}_{>0}$,
\begin{equation} \label{lgexactsq}
	L_{l-1} \cong L_{l} \times_{0_{BG_{l}} } 0_{BG_{l-1}}.
\end{equation}

Furthermore, we construct relative spin structures on $L_l$ from a $G$-equivariant relative spin structure on $L$ defined as follows:

\begin{defn} \label{eqspin}
    A $V$-relative spin structure on $L$ is $G$-equivariant if

    \begin{enumerate}
        \item $V$ is a $G$-equivariant vector bundle, i.e. $G$-action on $Y$ lifts equivariantly to $V$ by linear isomorphisms.
        \item The spin structure of $V|_{L}\oplus TL$ is $G$-equivariant, i.e. $G$-action on the oriented orthogonal frame bundle $P_{SO}(V|_{L}\oplus TL)$ lifts equivariantly to its double cover $P_{Spin}(V|_{L}\oplus TL)$.
    \end{enumerate}
\end{defn}
Given a $G$-equivariant $V$-relative spin structure on $L$, for each $l$, it induces a $G$-equivariant $V\oplus p_l^*T(EG_l)$ \footnote{where $p_l :T^*(EG_l) \rightarrow EG_l$ is the cotangent bundle projection.}-relative spin structure on $L \times 0_{EG_l}$, which descends to a $V_l$-relative spin structure on $\bar{L}$, where $V_l\rightarrow Y_l$ is defined as
$$V_l \coloneqq ((V\oplus p_l^*T(EG_l))|_{(\mu_Y + \mu_l)^{-1}(0)})/G.$$

\subsection{Lagrangian correspondence Borel spaces}
\label{subsec-lagcorrborelspace}

In this subsection, we generalise the Lagrangian Borel construction to Lagrangian correspondences.

%\subsubsection{Lagrangian correspondence Borel spaces}

%\begin{setup} \label{hamspacelagcorr}
%Given compact Lie groups $G,H$ with Hamiltonian actions $G$ (resp. $H$) on $(M, \omega_M)$ (resp. $(N, \omega_N)$) with moment map $\mu_M$ (resp. $\mu_N$), inducing Hamiltonian $G\times H$ action on $(M^- \times N, -\omega_M \oplus \omega_N))$ with moment map $(-\mu_M, \mu_N)$).
% Fix central elements $c_M\in \mathfrak{g}^*, c_N\in \mathfrak{h}^*$ such that $G$ (resp. $H$) acts freely on $\mu_{M}^{-1}(c_M)$ (resp. $\mu_{N}^{-1}(c_N)$) with symplectic quotient $\overline{M} \coloneqq M\sslash_{c_M} G$ (resp. $\overline{N} \coloneqq N\sslash_{c_N} H$).
%\end{setup}

Given two Hamiltonian spaces $((M, \omega_M), G, \mu_M)$ and $((N, \omega_N), H, \mu_N)$, formally we consider the Hamiltonian space $$((M^- \times N) \times (T^*EG^- \times T^*EH), G \times H, (-\mu_M + \mu_{N},-\mu_G + \mu_{H}))$$
which is canonically isomorphic to the following Hamiltonian space 
$$((M \times T^*EG)^- \times (N \times T^*EH), G \times H, \mu_{G\times H} \coloneqq (-(\mu_M + \mu_{G}), \mu_N + \mu_{H})).$$ 
Then $G \times H$ acts freely on $\mu_{G\times H}^{-1}(0, 0)$ with symplectic quotient $M_G^- \times N_H$. 

%\begin{setup} \label{equilagcorr}
%Under the setup \ref{ham}, for any $S$-equivariant Lagrangian correspondences $M \xrightarrow{L} N$, it follows that $\mu_S(L)\in Lie(S)^*$ is a central constant, where $\mu_S$ is the moment map for induced Hamiltionian $S$-action on $M^- \times N$. We assume that $\mu_S(L)=0$ and $S$ is admissible.
%\end{setup}

%For simplicity, we will restrict our attention to $S$-equivariant Lagrangian correspondences $M \xrightarrow{L} N$, where $S \subseteq G \times H$ is of one of the following forms: 

%\begin{enumerate}
%\item $S= \Gamma_\varphi $ is the graph of a Lie group morphism $G \xrightarrow{\varphi} H$. 
%\item $S= \Gamma_\psi $ is the graph of a Lie group morphism %$H \xrightarrow{\psi} G$. 
%\item $S = G \times H$.
%\end{enumerate}

%which would suffice for our purposes.\\

\begin{defn}
    Given Hamiltonian spaces $(M, G, \mu_M)$ and $(N, H, \mu_N)$, a $(G\times H)$-Lagrangian correspondence $M \xrightarrow{L} N$ is a $(G\times H)$-Lagrangian $$L \subseteq ((M^- \times N, -\omega_M \oplus \omega_N), G\times H, (-\mu_M, \mu_N)).$$
\end{defn}
%\begin{setup}
%\label{hamlagcorr}
%    Under the setup \ref{hamspacelagcorr}, consider a Lagrangian correspondence $M \xrightarrow{L} N$ for which as a Lagrangian $L$ in the Hamiltonian $(G \times H)$ space $M^- \times N$ is under the setup \ref{hamlag}.
%\end{setup}
Under such a setup, formally $L$ gives rise to a $(G \times H)$ Lagrangian $L \times (0_{EG} \times 0_{EH}) \subseteq (M^- \times N) \times ((T^*EG)^- \times T^*EH)$, which corresponds to another $(G \times H)$-Lagrangian in $(M \times T^*EG)^- \times (N \times T^*EH)$. Its reduction $M_G \xrightarrow{L_{G\times H}} N_H$ is called the Lagrangian correspondence Borel space.

As before, we replace $EG, EH$ by $EG_l, EH_l$ to obtain its finite-dimensional approximation $M_{l} \xrightarrow{L_{l}} N_{l}$ as a sequence of Lagrangian correspondences.

Moreover, a $G\times H$-equivariant $pr^*_M(V_M \oplus TM) \oplus pr^*_NV_N$-relative spin structure of $L$ induces a $pr^*_{M_l}(V_{M_l} \oplus TM_l) \oplus pr^*_{N_l} V_{N_l}$-relative spin structure of $L_l$ for each $l$, similarly as in Subsection \ref{lagborelspace}.

The following proposition asserts that the Borel construction of Lagrangian correspondences is compatible with their geometric compositions.
%under a mild assumption that is always satisfied for most of our intended applications.

\begin{prop} \label{equilagcorrcomp}
Given Hamiltonian spaces $(P, K, \mu_P)$, $(M, G, \mu_M)$ and $(N, H, \mu_N)$. For any  $(K\times G)$-Lagrangian correspondence $P \xrightarrow{L'} M$ and $(G\times H)$-Lagrangian correspondence $M \xrightarrow{L} N$ which are composable, then 
\begin{enumerate}
    \item Their composition $P \xrightarrow{L \circ L'} N$ is a $(K\times H)$-Lagrangian correspondence.
    \item For each $l\in \mathbb{Z}_{\geq 0}$, we have $L_l \circ L'_l = (L\circ L')_{l} \subseteq P^-_l \times N_l$.
\end{enumerate} 
\end{prop}

\begin{proof}
    (1) follows directly from definition; for (2),
we first show that $L_l \circ L'_l \subseteq (L\circ L')_{l}$: for any $([p, a], [n, c])\in L_l \circ L'_l$, there exists $[m, b] = [m', b']\in M_l$ such that 
$$((p, m') , (a, b')) \in L' \times (0_{EK_l} \times 0_{EG_l}) ; ((m, n) , (b, c)) \in L \times (0_{EG_l} \times 0_{EH_l}). $$
Note that from $[m, b] = [m', b']$, there exists $g\in G$ such that $g \cdot (m', b') = (m, b)$, therefore, $(1, g)\cdot ((p, m') , (a, b')) = ((p, m) , (a, b)) \in L' \times (0_{EK_l} \times 0_{EG_l})$. Hence, $(p, n) \in L \circ L'$ and $(a,c) \in 0_{EK_l} \times 0_{EG_l}$, so $([p, a], [n, c])\in(L\circ L')_{l}$.

Conversely, given $([p, a], [n, c])\in (L\circ L')_{l}$, there exists $m \in M$ such that $(p, m)\in L'$ and $(m, n) \in L$. Choose any $b\in 0_{EG_l}$, we have 
$$((p, m) , (a, b)) \in L' \times (0_{EK_l} \times 0_{EG_l}) ; ((m, n) , (b, c)) \in L \times (0_{EG_l} \times 0_{EH_l}). $$
Therefore, $([p, a], [n, c])\in L_l \circ L'_l$.
\end{proof}

\begin{remark}
    Note that we do not assume $L_l \circ L'_l$ is a clean composition, but still it is a Lagrangian by the equality $L_l \circ L'_l = (L\circ L')_{l}$. Later in application, we will assume that they are cleanly composable in a compatible way in Definition \ref{cleancompseqdef}.
\end{remark}

%\subsection{Inverse limits and Homological Perturbation Theory}
%In this subsection, we discuss the notion
%of inverse limit in the theory of $A_\infty$ algebras and tri-modules; we also study the homological perturbation theory of them. Finally we study how these two theories interact.

\subsection{Inverse Limit}
\label{subsec-invlimit}
In this section, 
we first recall several notions in the theory of (classical) inverse limit. Then
we study the inverse limit of a tower of $A_\infty$ algebras (resp. tri-modules) related by strict $A_\infty$ morphisms.

\subsubsection{Classical inverse limits}
For our purpose, we will only consider the inverse limit associated to (countable) towers of ($\mathbb{Z}_{\geq 0}$-)graded abelian groups (or  objects with further structures later) , i.e. sequences of the form 
$$ (C_l) = (C_0 \xleftarrow{g_1} C_1\xleftarrow{g_2} \dots \leftarrow C_{l-1} \xleftarrow{g_l} C_{l}\leftarrow  \dots)$$
where $C_l = \displaystyle{\bigoplus_{m \geq 0}}\, C^m_l$ are graded abelian groups and $g_l$ are group morphisms of degree 0. Note that this  induces towers of abelian groups degree-wise: for each $m\in \mathbb{Z}_{\geq 0}$, there exists a tower of abelian groups
$$ (C^m_l) = (C^m_0 \xleftarrow{g_1} C^m_1\xleftarrow{g_2} \dots \leftarrow C^m_{l-1} \xleftarrow{g_l} C^m_{l}\leftarrow  \dots).$$

\begin{defn}
    Given a tower of graded abelian groups $(C_l)_{l\in\mathbb{Z}_{\geq 0}}$, its inverse limit $\varprojlim C_l$ is defined as a graded abelian group by 
    $$\varprojlim  C_l \coloneqq \bigoplus_{m \geq 0} \, (\varprojlim  C^m_l) = \{x = \sum_{m=1}^{N} x^{(m)}|\,  %x^{(m)} = (x^{(m)}_l)\in \varprojlim  C^m_l \Leftrightarrow% 
    g_l(x^{(m)}_l) = x^{(m)}_{l-1}, \forall l\in\mathbb{Z}_{>0}\}$$
    endowed with entrywise addition and additive unit.\\
    The projections $\{\pi_l: \varprojlim C_l \rightarrow C_l\}_{l \in \mathbb{Z}_{\geq 0}}$ are the natural projections to $l$-th entry.
\end{defn}

\begin{remark}
    Observe that $\varprojlim  C_l$ as a graded abelian group above is an abelian subgroup of $\varprojlim  C_l$ as an (ungraded) abelian group, so for simplicity, we will denote an element of the former by $x = (x_l)$, where $x_l = \displaystyle{\sum_{m=1}^{N}} x^{(m)}_l$.
    
    Nevertheless, note that they are not the same in general. For example, consider the sequence of truncated polynomial rings $C_l = R[x]/(x^{l+1})$ with $\deg x = 2$ , then as a graded abelian group $\varprojlim  C_l = R[x]$ is the polynomial ring, while  as an ungraded abelian group $\varprojlim  C_l = R \llbracket x \rrbracket$ is the formal power series ring.
\end{remark}

\begin{remark}
    The definitions of inverse limits of graded modules (resp. graded algebras) over graded ring and their (co)chain complexes are defined similarly.
\end{remark}
We consider the inverse limits of modules over inverse limits of rings below:
\begin{prop} \label{invlimmod}
    Given towers of graded rings $(C_l)$ and abelian groups $(D_l)$
    $$ (C_l) = (C_0 \xleftarrow{g_1} C_1\xleftarrow{g_2} \dots \leftarrow C_{l-1} \xleftarrow{g_l} C_{l}\leftarrow  \dots),$$
    $$ (D_l) = (D_0 \xleftarrow{f_1} D_1\xleftarrow{f_2} \dots \leftarrow D_{l-1} \xleftarrow{f_l} D_l\leftarrow  \dots),$$
    assume that $(D_l)$ is a graded left $(C_l)$-module, in the sense that  for each $l\in\mathbb{Z}_{\geq 0}$, $D_l$ is a graded left $C_l$ module with module structure $\lambda_l: C_l \times D_l \rightarrow D_l$ such that for any $l>0$, $f_l$ is a module morphism of degree 0 over $g_l$, i.e. for any $x_l \in C_l, y_l \in D_l$, 
    $$f_l(\lambda_l(x_l, y_l)) = \lambda_{l-1}(g_l(x_l), f_l(y_l)).$$
    Then $\varprojlim D_l$ has a natural graded left $\varprojlim C_l$-module structure $$\lambda_\infty: \varprojlim C_l \times \varprojlim D_l \rightarrow \varprojlim D_l$$
    defined by the entrywise module structure, i.e. $(\lambda_\infty(x, y)\coloneqq (\lambda_l(x_l, y_l))$.\\
    Moreover, $\pi^D_l: \varprojlim D_l \rightarrow D_l$ are degree 0 module morphisms over $\pi_l: \varprojlim C_l \rightarrow C_l$.
\end{prop}
The proof follows directly from definitions, and therefore is omitted.
\begin{remark}
    Similar statements hold for right, bi- and tri-modules.
\end{remark}

\begin{remark}
    The same statement holds if for all $l$, $C_l, D_l$ are
\begin{itemize}
    \item graded $R$-algebras and $R$-modules such that $(D_l)$ is a graded $(C_l)$-module.
    \item dg algebras and dg modules such that $(D_l)$ is a left dg module over $(C_l)$.
\end{itemize}
\end{remark}

\begin{example}
    Given towers of graded rings $(C_l)$ and $(D_l)$, assume that there is a tower of ring morphisms of degree 0 from $(C_l)$ to $(D_l)$, i.e. a sequence of ring morphisms $(\varphi_l: C_l \rightarrow D_l)_{l\in\mathbb{Z}_{\geq 0}}$ of degree 0 such that for each $l>0$, $\varphi_{l-1} \circ g_l = f_l \circ \varphi_l$. Then $(D_l)$ is a graded left $(C_l)$-module via $(\varphi_l)$. In this case, the left $\varprojlim C_l$-module structure on $\varprojlim D_l$ obtained from (\ref{invlimmod}) is the same as the one induced from the ring morphism between inverse limits $(\varphi_l):\varprojlim C_l \rightarrow \varprojlim D_l$.
\end{example}

\begin{example}
    Given a tower of dg modules $(D_l)$ as a left dg module over a tower of dg algebras $(C_l)$, inducing \begin{itemize}
        \item $\varprojlim D_l$ as a left dg module over $\varprojlim C_l$ by Proposition \ref{invlimmod}. After taking cohomologies, $H^\bullet(\varprojlim D_l)$ is a graded left $H^\bullet(\varprojlim C_l)$-module.
        \item A tower of graded $R$-modules $(H^\bullet(D_l))$ as a graded left module over a tower of graded $R$-algebras $(H^\bullet(C_l))$. Therefore,  $\varprojlim H^\bullet(D_l)$ is a  graded left $\varprojlim H^\bullet(C_l)$-module by Proposition \ref{invlimmod}.
     \end{itemize}
     The projection maps $\pi^D_l$  induces   $([\pi^D_l]): H^\bullet(\varprojlim D_l) \rightarrow \varprojlim H^\bullet(D_l)$ as a module morphism of degree 0 over $([\pi_l]): H^\bullet(\varprojlim C_l) \rightarrow \varprojlim H^\bullet(C_l)$.
\end{example}

We also recall the following statement comparing the inverse limit of cohomologies with the cohomology of inverse limit:

\begin{prop} \cite[Variant of Theorem 3.5.8]{weibel} \label{invlimcohom}
    Given a tower of cochain complexes of $R$-modules $(C_l)$, inducing a tower of graded $R$-modules $(H(C_l))$. Assume that $(C_l)$ satisfies the Mittag-Leffler condition (see e.g. \cite[Definition 3.5.6]{weibel}).  Then for each $m\in\mathbb{Z}$, there is a short exact sequence of $R$-modules
\begin{equation*}
    \begin{tikzcd}
  0 \arrow[r] & \varprojlim^1 H^{m-1}(C_l)  \arrow[r] &  H^{m}(\varprojlim C_l)  \arrow[r, "(\mathrm{[}\pi_l\mathrm{]})"] & \varprojlim H^{m}(C_l)  \arrow[r] & 0 
\end{tikzcd}
\end{equation*}
where  $\varprojlim^1 \cong R^1 \varprojlim$ is the first derived functor of $\varprojlim$ (see e.g. \cite[Corollary 3.5.4]{weibel}).
In particular, if $\varprojlim^1 H^{m-1}(C_l) = 0$, then $H^{m}(\varprojlim C_l) \xrightarrow{([\pi_l])} \varprojlim H^{m}(C_l)$ is an isomorphism, i.e. ``Taking inverse limit commutes with taking cohomology".
    
\end{prop}

\begin{remark}
    Similar statements hold in the following setting:
    \begin{itemize}
        \item $(C_l)$ is a tower of dg algebras.
        \item $(D_l)$ is a tower of dg modules over a tower of dg algebras $(C_l)$.
    \end{itemize}
%for which the notion of morphism in Proposition \ref{invlimcohom} will be of the corresponding categories.
\end{remark}

\begin{example}
    Given a sequence of closed manifolds $(S_l)$ with closed embeddings $\iota_l: S_{l-1} \rightarrow S_l$ between them. This induces a tower of de Rham dg algebras $(\Omega(S_l))$ connected by pull-backs $g_l \coloneqq \iota^*_l: \Omega(S_l) \rightarrow \Omega(S_{l-1})$, which are surjective since $\iota_l$ are proper. It follows that $(\Omega(S_l))$ satisfies the Mittag-Leffler condition. Moreover, the cohomological sequence $(H(S_l))$ is a tower of finite dimensional $\mathbb{R}$-vector spaces since $S_l$ are closed. Therefore, by \cite[Exercise 3.5.2]{weibel} $\varprojlim^1 H(S_l) = 0$, and hence $H(\varprojlim \Omega(S_l)) \cong \varprojlim H(S_l)$ by Proposition \ref{invlimcohom}.
\end{example}

Hence, one may ask when $\varprojlim^1 H(C_l) = 0$ occurs. We recall the following notion:

\begin{defn}
    A tower of graded cochain complexes of $R$-modules $(C_l)$ satisfies homological stability if the associated cohomological sequence stabilizes degree-wise: for any $m\in \mathbb{Z}_{\geq 0}$, there exists $l_0 = l_0(m)$ such that $H^m(C_l)$ stabilizes for $l\geq l_0$:
    $$ \dots \xleftarrow{[g_{l_0-1}]} H^m(C_{l_0-1}) \xleftarrow[\sim]{[g_{l_0}]} H^m(C_{l_0+1})\xleftarrow[\sim]{[g_{l_0+1}]} H^m(C_{l_0+1}) \xleftarrow[\sim]{} \dots $$
    In particular, $H^m(C_{l})\xleftarrow{\pi^H_l}\varprojlim H^m(C_{l})$ is an isomorphism for all $l\geq l_0$. 
\end{defn}

%\begin{remark}
%    While ``cohomological stability" might suit better for the above definition, we follow the standard terminology of ``homological stability" which appears in many different contexts. See e.g. \cite{wahl} for a recent survey.
%\end{remark}

\begin{example}
    Given a compact Lie group $G$, recall that we approximate the universal bundle $EG$ by a sequence of closed manifolds $(EG_l)$ as in (\ref{egemb}). For each $l\in\mathbb{Z}_{\geq 0}$, since $EG_l$ is $(l-1)$-connected, $H^{m}(EG_l) = 0$ for all $0<m<l$. This implies $(\Omega(EG_l))$ satisfies homological stability.
    
    Moreover, for any closed $G$-manifold $L$, its Borel space $L_G \coloneqq L \times_{G} EG$ is approximated by a sequence of closed manifolds $(L_l \coloneqq L \times_{G} EG_l)$. It follows that for each $m$,  $H^{m}(L_l)$ stabilizes to $H^{m}(L \times_{G} EG)$ for all $l\geq m$ (see e.g. \cite[Theorem A.7(b)]{tuequiv}). Therefore, $(\Omega(L_l))$ also satisfies homological stability.
\end{example}

\begin{prop} \label{homstabml}
    Given a tower of graded cochain complexes of $R$-modules $(C_l)$, assume that $(\overline{C}_l)$  satisfies homological stability and Mittag-Leffler condition, then $H(\varprojlim C_l) \xrightarrow{([\pi_l])} \varprojlim H(C_l)$ is an isomorphism of graded $R$-modules.
\end{prop} 
\begin{proof}
    Homological stability of $(C_l)$ implies its cohomological sequence $(H(C_l))$  satisfies Mittag-Leffler condition, and hence by \cite[Proposition 3.5.7]{weibel}, \\
    $\varprojlim^1 H(C_l)=0$. 
    Applying Proposition \ref{invlimcohom} to $(C_l)$, which satisfies Mittag-Leffler condition by assumption, yields the desired conclusion.
\end{proof}
\subsubsection{$A_\infty$ inverse limits}

We study the inverse limit of a tower of $A_\infty$ algebras.
\begin{prop} \label{invlimitalg}
 %Given a sequence of $\mathbb{G}$-gapped unital $A_\infty$ algebras $\{(C_l, \{m^{(l)}_k\}, e_l)\}_{l \geq 0}$ with $\mathbb{G}$-gapped unital strict $A_\infty$ morphisms  
 %$$ C_0 \xleftarrow{g_1} C_1\xleftarrow{g_2} \dots \leftarrow C_{l-1} \xleftarrow{g_l} C_{l}\leftarrow  \dots$$
 Given a tower of graded $\Lambda_0$-modules $(C_l)$, assume that for each $l$, $C_l = (C_l, \{m^{(l)}_k\}, e_l)$ is a $\mathbb{G}$-gapped unital $A_\infty$ algebra and $g_l$ is a $\mathbb{G}$-gapped unital strict $A_\infty$ morphism, 
 then the inverse limit $C_\infty \coloneqq \varprojlim  C_l$ 
 admits a $\mathbb{G}$-gapped unital $A_\infty$ structure $(C_\infty, \{m^{(\infty)}_k\}, e_\infty)$ defined by 
 $$m^{(\infty)}_k(x_1,\dots,x_k) \coloneqq (m^{(l)}_k(x^{(l)}_1,\dots,x^{(l)}_k))_{l \geq 0},$$
 where $x_i = (x^{(l)}_i)_{l \geq 0} \in C_\infty$ for $1 \leq i \leq k$. The strict unit is given by $e_\infty \coloneqq (e_l)_{l \geq 0}$.\\
 Moreover, the projection 
 $\pi_l: C_\infty \rightarrow C_l$
 %$$x = (x^{(l)})_{l \geq 0} \mapsto x^{(l)}$$
 is a $\mathbb{G}$-gapped unital strict $A_\infty$ morphism.
\end{prop}
\begin{proof}
    Both $A_\infty$ and unitality relations  follow directly from their entrywise equations. That $\pi_l$ satisfies the stated property also follows from definition.
\end{proof}

The analogous statement for $A_\infty$ tri-modules is as follows:
\begin{prop} \label{invlimittrimod}
 Given three sequences of $\mathbb{G}$-gapped unital $A_\infty$ algebras $$((C_l, \{m^{(l)}_k\}, e_l)), ((C'_l, \{m'^{(l)}_{k'}\}, e'_l)), ((C''_l, \{m''^{(l)}_{k''}\}, e''_l))$$ with $\mathbb{G}$-gapped unital strict $A_\infty$ morphisms  
 $$(\dots \leftarrow C_{l-1} \xleftarrow{g_l} C_{l}\leftarrow  \dots); (\dots  \xleftarrow{g'_l} C'_{l}\leftarrow  \dots); (\dots \xleftarrow{g''_l} C''_{l}\leftarrow  \dots)$$
 with inverse limits $(C_\infty, \{m^{(\infty)}_k\}), (C'_\infty, \{m'^{(\infty)}_{k'}\}), (C''_\infty, \{m''^{(\infty)}_{k''}\})$.\\
 Given also a sequence of $\mathbb{G}$-gapped unital   $A_\infty$ tri-modules $((D_{l}, \{n^{(l)}_{k'', k', k}\})$ ,in which $D_l$ is a unital left $C_l$, right $(C'_l, C''_l)$ $A_\infty$ tri-module, with $\mathbb{G}$-gapped unital strict $A_\infty$ tri-module morphisms along $((g_l), (g'_l), (g''_l))$. 
$$ D_0 \xleftarrow{f_1} D_1\xleftarrow{f_2} \dots \leftarrow D_{l-1} \xleftarrow{f_l} D_l\leftarrow  \dots$$
 Then the inverse limit $D_\infty \coloneqq \varprojlim  D_l$ 
 admits a $\mathbb{G}$-gapped unital left $C_\infty$, right $(C'_\infty, C''_\infty)$ $A_\infty$ tri-module structure $(D_\infty, \{n^{(\infty)}_{k'', k', k}\})$ defined by 
 $$n^{(\infty)}_{k'',k',k}(x''_1, \dots, x''_{k''}; y; x'_1, \dots, x'_{k'}; x_1, \dots, x_{k}) $$
 $$= (n^{(l)}_{k'',k',k}(x''^{(l)}_1, \dots, x''^{(l)}_{k''}; y^{(l)}; x'^{(l)}_1, \dots, x'^{(l)}_{k'}; x^{(l)}_1, \dots, x^{(l)}_{k})),$$
 where $x_i = (x^{(l)}_i) \in C_\infty$, $x'_i = (x'^{(l)}_i) \in C'_\infty$, $x''_i = (x''^{(l)}_i) \in C''_\infty$, $y=(y^{(l)}) \in D_\infty$.\\
 Moreover, the projection 
 $\pi^D_l: D_\infty \rightarrow D_l$
 is a $\mathbb{G}$-gapped unital strict $A_\infty$ tri-module morphism along $(\pi_l, \pi'_l, \pi''_l)$.
\end{prop}

\subsection{Homological Perturbation Theory and Inverse limits}
\label{subsec-hptinvlimit}
In this subsection, we study the homological perturbation theory of the inverse limits of $A_\infty$ algebras and tri-modules.
\subsubsection{$A_\infty$ Algebras}
We first consider the case of $A_\infty$ algebras:

\begin{setup}
    Given a sequence of $\mathbb{G}$-gapped unital $A_\infty$ algebras and strict $A_\infty$ morphisms  
  $C_0 \xleftarrow{g_1} C_1\xleftarrow{g_2} \dots \leftarrow C_{l-1} \xleftarrow{g_l} C_{l}\leftarrow  \dots$ with the inverse limit $(C_\infty, \{m^{(\infty)}_k\}, e_\infty)$. For each $0\leq l \leq \infty$, apply Corollary \ref{canonicalalg} to obtain a strong contraction $(i_l, p_l, h_l)$ and a canonical model $(H^\bullet(C_l) \coloneqq H^\bullet(\overline{C}_l, \overline{m}^{(l)}_{1}; \Lambda_0), \{m^{(l), H}_k\},  [e_l])$. 
  \end{setup}

Notice that the choice of strong contraction for each $l$ (and hence the canonical model) is independent from each other. However, we have the following:

\begin{prop} \label{canalgmor}
For each $l\in\mathbb{Z}_{>0}$, the induced cohomological maps 
$$ H(\overline{C}_{l-1})\xleftarrow{[g_l]}H(\overline{C}_{l})\xleftarrow{[\pi_l]}H(\overline{C}_\infty)$$
are unital algebra morphisms with respect to $(H^\bullet(\overline{C}_l), \overline{m}^{(l),H}_{2}, [e_l])$.
\end{prop}

\begin{proof}
    This follows immediately from Remark \ref{alglowdeg}.
\end{proof}

%\begin{remark}
%    One can check directly that $[g_l] = p_{l-1} \circ g_l \circ i_l$, which makes the appearance of $[g_l]$ above more natural.
%\end{remark}

\begin{corollary}
    The induced cohomological sequence $$H(\overline{C}_{0}) \xleftarrow{[g_1]} H(\overline{C}_{1})\xleftarrow{[g_2]} \dots \leftarrow H(\overline{C}_{l-1}) \xleftarrow{[g_l]} H(\overline{C}_{l})\leftarrow  \dots$$ 
    is a tower of graded $R$-algebras with respect to $(H^\bullet(\overline{C}_l), \overline{m}^{(l),H}_{2}, [e_l])$, inducing the inverse limit $\varprojlim H^\bullet(\overline{C}_l)$ as a graded $R$-algebras. Moreover, The projection maps $\pi_l$  induces a graded algebra morphism $([\pi_l]): H^\bullet(\overline{C}_\infty) \rightarrow \varprojlim H^\bullet(\overline{C}_l)$.
\end{corollary}

\subsubsection{$A_\infty$ tri-modules} We then consider the case of $A_\infty$ tri-modules:

\begin{setup} \label{trimodinvlim}
There are three sequences of gapped unital $A_\infty$ algebra morphisms  
 $$(\dots \leftarrow C_{l-1} \xleftarrow{g_l} C_{l}\leftarrow  \dots); (\dots \xleftarrow{g'_l} C'_{l}\leftarrow  \dots); (\dots \xleftarrow{g''_l} C''_{l}\leftarrow  \dots)$$
with inverse limits $C_\infty, C'_\infty, C''_\infty$ respectively.

 Also, there is a sequence of $\mathbb{G}$-gapped unital strict $A_\infty$ tri-module morphisms 
 $D_0 \xleftarrow{f_1} D_1\xleftarrow{f_2} \dots \leftarrow D_{l-1} \xleftarrow{f_l} D_l\leftarrow  \dots$
 with the inverse limit $D_\infty$.\\
For each $0\leq l \leq \infty$, we apply Corollary \ref{canonicalalg} to $C_l$ to obtain a strong contraction $(i_l, p_l, h_l)$ and a canonical $A_\infty$ algebra $H(C_l)$. Similarly for $C'_l$ and $C''_l$. Also, we apply Corollary \ref{canonicaltrimod} to $D_l$ to obtain a strong contraction $(i^D_l, p^D_l, h^D_l)$ and a canonical $A_\infty$ tri-module $(H^\bullet(\overline{D}_l, \overline{n}^{(l)}_{0, 0, 0}; \Lambda_0), \{n^{(l), H}_{k'', k', k}\})$.
\end{setup}

Again, these canonical $A_\infty$ tri-modules are a priori unrelated to each other. However, we have the following: 

\begin{prop} \label{cantrimodmor}
For each $l\in\mathbb{Z}_{>0}$, the induced cohomological maps 
$$ H(\overline{D}_{l-1})\xleftarrow{[f_l]}H(\overline{D}_{l})\xleftarrow{[\pi^D_l]}H(\overline{D}_\infty)$$
are unital tri-module morphisms w.r.t. $(H^\bullet(\overline{D}_l), \overline{n}^{(l), H}_{1, 0, 0}, \overline{n}^{(l), H}_{0, 1, 0}, \overline{n}^{(l), H}_{0, 0, 1})$ (along the algebra morphisms $([g''_l], [g'_l], [g_l])$ and $([\pi''_{l}],[\pi'_{l}], [\pi_{l}])$ from Proposition \ref{canalgmor}).
\end{prop}
\begin{proof}
    This follows immediately from Remark \ref{trimodlowdeg}.
\end{proof}

\begin{corollary}
    The induced cohomological sequence $$H(\overline{D}_{0}) \xleftarrow{[f_1]} H(\overline{D}_{1})\xleftarrow{[f_2]} \dots \leftarrow H(\overline{D}_{l-1}) \xleftarrow{[f_l]} H(\overline{D}_{l})\leftarrow  \dots$$ 
    is a unital left $(H(\overline{C}''_{l}))$, \\
    right $((H(\overline{C}'_{l})), (H(\overline{C}_{l})))$ tri-module, inducing the inverse limit $\varprojlim H^\bullet(\overline{D}_l)$ as a unital left $\varprojlim H(\overline{C}''_{l})$, right $(\varprojlim H(\overline{C}'_{l}), \varprojlim H(\overline{C}_{l}))$ tri-module.\\
    
    Moreover, the projection maps $\pi^D_l$  induces   $$([\pi^D_l]): H^\bullet( \overline{D}_\infty) \rightarrow \varprojlim H^\bullet(\overline{D}_l)$$ as a tri-module morphism along $(([\pi''_{l}]),([\pi'_{l}]), ([\pi_{l}]))$.
\end{corollary}

In particular, we have the following commutative diagrams:

\begin{corollary}\label{cyclicdiagram}
    Fix an element $\overline{\mathbf{1}}_\infty = (\overline{\mathbf{1}}_l)\in H^0(\overline{D}_\infty)$, then for each $l\in\mathbb{Z}_{>0}$, the following commutative diagrams hold:
    \begin{equation}
	\begin{tikzcd}
		&H^\bullet(\overline{C}''_{l-1}) \arrow[d, swap, "\overline{n}^{(l-1), H}_{1, 0, 0}(-; \overline{\mathbf{1}}_{l-1})"] 
		&H^\bullet(\overline{C}''_{l}) \arrow[l, swap, "\mathrm{[}g''_l\mathrm{]}"]\arrow{d}{\overline{n}^{(l), H}_{1, 0, 0}(-; \overline{\mathbf{1}}_{l})}&\varprojlim H^\bullet(\overline{C}''_{l}) \arrow[l, swap, "\pi^{H''}_l"]\arrow{d}{(\overline{n}^{(l), H}_{1, 0, 0}(-; \overline{\mathbf{1}}_{l}))}&H^\bullet(\overline{C}''_{\infty}) \arrow[l, swap, "(\mathrm{[}\pi''_{l}\mathrm{]})"]\arrow{d}{\overline{n}^{(\infty), H}_{1, 0, 0}(-; \overline{\mathbf{1}}_{\infty})}\\
		&H^\bullet(\overline{D}_{l-1})
     &\arrow[l, swap, "\mathrm{[}f_l\mathrm{]}"] H^\bullet(\overline{D}_{l})&\arrow[l, swap, "\pi^{H_D}_l"] \varprojlim H^\bullet(\overline{D}_{l})&H^\bullet(\overline{D}_{\infty}) \arrow[l, swap, "(\mathrm{[}\pi^D_{l}\mathrm{]})"]
	\end{tikzcd}
\end{equation}

Similarly for $H(\overline{C}'_l)$ and $H(\overline{C}_l)$.
\end{corollary}

\subsubsection{Cyclic Property}
We construct a cyclic element via Corollary \ref{cyclicdiagram}:

\begin{prop} \label{invlimcyclic}
    Under the setup \ref{trimodinvlim}, assume further that the sequences 
    $(\overline{C}''_l)$,  $(\overline{D}_l)$ satisfy homological stability and Mittag-Leffler condition, then for any $\mathbb{G}$-gapped element $\mathbf{1}_\infty = (\mathbf{1}_l)\in H^0(D_\infty)$ satisfying the following property: 
    \begin{itemize}
        \item There exists a sequence of integers $(r(l))_l$, increasing to $+\infty$ as $l\rightarrow +\infty$, such that for each $l$, the following is an isomorphism for all $m\leq r(l)$:
        $$\overline{n}^{(l), H}_{1,0,0}(-; \overline{\mathbf{1}}_l): H^m(\overline{C}''_{l}) \rightarrow H^m(\overline{D}_{l}).$$ 
    \end{itemize}
    Then $\mathbf{1}_\infty$ is left cyclic.
\end{prop}
\begin{proof}
    It suffices to show that for any $m\in\mathbb{Z}$, $\overline{n}^{(\infty), H}_{1,0,0}(-; \overline{\mathbf{1}}_\infty): H^{m}(\overline{C}'') \rightarrow H^{m}(\overline{D})$ is an isomorphism of $R$-modules. Apply Proposition \ref{homstabml} to $(\overline{C}''_l)$ and $(\overline{D}_l)$ imply that $([\pi''_l])$ and  $([\pi^D_l])$ are isomorphisms; Moreover, by homological stability, there exists $l_0$ such that for all $l\geq l_0$, both $\pi^{H''}_l$ and $\pi^{H_D}_l$ are isomorphisms at degree $m$; then choose $l$ sufficiently large such that $m \leq r(l)$, hence $\overline{n}^{(l), H}_{1, 0, 0}$ is an isomorphism at degree $m$. Therefore, the result follows from Corollary \ref{cyclicdiagram} at degree $m$. 
\end{proof}

\subsection{Equivariant de Rham Model}
\label{subsec-equivfloercpx}
In this subsection, we define the equivariant de Rham model $CF_{G}(L)$ of a (closed, connected, $G$-equivariantly $V$-relatively spin) $G$-invariant Lagrangian $L$ in a Hamiltonian space $((Y, \omega_Y), G, \mu_Y)$. 

Formally, we define it as the (canonical model of the) Floer complex of its Borel space $CF(L_G)$; in practice, we first consider the sequence of Floer complexes of its approximation $\{CF(L_{l})\}_{l \geq 0}$ and its inverse limit 
$\varprojlim  CF(L_{l})$.
In order to endow it with a natural $A_\infty$ structure, by Proposition \ref{invlimitalg}, it suffices show that $\{(CF(L_{l}), \{m^{(l)}_k\}, e_l)\}_{l \geq 0}$ forms a sequence of unital $A_\infty$ algebras with (strict) $A_\infty$ algebra morphisms $g_l$:
$$ CF(L_{0}) \xleftarrow{g_1} CF(L_{1})\xleftarrow{g_2} \dots \leftarrow CF(L_{l-1}) \xleftarrow{g_l} CF(L_{l})\leftarrow  \dots$$

This motivates the following proposition:

\begin{prop} \label{invlimitlag}
    For each $l\in \mathbb{Z}_{>0}$, the pullback of the inclusion map $g_l\coloneqq \iota^*_l: CF(L_{l-1}) \leftarrow CF(L_{l})$ is a $\mathbb{G}$-gapped unital strict filtered $A_\infty$ algebra morphism, i.e.
    $$g_l(e_l)=e_{l-1},$$
    $$m^{(l-1)}_{k}(g_l(x_1), \dots, g_l(x_{k})) = g_l(m^{(l)}_{k}(x_1, \dots, x_{k})),$$
    for any $x_1,\dots,x_{k} \in CF(L_{l})$.
\end{prop}

\begin{corollary}
    The inverse limit 
     $\varprojlim CF(L_{l})$
    has a natural $\mathbb{G}$-gapped unital $A_\infty$ algebra structure $(\varprojlim  CF(L_{l}), \{m^{(\infty)}_k\}, e)$.
\end{corollary}
\begin{defn}
\label{equivfloercpx}
    Equivariant de Rham model $(CF_{G}(L), \{m^{G}_k\}, e^G)$ is defined as a canonical model of $\varprojlim  CF(L_{l}) $ using Corollary \ref{canonicalalg}.
\end{defn}

\begin{remark}
    Basically the same statements were proved in \cite[Proposition 3.8]{KLZ} in singular/Morse models. We prove them using de Rham model.
\end{remark}

\begin{remark}
    The gapping monoid $\mathbb{G}$ will be described in the proof.
\end{remark}

\begin{remark} \label{severallag}
    More generally, for several $G$-invariant Lagrangians with pairwise clean intersections, we carry about Borel construction and inductive construction with finite-dimensional approximations as above, using Floer theory with clean intersections \cite{Fukaya-corr} for the finite-dimensional approximations. Thus, $CF_G(L^{(1)},L^{(2)})$ is defined similarly.
\end{remark}

Before proving Proposition \ref{invlimitlag}, we recall the following lemma comparing the background datum underlying the Lagrangian Floer theory of $L$ and $L_l$:

\begin{lemma}\cite[Proposition  3.1]{KLZ}; \cite[Proposition 4.7]{Cazassus}\\
The (almost K\"{a}hler) embedding $(Y,\omega, J_Y) \xrightarrow{\iota _l} (Y_l, \omega_l, J_l)$ in (\ref{ynbgnfib}) induces an isomorphism of relative homology groups
$$(\iota _l)_*: H_2(Y,L;\mathbb{Z}) \rightarrow H_2(Y_l,L_l;\mathbb{Z})$$
which respects the energy functional and the Maslov indices, i.e.
$$E(\iota _l(\beta))=E(\beta) ; MI(\iota _l(\beta))=MI(\beta). $$
In particular, the gapping monoid $\mathbb{G}_{L_l}$ is canonically identified with $\mathbb{G}_L$, which will all be denoted as $\mathbb{G}$ by abuse of notations.\\

Moreover, $(\iota _l)_*$ restricts to the subspaces of effective disk classes:
$$(\iota _l)_*: H^{\rm{eff}}_2(Y,L; J_Y) \rightarrow H^{\rm{eff}}_2(Y_l,L_l; J_l).$$
\end{lemma}

\begin{proof}
The first assertion follows from the diagram (\ref{lnfib}) and the fact that $0_{BG_l}$ is a deformation retract of $T^*BG_l$; for the last assertion, since $\iota_l$ is almost K\"{a}hler, $(\iota _l)_*$ restricts to an injection $H^{\rm{eff}}_2(Y,L; J_Y) \xrightarrow{(\iota _l)_*}H^{\rm{eff}}_2(Y_l,L_l; J_l)$; also, for any $[u_l]\in H^{\rm{eff}}_2(Y_l,L_l)$, since $\pi_l$ is pseudo-holomorphic, $\pi_l \circ u_l$ is $J_{BG_l}$-holomorphic with $[\pi_l \circ u_l] = 0$, hence is constant. Therefore, $u_l$ maps into a fiber, i.e. $[u_l]\in H^{\rm{eff}}_2(Y,L)$.\\
The energy and index identities follow from $\iota_l$ being symplectic embedding.
\end{proof}

For any $\beta\in H_2(Y,L; J_Y)$ (or $H^{\rm{eff}}_2(Y,L; J_Y)$), denote its image as $\beta_l \coloneqq(\iota_l)_*(\beta)\in H_2(Y_l,L_l; J_l)$ (or $H^{\rm{eff}}_2(Y_l,L_l; J_l)$). The proof above shows the following:

\begin{corollary}   
    There exists a (topological) fiber bundle 
\begin{equation}\label{modulifibbdl}
    \mathcal{M}_{k+1}(Y,L,\beta) \rightarrow \mathcal{M}_{k+1}(Y_l,L_l,\beta_l)\xrightarrow{\pi_l}0_{BG_l},
    \end{equation}
    where $\pi_l(u_l)\in 0_{BG_l}$ is the constant determined by $\pi_l \circ u_l$.
    
\end{corollary}

Moreover, from the diagram (\ref{yglgemb}) and the fact the inclusions are almost K\"{a}hler, we have the following sequence of (topological) fiber bundles with fiber $\mathcal{M}_{k+1}(Y,L,\beta)$:
\begin{equation}\label{moduliseq}
	\begin{tikzcd}
		&\cdots
		&\mathcal{M}_{k+1}(Y_{l-1},L_{l-1},\beta_{l-1})\arrow{r}{i_l}\arrow[d] 
		&\mathcal{M}_{k+1}(Y_l,L_l,\beta_l) \arrow[d] 
		&\cdots\\
		&\cdots  
		&0_{BG_{l-1}}\arrow[r, hook]
     &0_{BG_l} 
		&\cdots
	\end{tikzcd}
\end{equation}

It follows from (\ref{lgexactsq}) that above are pull-back diagrams, i.e.  for each $l \in \mathbb{Z}_{>0}$,
\begin{equation} \label{modulibgexactsq}
	\mathcal{M}_{k+1}(Y_{l-1},L_{l-1},\beta_{l-1}) \cong \mathcal{M}_{k+1}(Y_l,L_l,\beta_l) \times_{0_{BG_{l}} } 0_{BG_{l-1}}.
\end{equation}

Furthermore, (\ref{moduliseq}) is compatible with the evaluation maps as follows:

\begin{corollary}
\label{incluevalcompat}
For each $l\in\mathbb{Z}_{>0}$, $k\in\mathbb{Z}_{\geq0}$ and $0\leq i \leq k$, the evaluation maps at the $i$-th marked point of $\mathcal{M}_{k+1}(Y_l,L_l,\beta_l)$ are compatible with (\ref{moduliseq}), i.e. 
\begin{equation}\label{moduliev}
	\begin{tikzcd}
		&\mathcal{M}_{k+1}(Y_{l-1},L_{l-1},\beta_{l-1})\arrow{r}{i_l}\arrow{d}{ev_i} 
		&\mathcal{M}_{k+1}(Y_l,L_l,\beta_l)\arrow{d}{ev_i}\\
		&L_{l-1}\arrow[r, hook]
     &L_l
	\end{tikzcd}
\end{equation}
\end{corollary}

Again, it follows from (\ref{modulibgexactsq}) that  for each $l \in \mathbb{Z}_{>0}$,
\begin{equation} \label{modulilgexactsq}
	\mathcal{M}_{k+1}(Y_{l-1},L_{l-1},\beta_{l-1}) \cong \mathcal{M}_{k+1}(Y_l,L_l,\beta_l) \times_{L_{l}} L_{l-1}.
\end{equation}

Therefore, after fixing a tree-like $K$ system $(\{\mathcal{M}_{k+1}(L, \beta), ev, E, MI)\}$ on $\{\mathcal{M}_{k+1}(L, \beta)\}_{k;\beta}$ and a compatible system of CF-perturbations \\
$\{\mathfrak{S}_{k+1}(L;\beta)\}_{\beta\in \mathfrak{G}}$, one could construct those for each $L_l$ such that (\ref{modulilgexactsq}) holds as Kuranishi spaces. This is  crucial in showing that the integration along fibers of $ev_0$ commutes with pullbacks of differential forms, which in turn implies $g_l$ is an $A_\infty$ algebra morphism. More details are provided in the following proof.

\begin{proof}[Proof of Proposition \ref{invlimitlag}] Unitality follows immediately from definition; to prove the $A_\infty$ morphism formula, 
following the strategy in \cite{KLZ}, for each $k\geq0$, we construct Kuranishi structure inductively (over $l$) of $\{\mathcal{M}_{k+1}(L_l)\}$ such that they are compatible under inclusions and evaluation maps. Roughly speaking, this is possible because once we fixed a Kuranishi structure of $\mathcal{M}_{k+1}(L)$, by (\ref{ynbgnfib}), we can construct a Kuranishi structure of $\mathcal{M}_{k+1}(L_l)$ canonically. Compatibility would follow from the exact squares in (\ref{lgbgemb}). Similarly for the construction and compatibility of the CF-perturbations of the Kuranishi structures.\\
More precisely, we perform the following constructions:\\
\begin{enumerate}
    \item We construct the following tree-like $K$ system on \\
    $\{\mathcal{M}_{k+1}(L, \beta)\}_{k\geq0;\beta\in \mathfrak{G}}$ by \cite[Theorem 2.5]{fooodisk2}:
    $$(\{(\mathcal{M}_{k+1}(L, \beta), \widehat{U}_{k+1}(L, \beta)), ev, E, MI)\}_{k\geq0; \beta\in \mathfrak{G}}.$$ 
    Then we construct a system of $\tau$-collared Kuranishi structures and $\tau$-collared CF-perturbations $\{(\widehat{\mathcal{U}}^+_{k+1}(L, \beta), \widehat{\mathfrak{S}}_{k+1}(L, \beta))\}_{\beta\in \mathfrak{G}}$ on \\
    $\{\mathcal{M}_{k+1}(L, \beta)\}_{k\geq0;\beta\in \mathfrak{G}}$ by \cite[Proposition 22.3]{foookuranishi} 
    %such that $ev_0$ is strongly submersive  
    , inducing a strictly-unital $\mathbb{G}$-gapped filtered $A_\infty$ algebra $(CF(L), \{m_k\}, e)$.
    \item For each $l\in \mathbb{Z}_{\geq 0}$, for each $k\in \mathbb{Z}_{\geq 0}$ and $\beta \in \mathfrak{G}$ (inducing $\beta_l \in \mathfrak{G}_l \coloneqq H_2(Y_l,L_l;\mathbb{Z})$), fix local bundle trivialisations of (\ref{modulifibbdl}) with base charts $\{U_b|b\in 0_{BG_l}\}$, we construct a Kuranishi structure on $\mathcal{M}_{k+1}(Y_l,L_l,\beta_l)$ such that for any $(p,b)\in \mathcal{M}_{k+1}(Y_l,L_l,\beta_l)$ with $p\in \mathcal{M}_{k+1}(Y,L,\beta)$ and $b\in 0_{BG_l}$, its Kuranishi neighbourhood $\mathcal{U}_{p,b}$ has the form $$(U_p \times U_b,pr_1^*E_p, pr_1^*s_p, \psi_p \times Id_{U_b}).$$
    \begin{itemize}
        \item $\mathcal{U}_{p} = (U_p ,E_p, s_p, \phi_p)$ is the Kuranishi neighbourhood of $p$.
        %\item $U_b$ is an (bundle-trivialising) chart of $b$.
        \item $pr_1: U_p \times U_b \rightarrow U_p$ is the projection to the first factor.
        \item $pr_1^*E_p$ is the pullback bundle of $E_p \rightarrow U_p$.
        \item $pr_1^*s_p$ is the pullback section of $s_p$.
        \item $\psi_p \times Id_{U_b}: U_p \times U_b \rightarrow \mathcal{M}_{k+1}(Y,L,\beta) \times U_b \subseteq \mathcal{M}_{k+1}(Y_l,L_l,\beta_l)$ is a homeomorphism onto the image $Im(\psi_p)\times U_b$.
    \end{itemize}
    \item Inductively on $l\in\mathbb{Z}_{>0}$, by shrinking the bundle charts if necessary, we require that for any $b_{l-1}\in 0_{BG_{l-1}}$ and its image $b_l\in 0_{BG_l}$ under the embedding $0_{BG_{l-1}} \hookrightarrow 0_{BG_l}$, $U_{b_{l-1}}$ is exactly the the preimage of $U_{b_l}$, i.e.
    \begin{equation} \label{bdlchartexactsq}
	U_{b_{l-1}} \cong U_{b_l} \times_{0_{BG_{l}} } 0_{BG_{l-1}}.
\end{equation}
\item It follows from (\ref{bdlchartexactsq}) that both (\ref{modulibgexactsq}) and (\ref{modulilgexactsq}) are isomorphisms of Kuranishi spaces, where the Kuranishi structures on the right are the fiber product Kuranishi structures (see e.g. \cite[Definition 4.9]{foookuranishi}).
\item For each $l\in\mathbb{Z}_{\geq 0}$,  $(\{\mathcal{M}_{k+1}(Y_l,L_l,\beta_l), ev, E, MI)\}$ is a tree-like $K$ system on $\{\mathcal{M}_{k+1}(Y_l,L_l,\beta_l)\}$.  
Moreover, there is an induced compatible system of $\tau$-collared Kuranishi structures and $\tau$-collared CF-perturbations $\{(\widehat{\mathcal{U}}^+_{k+1}(L_l, \beta_l), \widehat{\mathfrak{S}}_{k+1}(L_l, \beta_l))\}$ from that of $L$ such that $ev_0$ is again strongly submersive. By construction, these systems are compatible in $l$ in the sense that $\{(\widehat{\mathcal{U}}^+_{k+1}(L_{l-1}, \beta_{l-1}), \widehat{\mathfrak{S}}_{k+1}(L_{l-1}, \beta_{l-1}))\}_{\beta_{l-1}\in \mathfrak{G}_{l-1}}$ are restrictions of \\
$\{(\widehat{\mathcal{U}}^+_{k+1}(L_l, \beta_l), \widehat{\mathfrak{S}}_{k+1}(L_l, \beta_l))\}$ under $i_l$.
\end{enumerate}

To show $g_l$ is a $\mathbb{G}$-gapped strict $A_\infty$ morphism, it suffices to show that 
$$m^{(l-1)}_{k, \beta_{l-1}}(g_l(x_1), \dots, g_l(x_{k})) = g_l(m^{(l)}_{k, \beta_{l}}(x_1, \dots, x_{k}))$$
for each $\beta\in\mathfrak{G}$. Recall that 
$$m^{(l)}_{k, \beta_{l}}(x_1, \dots, x_{k}) = (-1)^*Corr_{\mathfrak{M}_{k+1}(\beta_l)}(pr^*_1(x_1)\wedge\cdots\wedge pr^*_k(x_k); \widehat{\mathfrak{S}}_{k+1}(L_l, \beta_l)),$$
\begin{align*}
&m^{(l-1)}_{k, \beta_{l-1}}(g_l(x_1), \dots, g_l(x_{k}))\\
&= (-1)^*Corr_{\mathfrak{M}_{k+1}(\beta_{l-1})}(pr^*_1(g_l(x_1))\wedge\cdots\wedge pr^*_k(g_l(x_k)); \widehat{\mathfrak{S}}_{k+1}(L_{l-1}, \beta_{l-1}))\\
&= (-1)^*Corr_{\mathfrak{M}_{k+1}(\beta_{l-1})}(g_l(pr^*_1(x_1)\wedge\cdots\wedge pr^*_k(x_k)); \widehat{\mathfrak{S}}_{k+1}(L_{l-1}, \beta_{l-1}))
\end{align*}
where $(\mathfrak{M}_{k+1}(\beta_l); (ev_1,\cdots, ev_k); ev_0)$ is the smooth correspondence from $L_l^k$ to $L_l$ induced from $\{(\widehat{\mathcal{U}}^+_{k+1}(L_l, \beta_l), \widehat{\mathfrak{S}}_{k+1}(L_l, \beta_l))\}_{\beta_l\in \mathfrak{G}_l}$, $pr_i: L^k_l \rightarrow L_l$ is the projection to the $i$-th factor, and $* = \sum_{i=1}^{k}i(\deg{x_i}+1)+1$.

Therefore, it suffices to show that 
for any $y \in \Omega(L^k_l)$, 
$$g_l(Corr_{\mathfrak{M}_{k+1}(\beta_l)}(y; \widehat{\mathfrak{S}}_{k+1}(L_l, \beta_l))) = Corr_{\mathfrak{M}_{k+1}(\beta_{l-1})}(g_l(y); \widehat{\mathfrak{S}}_{k+1}(L_{l-1}, \beta_{l-1})).$$
Now by Corollary \ref{incluevalcompat} applied to $(ev_1,\cdots, ev_k)$,

\begin{align*}
 Corr_{\mathfrak{M}_{k+1}(\beta_l)}(y; \widehat{\mathfrak{S}}_{k+1}( \beta_l))) =& (ev_0)_!((ev_1,\cdots, ev_k)^*y; \widehat{\mathfrak{S}}_{k+1}( \beta_l)),\\
Corr_{\mathfrak{M}_{k+1}(\beta_{l-1})}(g_l(y); \widehat{\mathfrak{S}}_{k+1}(\beta_{l-1}))
=& (ev_0)_!((ev_1,\cdots, ev_k)^*(g_l(y)); \widehat{\mathfrak{S}}_{k+1}( \beta_{l-1}))\\
=& (ev_0)_!(i^*_l((ev_1,\cdots, ev_k)^*(y)); \widehat{\mathfrak{S}}_{k+1}( \beta_{l-1})).
\end{align*}

Therefore, it suffices to show that 
$$g_l((ev_0)_!(w; \widehat{\mathfrak{S}}_{k+1}(L_l, \beta_l))) = (ev_0)_!(i^*_l(w); \widehat{\mathfrak{S}}_{k+1}(L_{l-1}, \beta_{l-1}))$$
for any differential form $w$ on $\mathcal{M}_{k+1}(Y_l,L_l,\beta_l)$. This follows from \cite[Proposition 10.26]{foookuranishi} and Corollary \ref{incluevalcompat} applied to $ev_0$ (or by \cite[Proposition 10.24]{foookuranishi} and that (\ref{modulilgexactsq}) holds as Kuranishi spaces with CF-perturbations). \end{proof}

\subsection{Equivariant Morse Model}
\label{sectionequivmorsemodel}
In this subsection, we recall the original construction of equivariant Morse model $CF^{Morse}_{G}(L)$ in \cite{KLZ}.

While we define $CF_G(L)$ as a canonical model of $\varprojlim CF(L_{l})$, $CF^{Morse}_{G}(L)$ has an advantage of having a natural $H_{G}(pt)$-linear extension of the $A_\infty$ structure:

\begin{theorem}\label{klzequivmorse}
\cite[Theorem 3.12]{KLZ}
There exists a $\mathbb{G}$-gapped, (strictly) unital $A_\infty$ algebra  $(CF^{Morse}_{G}(L), \{m^{Morse}_{k, G}\}, e)$, called the $G$-equivariant Morse model, which is an $A_\infty$ algebra over $\Lambda_0(H_{G}(pt))$.
\end{theorem}

\begin{defn}
 The $G$-equivariant weak Maurer-Cartan space and disc potential (for Morse model) of $L$
$$MC^{Morse}_{G}(L) \coloneqq MC_{weak}(CF^{Morse}_{G}(L)),$$
$$W^{Morse}_{L, G}: MC^{Morse}_{G}(L) \rightarrow \Lambda_0(H_{G}(pt))$$ are defined as the weak Maurer-Cartan space and potential function associated to the $A_\infty$ algebra $CF^{Morse}_{G}(L)$.
\end{defn}

In particular, when $L$ has minimal Maslov index $0$ and is weakly unobstructed, then it is shown in \cite[Corollary 3.15]{KLZ} that $W^{Morse}_{L, G}$ has the form
$$W^{Morse}_{L, G}(b)= W_L(b) + \sum_{i=1}^{k}\lambda_i h^{i}_L(b),$$
where $b\in MC_{weak}(L)$, $\lambda_1,\ldots,\lambda_k$ are the degree-two equivariant parameters of $G$ (with rank $k$), and $h^i_L: MC_{weak}(L)\rightarrow \Lambda_0(\R)$.

We briefly recall the construction of $CF^{Morse}_{G}(L)$ in Theorem \ref{klzequivmorse}, and refer the reader to  \cite{KLZ} for details.

The underlying vector space of $CF^{Morse}_{G}(L)$ is $C(f)\otimes_\R H^\bullet_G(pt; \R)$, where $f$ is a Morse function on $L$. Its $A_\infty$ structure is constructed by realising it as an inverse limit of a sequence of Morse models of the approximation spaces $(C(f_l),  \{m^{(l), Morse}_{k}\}, e_l)$, where $(f_l:L_l \rightarrow \R)$ is a sequence of Morse functions satisfying additional properties as in \cite[Definition 3.6]{KLZ}, and then apply Proposition \ref{invlimitalg}. The $A_\infty$ structure of each $C(f_l)$ is obtained from a singular cochain model on $L_l$ via Proposition \ref{algebratransfer} applying to a singular-to-Morse contraction $(i_{sing}, p_{sing}, h_{sing})$ defined in \cite[Section 2.3]{KLZ}.

\begin{remark}
For the sake of consistency with the de Rham model we are using, we replace $(i_{sing}, p_{sing}, h_{sing})$ by a family of Witten-Morse contractions $(i_{t}, p_{t}, h_{t})_{t>t_0}$ on $\Omega(L_l)$ (with fixed $\lambda_0 = 1$ and the corresponding $t_0$) to obtain a family of $A_\infty$ structures $\{m^{(l), t}_k\}$ on $C(f_l)$. Then we define $\{m^{(l), Morse}_{k}\}$ as the limit of $\{m^{(l), t}_k\}$ as $t\rightarrow \infty$, which can be identified with the usual $A_\infty$ structure on the Morse complex $C(f_l)$ by counting pearly trajectories. See \cite{CLMfukconj} for details in the case of de Rham dga $(\Omega(L_l), d, \wedge)$.
\end{remark}

 The underlying complex  of $CF^{Morse}_{G}(L)$ is  $(C(f)\otimes_\R H^\bullet_G(pt; \R), d_G)$, where $d_G$ is an $H^\bullet_G(L; \R)$-linear differential such that $H^\bullet (C(f)\otimes_\R H^\bullet_G(pt; \R), d_G) = H^\bullet_G(L; \R)$, Therefore, we could apply Proposition \ref{strongcontractionconstr} to obtain an $\R$-linear strong contraction $(i_G, p_G, h_G)$ from $(C(f)\otimes_\R H^\bullet_G(pt; \R), d_G)$ to $H^\bullet_G(L; \R)$ as complexes over $\R$.
 
 Therefore, we first perform a restriction of scalars of the (gapped, unital) $A_\infty$ algebra  $(CF^{Morse}_{G}(L), \{m^{Morse}_{k, G}\}, e)$ to $\Lambda_0(\R)$ coefficient (via $\R \cong H^0_G(pt; \R) \subseteq H^\bullet_G(pt; \R)$) to obtain an $A_\infty$ algebra  $(CF^{Morse}_{G, \R}(L), \{m^{Morse,  \R}_{k, G}\}, e)$ over $\Lambda_0(\R)$. Then we apply Proposition \ref{algebratransfer} to $CF^{Morse}_{G, \R}(L)$ to obtain an $A_\infty$ algebra structure on $H^\bullet_G(L)$ which is homotopic to $(H^\bullet_G(L), \{m^{G}_{k}\}, e^G)$. Apply Proposition \ref{restrscamcisom} and Remark \ref{hplmcspace} yields the following corollary:

 \begin{corollary} \label{equivcfmcisom}
    There exists a bijection 
    $$MC_{weak}(H^\bullet_G(L; \Lambda_0); \Lambda_+)\xrightarrow{h} MC_{weak}(CF^{Morse}_{G}(L); \Lambda_+(H^\bullet_G(pt))) \times_{\Lambda_0(H^\bullet_G(pt))} \Lambda_0$$ defined as $h([b]) = (exp(i_G)([b]), W^{Morse}_{L, G}(exp(i_G)([b])))$. Moreover, $h$ intertwines the potential function and the natural projection, i.e. 
$$W_H([b]) = pr_2(h([b]))=W^{Morse}_{L, G}(exp(i_G)([b])).$$
 \end{corollary}

%\begin{remark}
%    Careful reader may have realised that in the above argument, there are two $A_\infty$ algebra structure on $H_G(L; \Lambda_0)$ involved: one comes from a canonical model of the equivariant de Rham model $ \varprojlim  \Omega(L_{l}; \Lambda_0)$; another one is obtained from a canonical model of the equivariant Morse model $C(f)\otimes_\R H^\bullet_G(pt; \Lambda_0)$, whose $A_\infty$ structure comes from the singular cochain model (See e.g. \cite{KLZ} for details). \\
    
%    We expect that these two $A_\infty$ algebra structure on $H_G(L; \Lambda_0)$ are $A_\infty$ homotopic, which will in turn justify the above argument. A possible solution is to ``unify" the chain models in the beginning: for instance, one can use de Rham model and perform a de Rham-to-Morse Homological Perturbation (using Witten-Morse contractions defined in Corollary \ref{wittenmorsecontr}), or use singular cochain models in constructing (equivariant) Correspondence tri-modules. 
%\end{remark}

\section{Equivariant Lagrangian Correspondence} \label{sec:equivCor}
In this section, we establish a Floer theory for equivariant Lagrangian correspondences and address Teleman's conjecture. 

We first construct an equivariant extension of correspondence tri-module in Subsection \ref{subsec-equivcorrtrimod} and cyclic property in \ref{subsec-equivcyclicprop}. We then prove relations between the (equivariant) Lagrangian Floer theory of Hamiltonian $G$-manifolds $Y$ and their symplectic quotients $X$ in Subsection \ref{subsec-lftquotient}. These relations are applied to settle (a Floer-theoretic version of) a conjecture of Teleman in \cite{Teleman} on constructing mirrors of $X$ from that of $Y$ in Subsection \ref{subsec-telconj}. Finally, we discuss potential generalisations to singular quotients to study conjectures by Lekili and Segal \cite{LS} in Subsection \ref{singlevel}.

\subsection{Equivariant Correspondence Tri-module}
\label{subsec-equivcorrtrimod}

In this subsection, we construct correspondence tri-module for equivariant Lagrangian correspondences.

\begin{setup} \label{hamtriple}
Consider closed or tame Hamiltonian spaces 
$$((P, \omega_P), K, \mu_P), ((M, \omega_M), G, \mu_M), ((N, \omega_N), H, \mu_N)$$ and closed, connected $G\times H$ (resp. $K\times G, K\times H$)-invariant Lagrangian correspondences 
 $$M \xrightarrow{L} N; P\xrightarrow{L'} M; P \xrightarrow{L''} N.$$
 Moreover, we fix a $G\times H$ (resp. $K\times G, K\times H$)-equivariant relative spin structure on $L$ (resp. $L'$, $L''$) with respect to
 $pr^*_M(V_M \oplus TM) \oplus pr^*_NV_N$ (resp.  $pr^*_P(V_P \oplus TP) \oplus pr^*_MV_M; pr^*_P(V_P \oplus TP) \oplus pr^*_NV_N)$
 for some oriented real vector bundles $V_P\rightarrow P; V_M\rightarrow M; V_N\rightarrow N.$
 %Given compact Lie groups $G, H, K$, Hamiltonian actions $G$ (resp. $H$, $K$) on  symplectic manifolds $M$ (resp. $N$, $P$)  with moment map $\mu_M$ (resp. $\mu_N$, $\mu_P$).\\
% Fix central elements $c_M\in \mathfrak{g}^*, c_N\in \mathfrak{h}^*$, $c_P\in \mathfrak{k}^*$ such that $G$ (resp. $H, K$) acts freely on $\mu_{M}^{-1}(c_M)$ (resp. $\mu_{N}^{-1}(c_N), \mu_{P}^{-1}(c_P)$) with symplectic quotient $\overline{M} \coloneqq M\sslash_{c_M} G$ (resp. $\overline{N} \coloneqq N\sslash_{c_N} H, \overline{P} \coloneqq P\sslash_{c_P} K$).\\
 %Denote $S \coloneqq G\times H, S' \coloneqq K\times G, S'' \coloneqq K\times H$ 
\end{setup}

Under this setup, we study their Lagrangian correspondence Borel spaces 
$$M_G \xrightarrow{L_{G\times H}} N_H, P_K\xrightarrow{L'_{K\times G}} M_G, P_K \xrightarrow{L''_{K\times H}} N_H$$
via their finite dimensional approximations.  For each $l\in \mathbb{Z}_{\geq 0}$, we have
$$M_l \xrightarrow{L_l} N_l, P_l\xrightarrow{L'_l} M_l, P_l \xrightarrow{L''_l} N_l.$$

We would like to define the equivariant correspondence tri-module \\
$CF_{eq}(L''; L', L)$ by the correspondence tri-module $CF(L''_{K\times H}; L'_{K\times G}, L_{G\times H})$. In practice, we consider a sequence of correspondence tri-modules of their approximations $\{CF(L''_{l}; L'_{l}, L_{l})\}_{l \geq 0}$ and define it as the inverse limit 
$$CF_{eq}(L''; L', L) \coloneqq \varprojlim  CF(L''_{l}; L'_{l}, L_{l}).$$ 

It is endowed with a natural unital $A_\infty$ tri-module structure: by Proposition \ref{invlimittrimod}, it suffices to show that $\{(CF(L''_{l}; L'_{l}, L_{l}), \{n^{(l)}_{k'', k', k}\}\}_{l \geq 0}$ is a sequence of unital $A_\infty$ tri-modules with (strict) $A_\infty$ tri-module morphisms $f_l$,
$$ CF(L''_{0}; L'_{0}, L_{0}) \xleftarrow{f_1} \dots \leftarrow CF(L''_{l-1}; L'_{l-1}, L_{l-1}) \xleftarrow{f_l} CF(L''_{l}; L'_{l}, L_{l})\leftarrow  \dots$$
This motivates the following proposition:

\begin{prop} \label{invlimitlagcorr}
    For each $l\in \mathbb{Z}_{>0}$, assume that the following intersection is clean:
    $$I_l \coloneqq (L''_{l} \times L'_{l} \times L_{l})\cap \Delta_{P_l M_l N_l}  \subseteq P_l \times N_l \times P_l \times M_l \times M_l \times N_l .$$
    then the pullback of the inclusion map
    $CF(L''_{l-1}; L'_{l-1}, L_{l-1}) \xleftarrow{f_l} CF(L''_{l}; L'_{l}, L_{l})$ 
is a $\mathbb{G}$-gapped strict $A_\infty$ tri-module morphism, i.e. 
    $$n^{(l-1)}_{k'',k',k}(g''_l(x''_1), \dots, g''_l(x''_{k''}); f_l(y); g'_l(x'_1), \dots, g'_l(x'_{k'}); g_l(x_1), \dots, g_l(x_{k}))$$
$$= f_l(n^{(l)}_{k'',k',k}(x''_1, \dots, x''_{k''}; y; x'_1, \dots, x'_{k'}; x_1, \dots, x_{k}))$$
    for any $(x''_j) \in CF(L''_{l})^{\otimes k''}$, $(x'_j) \in CF(L'_{l})^{\otimes k'}$, $(x_j) \in CF(L_{l})^{\otimes k}$, $y \in CF(L''_{l}; L'_{l}, L_{l})$, where $g_l, g'_l, g''_l$ are the $A_{\infty}$-algebra morphisms defined in Proposition \ref{invlimitlag}.
\end{prop}

The rest of this subsection is to prove this proposition.

\begin{remark}
    The strict $A_\infty$ tri-module morphism $f_l$ is automatically unital, since the higher terms of $f_l$ are zero by definition. \end{remark}

    \begin{remark}
The gapping monoid $\mathbb{G}$ will be described in the course of its proof.
\end{remark}

\begin{corollary}
    The inverse limit 
  $\varprojlim  CF(L''_{l}; L'_{l}, L_{l})$
    has a natural $\mathbb{G}$-gapped unital left $\varprojlim CF(L''_l)$, right $(\varprojlim CF(L'_l), \varprojlim CF(L_l)$) $A_\infty$ tri-module structure $\{n_{k'', k', k}\}$.
\end{corollary}

\begin{defn}
Denote $(CF_{eq}(L''; L', L), \{n^{eq}_{k'', k', k}\})$ the equivariant correspondence tri-module
 defined by the canonical tri-module of  $\varprojlim  CF(L''_{l}; L'_{l}, L_{l})$, as a $\mathbb{G}$-gapped unital left $CF_{K\times H}(L'')$, right $(CF_{K\times G}(L'), CF_{G\times H}(L))$ $A_\infty$ tri-module.
\end{defn}

The proof of Proposition \ref{invlimitlagcorr} is similar to that of Proposition \ref{invlimitlag}: fix a system of Kuranishi structures  $\mathcal{U} = \{(\mathcal{U}_{k'',k',k }(L'', L', L; E)\}$ on $\{\mathcal{M}_{k'', k', k}(L''; L', L; E)\}$, we inductively construct Kuranishi structures on the moduli spaces of quilted drums $\mathcal{M}_{k'', k', k}(L''_l; L'_l, L_l; E)$ with respect to fiber bundles defined as follows:
\begin{prop}
For each $l\in\mathbb{Z}_{\geq 0}$, there exists a topological fiber bundle
\begin{equation}\label{quilteddrumfibbdl}
    \mathcal{M}_{k'', k', k}(L''; L', L; E) \rightarrow \mathcal{M}_{k'', k', k}(L''_l; L'_l, L_l; E)\xrightarrow{\pi_l}I^{K,G,H}_l
    \end{equation}
    %where $I^{K,G,H}_l \coloneqq (0_{BS_l}\times 0_{BS'_l}\times 0_{BS''_l})\cap\Delta_{T^*BK \times T^*BG \times T^*BH}$
    
    where $\Delta: T^*BK \times T^*BG \times T^*BH \rightarrow (T^*BK \times T^*BG \times T^*BH)^2$ is the diagonal map, and $I^{K,G,H}_l \coloneqq \Delta^{-1} (0_{B(G\times H)_l}\times 0_{B(K\times G)_l}\times 0_{B(K\times H)_l})$.
\end{prop}

\begin{proof}
    Given a quilted drum $u_l = (u_{P_l}, u_{M_l}, u_{N_l})$ in $\mathcal{M}_{k'', k', k}(L''_l; L'_l, L_l; E)$, consider its projection 
    $\overline{u_l} \coloneqq (u_{K_l}, u_{G_l}, u_{H_l}) \coloneqq (\pi_{P_l}\circ u_{P_l}, \pi_{M_l}\circ u_{M_l}, \pi_{N_l}\circ u_{N_l})$
    as a quilted drum with patches targeting $(T^*BK_l, T^*BG_l, T^*BH_l)$ and seams targeting $(0_{B(G\times H)_l}\times 0_{B(K\times G)_l}\times 0_{B(K\times H)_l})$. We claim that $\overline{u_l}$ is constant by showing $E(\overline{u_l})=0$: 

\begin{align*}
&\int_{\Sigma_1}u^*_{K_l}\omega_{K_l} + \int_{\Sigma_2}u^*_{G_l}\omega_{G_l} + \int_{\Sigma_3}u^*_{H_l}\omega_{H_l}=(\int_{\sigma_1}u^*_{K_l}\alpha_{K_l} - \int_{\sigma_3}u^*_{K_l}\alpha_{K_l}) \\
&+ (\int_{\sigma_2}u^*_{G_l}\alpha_{G_l} - \int_{\sigma_1}u^*_{G_l}\alpha_{G_l}) + (\int_{\sigma_3}u^*_{H_l}\alpha_{H_l} - \int_{\sigma_2}u^*_{H_l}\alpha_{H_l})\\
&=\int_{\sigma_1}(u_{K_l}, u_{G_l})^*(\alpha_{K_l},  -\alpha_{G_l}) + \int_{\sigma_2}(u_{G_l}, u_{H_l})^*(\alpha_{G_l},  -\alpha_{H_l}) \\
&+\int_{\sigma_3}(u_{K_l}, u_{H_l})^*(-\alpha_{K_l},  \alpha_{H_l})= 0 + 0 + 0 =0,
\end{align*}
where the first equality is by Stoke's Theorem, and the third equality is by the seam conditions and that 
$$(-\alpha_{K_l},  \alpha_{H_l})|_{0_{B(K\times H)_l}}=0; (-\alpha_{K_l},  \alpha_{G_l})|_{0_{B(K\times G)_l}}=0; (-\alpha_{G_l},  \alpha_{H_l})|_{0_{B(G\times H)_l}}=0.$$ 
Therefore, $\overline{u_l} \equiv (a, b, c)$ is a constant map. Seam conditions imply 
$$(a,c)\in 0_{B(K\times H)_l}, (a,b)\in 0_{B(K\times G)_l}, (b,c)\in 0_{B(G\times H)_l}.$$
Hence $(a, b, c) \in I^{K,G,H}_l$. Define $\mathcal{M}_{k'', k', k}(L''_l; L'_l, L_l; E)\xrightarrow{\pi_l}I^{K,G,H}_l$ by 
$$\pi_l(u_l) = \overline{u_l} \equiv (a, b, c).$$
The fiber of $\pi_l$ at $(a, b, c)$ consists of quilted drums $u_l = (u_{P_l}, u_{M_l}, u_{N_l})$ such that 
$$\pi_{P_l}\circ u_{P_l} \equiv a; \pi_{M_l}\circ u_{M_l} \equiv b; \pi_{N_l}\circ u_{N_l} \equiv c,$$
implying $u_{P_l}: \Sigma_1 \rightarrow (\pi_{P_l})^{-1}(a) \cong P $. Similarly, $u_{M_l}: \Sigma_1 \rightarrow M$ and $u_{N_l}: \Sigma_1 \rightarrow N$. \\
Moreover, $(a,b)\in 0_{B(K\times H)_l}$ implies $(u_{P_l}, u_{M_l})|_{\sigma_1}: \sigma_1 \rightarrow \pi_{L'_l}^{-1}(a,b)\cong L'$. Similarly, $(u_{M_l}, u_{N_l})_{\sigma_2}: \sigma_2 \rightarrow L$ and  $(u_{P_l}, u_{N_l})_{\sigma_3}: \sigma_3 \rightarrow L''$. Therefore, $\pi_l^{-1}(a,b,c) \cong \mathcal{M}_{k'', k', k}(L''; L', L; E)$.    
\end{proof}

In particular, the gapping monoid $\mathbb{G}_{L''_l, L'_l, L_l}$ can be canonically identified with  $\mathbb{G}_{L'', L', L}$, and will all be denoted by $\mathbb{G}$ by abuse of notations.

Therefore, using the induced charts of $I^{K,G,H}_l$ from \\
$0_{B(G\times H)_l}, 0_{B(K\times G)_l}, 0_{B(K\times H)_l}$, $\mathcal{M}_{k'', k', k}(L''_l; L'_l, L_l; E)$ admits a fiber bundle Kuranishi structure $\mathcal{U}_l = \{(\mathcal{U}_{k'',k',k }(L''_l, L'_l, L_l; E)\}$. 

    It follows that these bundles are related by inclusions in the following sense:
    \begin{equation}
    \label{quilteddrumseq}
	\begin{tikzcd}
		&\cdots
		&\mathcal{M}_{k'', k', k}(L''_{l-1}; L'_{l-1}, L_{l-1}; E)\arrow{r}{i_l}\arrow[d] 
		&\mathcal{M}_{k'', k', k}(L''_l; L'_l, L_l; E) \arrow[d]\\
		&\cdots
		&I^{K,G,H}_{l-1}\arrow[r, hook]
     &I^{K,G,H}_l
	\end{tikzcd}
\end{equation}
Then after shrinking the bundle charts if necessary, we may assume that
\begin{equation} \label{quilteddrumexactsq}
\mathcal{M}_{k'', k', k}(L''_{l-1}; L'_{l-1}, L_{l-1}; E) \cong\mathcal{M}_{k'', k', k}(L''_l; L'_l, L_l; E) \times_{I^{K,G,H}_{l-1}}I^{K,G,H}_l
\end{equation}
    are isomorphisms of Kuranishi spaces.\\
    Moreover, the evaluation maps $ev_{L_l}$ are compatible with (\ref{quilteddrumseq}), i.e. 
\begin{equation}
\label{quilteddrumseamevexactsq}
	\begin{tikzcd}
		&\mathcal{M}_{k'', k', k}(L''_{l-1}; L'_{l-1}, L_{l-1}; E)\arrow{r}{i_l}\arrow{d}{ev_i} 
		&\mathcal{M}_{k'', k', k}(L''_l; L'_l, L_l; E)\arrow{d}{ev_i}\\
		&L_{l-1}\arrow[r, hook]
     &L_l
	\end{tikzcd}
\end{equation}
Similarly for $ev_{L'_l}$ and $ev_{L''_l}$. Also, the inclusions $i_l, \iota_l$ are compatible with the evaluation maps at infinity ends $ev^{(l)}_{\pm\infty}$ of $\mathcal{M}_{k'', k', k}(L''_l; L'_l, L_l; E)$, i.e. 
\begin{equation}
\label{quilteddrumasympevexactsq}
	\begin{tikzcd}
		&\mathcal{M}_{k'', k', k}(L''_{l-1}; L'_{l-1}, L_{l-1}; E)\arrow{r}{i_l}\arrow{d}{ev^{(l)}_{\pm\infty}} 
		&\mathcal{M}_{k'', k', k}(L''_l; L'_l, L_l; E)\arrow{d}{ev^{(l)}_{\pm\infty}}\\
		&I_{l-1}\arrow{r}{\iota_l}
     &I_l
	\end{tikzcd}
\end{equation}
Again, it follows from (\ref{quilteddrumexactsq}) that the above pullback diagram
\begin{equation} 
\label{quilteddrumasympevexactsq2}
\mathcal{M}_{k'', k', k}(L''_{l-1}; L'_{l-1}, L_{l-1}; E) \cong\mathcal{M}_{k'', k', k}(L''_l; L'_l, L_l; E) \times_{ev^{(l)}_{\pm\infty}, I_{l-1}}I_l
\end{equation}
are isomorphisms of Kuranishi spaces.\\
Therefore, after fixing a system of collared Kuranishi structures and collared CF-perturbations 
$$(\widehat{\mathcal{U}}^+, \widehat{\mathfrak{S}})= \{(\widehat{\mathcal{U}}^+_{k'',k',k }(L'', L', L; E), \widehat{\mathfrak{S}}_{k'',k',k }(L'', L', L; E)\}$$ on $\{\mathcal{M}_{k'', k', k}(L''; L', L; E)\}$, there is an induced compatible system of collared Kuranishi structures and collared CF-perturbations $$(\widehat{\mathcal{U}}_l^+, \widehat{\mathfrak{S}}_l) = \{(\widehat{\mathcal{U}}^+_{k'',k',k }(L''_l, L'_l, L_l; E), \widehat{\mathfrak{S}}_{k'',k',k }(L''_l, L'_l, L_l; E)\}$$ on $\{\mathcal{M}_{k'', k', k}(L''_l; L'_l, L_l; E)\}$ for each $l$ such that $ev^{(l)}_{\pm \infty}$ are strongly submersive. By construction, these systems are compatible in $l$ in the sense that $(\widehat{\mathcal{U}}_{l-1}^+, \widehat{\mathfrak{S}}_{l-1})$ are restrictions of $(\widehat{\mathcal{U}}_{l}^+, \widehat{\mathfrak{S}}_{l})$ under $i_l$.

\begin{proof}[Proof of Proposition \ref{invlimitlagcorr}] To show $f_l$ is a $\mathbb{G}$-gapped strict $A_\infty$ tri-module morphism, it suffices to show that for each fixed $E_0$, for all $E<E_0$, $\epsilon > 0$ 
$$n^{(l-1), E, \epsilon}_{k'',k',k}(g''_l(x''_1), \dots, g''_l(x''_{k''}); f_l(y); g'_l(x'_1), \dots, g'_l(x'_{k'}); g_l(x_1), \dots, g_l(x_{k}))$$
$$= f_l(n^{(l), E, \epsilon}_{k'',k',k}(x''_1, \dots, x''_{k''}; y; x'_1, \dots, x'_{k'}; x_1, \dots, x_{k})).$$

Recall that the RHS is defined as 
$$f_l(ev^{(l)}_{+\infty !}((ev_{L''_l}^*x'') \wedge ((ev^{(l)}_{-\infty})^*y) \wedge (ev_{L'_l}^*x') \wedge (ev_{L_l}^*x); \widehat{\mathfrak{S}}^\epsilon_l)),$$

and the LHS is defined as 
\begin{align*}
    &ev^{(l-1)}_{+\infty !}((ev_{L''_{l-1}}^*g''_l(x'')) \wedge ((ev^{(l-1)}_{-\infty})^*f_l(y)) \wedge (ev_{L'_{l-1}}^*g'_l(x')) \wedge (ev_{L_{l-1}}^*g_l(x)); \widehat{\mathfrak{S}}^\epsilon_{l-1})\\
    =&ev^{(l-1)}_{+\infty !}(i^*_{l}(ev_{L''_{l}}^*(x'')) \wedge (i^*_{l}((ev^{(l)}_{-\infty})^* y)) \wedge (i^*_{l}(ev_{L'_{l}}^*(x')) \wedge (i^*_{l}(ev_{L_{l}}^*(x)); \widehat{\mathfrak{S}}^\epsilon_{l-1})\\
    =&ev^{(l-1)}_{+\infty !}(i^*_{l}((ev_{L''_{l}}^*(x'')) \wedge ((ev^{(l)}_{-\infty})^* y) \wedge (ev_{L'_{l}}^*(x')) \wedge (ev_{L_{l}}^*(x))); \widehat{\mathfrak{S}}^\epsilon_{l-1}),
\end{align*}

where the first equality is by (\ref{quilteddrumseamevexactsq}) and (\ref{quilteddrumasympevexactsq}) respectively. It suffices to show that 
$$ev^{(l-1)}_{+\infty !}(i^*_{l} (w); \widehat{\mathfrak{S}}^\epsilon_{l-1}) = f_l(ev^{(l)}_{+\infty !}(w; \widehat{\mathfrak{S}}^\epsilon_l))$$
for any differential form $w$ on $\mathcal{M}_{k'', k', k}(L''_l; L'_l, L_l; E)$. This follows from \cite[Proposition 10.26]{foookuranishi} and (\ref{quilteddrumasympevexactsq}) applied to $ev_{+\infty}$ (or alternatively \cite[Proposition 10.24]{foookuranishi} and that (\ref{quilteddrumasympevexactsq2}) holds as Kuranishi spaces with CF-perturbations). 
\end{proof}

\begin{remark}
\label{equivcorrmodoverequivmorsemodel}
In spirit of Subsection \ref{sectionequivmorsemodel}, while we have defined the equivariant correspondence tri-module $(CF_{eq}(L''; L', L), \{n^{eq}_{k'', k', k}\})$ as a left
$CF_{K\times H}(L'')$, right $(CF_{K\times G}(L'), CF_{G\times H}(L))$ $A_\infty$ tri-module, we could replace $CF_{G\times H}(L)$ by any ``intermediate model" (e.g. the equivariant Morse model) between $\varprojlim CF(L_l)$ and $CF_{G\times H}(L)$ by applying Corollary \ref{hpltrimodbasechange} to the corresponding contraction. Similarly for $CF_{K\times G}(L')$ and $CF_{K\times H}(L'')$. 
\end{remark}

\subsection{Equivariant Cyclic Property} 
\label{subsec-equivcyclicprop}
In this subsection, we develop an equivariant extension of cyclic property under the following setup:
\begin{setup} \label{equivcyclicsetup}
    Under the Setup \ref{hamtriple}, assume further that $L, L'$ are cleanly composable, and $L'' = L\circ L'$ as $K \times H$-Lagrangian correspondence. 
\end{setup}

By Proposition \ref{equilagcorrcomp}, for each $l\in\mathbb{Z}_{\geq 0}$, the correspondence tri-module $CF((L \circ L')_l; L'_l, L_l))$ is naturally identified with $CF(L_l \circ L'_l; L'_l, L_l)$ for which the constant one function $const^{(l)}_1: L_l \circ L'_l\cong ((L_l \circ L'_l)\times L'_l \times L_l)\cap \Delta \rightarrow \mathbb{R}$ defines an $\overline{n}^{(l)}_{0,0,0}$-closed element $const^{(l)}_1 \in CF^0(L_l \circ L'_l; L'_l, L_l)$ and hence an element $\mathbf{1}_l \in CF^0_{can}(L_l \circ L'_l; L'_l, L_l)$. Running over all $l$, we obtain an $\overline{n}^{(\infty)}_{0,0,0}$-closed element $(const^{(l)}_1) \in \varprojlim CF^0(L_l \circ L'_l; L'_l, L_l)$ and hence an element $\mathbf{1}_\infty \coloneqq (\mathbf{1}_l) \in CF^0_{eq}(L \circ L'; L', L)$. Note that a priori these $\mathbf{1}_l$ need not be left cyclic, as $L_l \circ L'_l$ need not be a transverse composition. In view of this, we introduce the following definition:

\begin{defn}
\label{cleancompseqdef}
    Under the setup \ref{equivcyclicsetup}, we say the sequence $(L_l \circ L'_l)$ is  cleanly composable if each $L_l \circ L'_l$ is cleanly composable and the corresponding fibrations are compatible over $l$, i.e. we have the following commutative diagrams
    \begin{equation} \label{cleancompseq}
	\begin{tikzcd}
		&F = F_0 \arrow[r, hook] \arrow[d]
		& F_1\arrow[r, hook] \arrow[d]  
		&\cdots
		&\arrow[r, hook]
		&F_l  \arrow[d] \\
  &I = I_0 \arrow[r, hook] \arrow[d, "p_{L\circ L'}"]
		& I_1\arrow[r, hook] \arrow[d, "p_{L_1\circ L'_1}"]  
		&\cdots
		&\arrow[r, hook]
		&I_l  \arrow[d, "p_{L_l\circ L'_l}"] \\
		&L\circ L' = L_0\circ L'_0  \arrow[r, hook]
		& L_1\circ L'_1 \arrow[r, hook]
		&\cdots
		&\arrow[r, hook]
		&L_l\circ L'_l 
	\end{tikzcd}
\end{equation}
where $I_l \coloneqq ((L_{l} \circ L'_{l})  \times L'_{l} \times L_{l})\cap \Delta_{P_l M_l N_l}$.
\end{defn}

\begin{theorem}
\label{equivcyclictheorem}
Under the setup \ref{equivcyclicsetup}, assume the sequence $(L_l \circ L'_l)$ is  cleanly composable and in addition the following:
\begin{enumerate}
    \item The sequence $(\dots \leftarrow \Omega(I_{l-1})\xleftarrow{f_l}\Omega(I_l)\leftarrow\dots)$ is homologically stable.
    \item For each $l\in\mathbb{Z}_{\geq 0}$, there exists $r(l)\in\mathbb{Z}$, increasing to $+\infty$ as $l\rightarrow +\infty$, such that $H^{m}(F_l)=0$ for all $0<m\leq r(l)$.
    
\end{enumerate}
Then $\mathbf{1}_\infty \in CF^0_{eq}(L \circ L'; L', L)$ is left cyclic (i.e.``equivariant cyclic").
\end{theorem}

\begin{proof}
    (2) implies $n^{(l), H}_{1,0,0, \beta_0}(-; \mathbf{1}_l) = p^*_{L_l\circ L'_l}: H^m(L_l\circ L'_l)\rightarrow H^m(I_l)$ is an isomorphism for all $m\leq r(l)$. The result follows from Proposition \ref{invlimcyclic} and that the sequence $(\dots \leftarrow \Omega((L\circ L')_{l-1})\xleftarrow{g''_l}\Omega((L\circ L')_{l})\leftarrow\dots)$ is homologically stable.
\end{proof}

Therefore, together with Proposition \ref{algcompweak} implies the following construction of ``composition of equivariant deformation cochains" as follows:

\begin{corollary}\label{compequivdefcochain}

There exists a map 
\begin{equation}
\label{equivcomp}
    CF^{odd}_{G\times H, +}(L) \times CF^{odd}_{K\times G, +}(L') \xrightarrow{\circ} CF^{odd}_{K \times H, +}(L''),   \end{equation}  
    $$(b,b') \mapsto b'' \coloneqq b \circ b',$$

    characterised by the equation $n^{eq, b'',b',b}_{0,0,0}(\mathbf{1}_\infty)=0$. 
    
    Moreover, (\ref{equivcomp}) restricts to a map between equivariant weak Maurer-Cartan sets
    \begin{equation} 
    \widehat{MC}^{G\times H}_{weak}(L) \times \widehat{MC}^{K\times G}_{weak}(L') \xrightarrow{\circ} \widehat{MC}^{K\times H}_{weak}(L'')
    \end{equation} 
in which their equivariant disk potentials satisfy
    \begin{equation}
    \label{equivmccomp}
    W^{G\times H}_L(b) + W^{K\times G}_{L'}(b') = W^{K\times H}_{L''}(b'').
     \end{equation}

     Also, (\ref{equivmccomp}) descends to a map between equivariant weak Maurer-Cartan spaces
    \begin{equation} 
    MC^{G\times H}_{weak}(L) \times MC^{K\times G}_{weak}(L') \xrightarrow{\circ} MC^{K\times H}_{weak}(L'').
    \end{equation}
    
\end{corollary}

\subsection{Application to Floer Theory of Symplectic quotients}\label{subsec-lftquotient}
In this subsection, we relate the equivariant Lagrangian Floer theory of Hamiltonian $G$-manifolds and the Lagrangian Floer theory of their symplectic quotients via Theorem \ref{equivcyclictheorem}.
\begin{setup} \label{hamsmred}
Given a Hamiltonian space $((Y, \omega_Y), G, \mu_Y)$, assume that $\mu_Y$ is proper and $G$ acts freely on the moment level $Z\coloneqq\mu_{Y}^{-1}(0)$, then its symplectic quotient $$(X \coloneqq Y\sslash_0 G = Z/G, \omega_X \coloneqq \omega_{red})$$ is a closed symplectic manifold with a principal $G$ bundle $\pi: Z \rightarrow X$.
\end{setup}
%Moreover, each $G$-invariant almost K\"{a}hler structure $(Y, \omega_Y, J_Y, g_Y)$ on $Y$ descends canonically to an almost K\"{a}hler structure $(X, \omega_X, J_X, g_X)$ on $X$.
%\begin{setup}
%Under setup \ref{ham}, by moment map equation, every  Lagrangian $L \subseteq Y$ which lies in $\mu_Y^{-1}(c)$ is $G$-invariant, and hence descends to the quotient $\bar{L} \subseteq X$.
%\end{setup}

We also fix the Lagrangians that we are interested in:

\begin{setup}
\label{hamlagsmred}
    Under the Setup \ref{hamsmred}, fix a closed, connected, $G$-equivariantly $V$-relatively spin $G$-invariant Lagrangian $L \subseteq Z \subseteq Y$, which descends to a closed, connected Lagrangian $\bar{L} \coloneqq L/G \subseteq X$ with the quotient $\bar{V}
    $-relatively spin structure, where $\bar{V}\rightarrow X$ is defined as 
    $\bar{V}\coloneqq (V|_Z)/G.$
    
    Also, consider the moment level Lagrangian $Y \xrightarrow{L^\pi} X$ as a closed, connected $G$-Lagrangian correspondence defined as the graph of $\pi$:
    $$L^\pi  = \{(y, \pi(y))| y\in Z\} \cong Z,$$
  with the canonical $G$-equivariant relatively spin structure with respect to $pr^*_Y(V \oplus TY) \oplus pr^*_X\bar{V}$ \footnote{which is obtained from identifying $(pr^*_Y(V \oplus TY) \oplus pr^*_X\bar{V})|_{L^\pi} \oplus TL^\pi \rightarrow L^\pi$  with $(V|_Z \oplus TZ)^{\oplus2} \oplus (Z\times \mathfrak{g}^*)\rightarrow Z$.}.

   Note that $\bar{L} = L^\pi \circ L$ is a clean composition of $L$ and $L^\pi$.
\end{setup}

We consider the equivariant correspondence tri-module $CF_{eq}(\bar{L}; L, L^\pi)$ with $K = 1 = H$, $G = G$. Therefore, it is a $\mathbb{G}$-gapped unital left $CF_{can}(\bar{L})$, right $(CF_{G}(L), CF_{G}(L^\pi))$ $A_\infty$ tri-module. To apply Theorem \ref{equivcyclictheorem}, we first show the following:

\begin{lemma}
    For each $l\in\mathbb{Z}_{\geq 0}$, $L^\pi_l \circ L_l$ is a clean composition. The corresponding fibration can be identified with the following:
    \begin{equation} 
\begin{tikzcd} 
& F_l  \arrow[r] &(X_l \times L_l\times L^\pi_l)\cap \Delta_{Y_l X_l}   \arrow{r} & L^\pi_l \circ L_l \\
& 0_{EG_l}  \arrow[u, phantom, sloped, "\cong"] \arrow[r] &L_l \arrow[u, phantom, sloped, "\cong"]\arrow{r} & \bar{L} \arrow[u, phantom, sloped, "="]
\end{tikzcd}
\end{equation}
    
\end{lemma}
\begin{proof}
    Note that 
\begin{align*}
    (X_l \times L_l\times L^\pi_l)\cap \Delta_{Y_l X_l} &= \{(p_c(y), [y, a], p_c(y) , [y, a])| y\in L; a \in 0_{EG_l}\}\\
    &\cong \{[y, a])| y\in L; a \in 0_{EG_l}\} = L_l.
\end{align*}
Moreover, $L^\pi_l \circ L_l = \{p_c(y)| y\in L\} = \bar{L}$. The fibration can therefore be identified with the projection $L_l = L \times_{G} 0_{EG_l} \rightarrow \bar{L}$ with fiber $0_{EG_l}$.
\end{proof}

The above proof readily shows the following:

\begin{corollary} 
\label{redlagseq}
    The sequence of fibrations in Definition \ref{cleancompseq} can be identified with \begin{equation}
	\begin{tikzcd}
		&G \cong 0_{EG_0} \arrow[r, hook] \arrow[d]
		& 0_{EG_1}\arrow[r, hook] \arrow[d]  
		&\cdots
		&\arrow[r, hook]
		&0_{EG_l} \arrow[r, hook] \arrow[d] 
		&\cdots\\
  &L = L_0 \arrow[r, hook] \arrow[d, "p_{L}"]
		& L_1\arrow[r, hook] \arrow[d, "p_{L_1}"]  
		&\cdots
		&\arrow[r, hook]
		&L_l \arrow[r, hook] \arrow[d, "p_{L_l}"] 
		&\cdots\\
		&\bar{L}  \arrow[r, phantom, "="]
		&\bar{L} \arrow[r, phantom, "="]
		&\cdots
		&\arrow[r, phantom, "="]
		&\bar{L}\arrow[r, phantom, "="]
		&\cdots
	\end{tikzcd}
\end{equation}
In particular, the sequence is cleanly composable.
\end{corollary}

We are now ready to show $CF^0_{eq}(\bar{L}; L, L^\pi)$ admits a left cyclic element:

\begin{prop}
    $\mathbf{1}_\infty\coloneqq (\mathbf{1}_l) \in CF^0_{eq}(\bar{L}; L, L^\pi)$ is left cyclic.
\end{prop}

\begin{proof}
    By Theorem \ref{equivcyclictheorem}, it suffices to show those assumptions hold: (1) follows from Corollary \ref{redlagseq} and that $(\dots \leftarrow \Omega(L_{l-1})\xleftarrow{g_l}\Omega(L_l)\leftarrow\dots)$ satisfies homological stability; (2) follows from $EG_l$ being $r(l)$-connected with $r(l)=l-1$.
\end{proof}

Applying Corollary \ref{compequivdefcochain} yields the following:

\begin{corollary}
\label{compequiud}
 There exists a map 
\begin{equation} 
\label{udcomp}
    CF^{odd}_{G, +}(L^\pi) \times CF^{odd}_{G, +}(L)\xrightarrow{\circ} CF^{odd}_{can, +}(\bar{L}),  \end{equation}  
    $$(b_{L^\pi}, b_{L}) \mapsto b_{\bar{L}} \coloneqq b_{L^\pi} \circ b_{L},$$
    characterised by the equation $n^{eq, b_{\bar{L}} ,b_{L} ,b_{L^\pi}}_{0,0,0}(\mathbf{1}_\infty)=0$.
    
    Moreover, (\ref{udcomp}) restricts to a map between their weak Maurer-Cartan sets
    \begin{equation}
    \label{udmccomp}
    \widehat{MC}^G_{weak}(L^\pi) \times \widehat{MC}^G_{weak}(L) \xrightarrow{\circ} \widehat{MC}_{weak}(\bar{L})
    \end{equation}
    such that their potential functions satisfy
    \begin{equation} 
    W^G_{L^\pi}(b_{L^\pi}) + W^G_{L}(b_{L}) = W_{\bar{L}}(b_{\bar{L}}).
     \end{equation}
     Furthermore, (\ref{udmccomp}) descends to a map between their weak Maurer-Cartan spaces:
    \begin{equation}  MC^G_{weak}(L^\pi) \times MC^G_{weak}(L) \xrightarrow{\circ} MC_{weak}(\bar{L}).
    \end{equation}
\end{corollary}

Moreover, in this case $\mathbf{1}_l$ satisfies an extra cyclic property as follows:
\begin{prop}
For each $l$, 
$\mathbf{1}_l \in CF^0(\bar{L}; L_l, L^\pi_l)$ is right $CF_{can}(L_l)$-cyclic.
\end{prop}
\begin{corollary}
$\mathbf{1}_\infty \in CF^0_{eq}(\bar{L}; L, L^\pi)$ is right $CF_G(L)$-cyclic.  
\end{corollary}

 \begin{proof}
    It suffices to observe that  $\overline{n}^{(l), H}_{0,1,0}(\mathbf{1}_l; -) : CF_{can}(L_l)\rightarrow CF_{can}(\bar{L}; L_l, L^\pi_l)$ can be identified with the identity map on $H(L_l)$.
\end{proof}

Therefore, $\mathbf{1}_\infty$ is a bi-cyclic element, for which Corollaries \ref{bicycisom}, \ref{bicycchisom} and Proposition \ref{ringisom} can be applied to obtain the following corollary:
\begin{corollary}
\label{udisom}
Given $b_{L^\pi}\in CF^{odd}_{G, +}(L^\pi)$, there exists bijections
    \begin{equation}
    \label{updowndef}
    \begin{tikzcd}
    CF^{odd}_{G, +}(L)
    \arrow[r,rightarrow, yshift=0.3ex, "b_{L^\pi}\circ (-)"] \arrow[r,leftarrow, yshift=-0.3ex, "(-)\circ b_{L^\pi}"']
& CF^{odd}_{can, +}(\bar{L}),
\end{tikzcd}
\end{equation}
$$ b_L \leftrightarrow b_{\bar{L}},$$
characterised by the equation $n^{eq, def}_{0,0,0}(\mathbf{1}_\infty)=0$, where $n^{eq, def}_{\cdot,\cdot,\cdot}=n^{eq, b_{\bar{L}}, b_L, b_{L^\pi}}_{\cdot,\cdot,\cdot}$.\\ 
Moreover, it induces the following pre-chain isomorphisms (up to a sign)
\begin{equation}
\label{updownprechisom}
(CF_{G}(L), m^{G, b_L}_1) \xrightarrow[\sim]{\phi_L} (CF_{eq}(\bar{L}; L, L^\pi), n^{eq, def}_{0,0,0}) \xleftarrow[\sim]{\phi_{\bar{L}}} (CF_{can}(\bar{L}), m^{b_{\bar{L}}}_1), 
\end{equation}
where $\phi_L = n^{eq, def}_{0,1,0}(\mathbf{1}_\infty; -)$; $\phi_{\bar{L}} = n^{eq, def}_{1,0,0}(-; \mathbf{1}_\infty)$.\\
Therefore, the composition $\kappa \coloneqq (\phi_{\bar{L}})^{-1} \circ \phi_L$ is a pre-chain isomorphism.

If in addition $b_{L^\pi}\in \widehat{MC}^G_{weak}(L^\pi)$, then (\ref{updowndef}) restricts to
    \begin{equation} \label{updownmcisom}
    \begin{tikzcd}
    \widehat{MC}^G_{weak}(L)
    \arrow[r,rightarrow, yshift=0.3ex, "b_{L^\pi}\circ (-)"] \arrow[r,leftarrow, yshift=-0.3ex, "(-)\circ b_{L^\pi}"']
& \widehat{MC}_{weak}(\bar{L}),
\end{tikzcd}
\end{equation}
$$ b_L \leftrightarrow b_{\bar{L}}$$
such that their potentials satisfy the following equation:
\begin{equation}\label{www}
    W^G_{L^\pi}(b_{L^\pi})+ W^G_L(b_L) = W_{\bar{L}}(b_{\bar{L}}).
\end{equation}
Furthermore, (\ref{updownmcisom}) descends to
    \begin{equation} \label{updownmcspaceisom}
    \begin{tikzcd}
    MC^G_{weak}(L)
    \arrow[r,rightarrow, yshift=0.3ex, "b_{L^\pi}\circ (-)"] \arrow[r,leftarrow, yshift=-0.3ex, "(-)\circ b_{L^\pi}"']
& MC_{weak}(\bar{L}),
\end{tikzcd}
\end{equation}
$$ [b_L] \leftrightarrow [b_{\bar{L}}]$$
which depends only on the gauge equivalence class $[b_{L^\pi}]\in MC^G_{weak}(L^\pi)$.

Moreover, (\ref{updownprechisom}) are chain isomorphisms (up to a sign), which induces the following isomorphisms of Floer cohomologies as gapped $\Lambda_0$-modules
\begin{equation}
HF_{G}(L, m^{G, b_L}_1) \xrightarrow[\sim]{[\phi_L]} HF_{eq}(\bar{L}; L, L^\pi, n^{eq, def}_{0,0,0}) \xleftarrow[\sim]{[\phi_{\bar{L}}]} HF(\bar{L}, m^{b_{\bar{L}}}_1).
\end{equation}
In this case, $\kappa$ is a chain isomorphism, which descends to Floer cohomology as an algebra isomorphism (up to a sign) 
\begin{equation}
    [\kappa]: (HF_{G}(L, m^{G, b_L}_1), [m^{G, b_L}_2], [e^G_L]) \rightarrow (HF(\bar{L}, m^{b_{\bar{L}}}_1), [m^{b_{\bar{L}}}_2], [e_{\bar{L}}]).
\end{equation}
\end{corollary}

On the other hand, by Remark \ref{equivcorrmodoverequivmorsemodel}, we could replace both $CF_G(L)$ and $CF_G(L^\pi)$ by Morse models $CF^{Morse}_{G}(L)$ and $CF^{Morse}_{G}(L^\pi)$. Hence we consider 
$$(CF_{eq}(\bar{L}; L, L^\pi), \{\tilde{n}^{eq}_{k'', k', k}\})$$ 
as a $\mathbb{G}$-gapped unital left $CF_{can}(\bar{L})$, right $(CF^{Morse}_{G}(L), CF^{Morse}_{G}(L^\pi))$ $A_\infty$ tri-module. By deforming $\{\tilde{n}^{eq}_{k'', k', k}\}$ by weak bounding cochains, we get:
\begin{prop} \label{equivobstr}
    Assume further that $L,L^\pi,\bar{L}$ has minimal Maslov index $0$ and are weakly unobstructed, then for any $b_L\in MC_{weak}(L), b_{L^\pi}\in MC_{weak}(L^\pi), b_{\bar{L}}\in MC_{weak}(\bar{L}), y\in CF_{eq}(\bar{L}; L, L^\pi)$, we have
\begin{align*}
        &\tilde{n}^{eq, def}_{0,0,0}(\tilde{n}^{eq, def}_{0,0,0}(y)) + (W_{\bar{L}}(b_{\bar{L}}) - W_{L}(b_{L}) - W_{L^\pi}(b_{L^\pi})) \\
        &+ (-1)^{||y||}\sum_{i=1}^{k}\lambda_i(h^i_{L}(b_L)+ h^i_{L^\pi}(b_{L^\pi})))\cdot y = 0.
    \end{align*}
    %\begin{align*}
    %    &n^{eq, def}_{0,0,0}(n^{eq, def}_{0,0,0}(y)) + (W_{\bar{L}}(b_{\bar{L}}) - W_{L}(b_{L}) - W_{L^\pi}(b_{L^\pi}))y \\
    %    &+ \sum_{i=1}^{k}(h^i_{L}(b_L)\cdot n^{eq, def}_{0,1,0}(y;  \boldsymbol{\lambda}^L_i) + h^i_{L^\pi}(b_{L^\pi})\cdot n^{eq, def}_{0,0,1}(y; \boldsymbol{\lambda}^{L^\pi}_i)) = 0
    %\end{align*}
    %where $\boldsymbol{\lambda}^L_i = e_L\otimes \lambda_i\in CF^{Morse}_{G}(L)$, and similarly for $\boldsymbol{\lambda}^{L^\pi}_i$.
\end{prop}
\begin{proof}
    From the first $A_\infty$ relation associated to the deformed tri-module \\
    $(CF_{eq}(\bar{L}; L, L^\pi), \{\tilde{n}^{eq, def}_{k'', k', k}\})$ with input $y$, we have 
    \begin{align*}
        &\tilde{n}^{eq, def}_{0,0,0}(\tilde{n}^{eq, def}_{0,0,0}(y)) + (W_{\bar{L}}(b_{\bar{L}}) - W_{L}(b_{L}) - W_{L^\pi}(b_{L^\pi}) \\
        &+\sum_{i=1}^{k}(h^i_{L}(b_L)\cdot \tilde{n}^{eq, def}_{0,1,0}(y;  \boldsymbol{\lambda}^L_i) + h^i_{L^\pi}(b_{L^\pi})\cdot \tilde{n}^{eq, def}_{0,0,1}(y; \boldsymbol{\lambda}^{L^\pi}_i)) = 0,
    \end{align*}
    where $\boldsymbol{\lambda}^L_i = e_L\otimes \lambda_i\in CF^{Morse}_{G}(L)$, and similarly for $\boldsymbol{\lambda}^{L^\pi}_i$. The statement follows from the (partial) unitality of $\boldsymbol{\lambda}^L_i$ and $\boldsymbol{\lambda}^{L^\pi}_i$ in the sense of \cite[Section 3.2]{KLZ} (for $A_\infty$ algebra case). Namely, for each $i$, we have
    $$\tilde{n}^{eq, def}_{0,1,0}(y;  \boldsymbol{\lambda}^L_i) = \lambda_i \cdot y = \tilde{n}^{eq, def}_{0,0,1}(y; \boldsymbol{\lambda}^{L^\pi}_i)$$
    in which the equality on the left is proved by (for each approximation space $L_l$ with $l>0$) identifying non-constant pseudo-holomorphic quilted drums whose image under $ev_{L_l}$ lie in $p^{-1}_l(D_{\lambda_i})$ and quilted drums whose projection onto $BG_l$ lie in $D_{\lambda_i}$, where $D_{\lambda_i}\subseteq BG_l$ is the Poincare dual submanifold of $\lambda_i\in H^2(BG_l)$. Similarly for the equality on the right, by replacing $L_l$ with $L^\pi_l$ in the above argument.
\end{proof}
Combining the above proposition with the left cyclic property of $\mathbf{1}_\infty$ (with respect to $\{\tilde{n}^{eq}_{k'', k', k}\}$) yields the following corollary:

\begin{corollary}
\label{equivmorseud} Consider $(h_{L^\pi}+h_{L})^{-1}(0) \subseteq MC_{weak}(L) \times MC_{weak}(L^\pi)$, 
there exists a map 
$$(h_{L^\pi}+h_{L})^{-1}(0)\xrightarrow{\circ} MC_{weak}(\bar{L}),  $$
$$(b_{L^\pi}, b_{L}) \mapsto b_{\bar{L}} \coloneqq b_{L^\pi} \circ b_{L}$$
    characterised by the equation $\tilde{n}^{eq, def}_{0,0,0}(\mathbf{1}_\infty)=0$. Their potentials satisfy
    \begin{equation} \label{eq:WWW}
    W_{L^\pi}(b_{L^\pi}) + W_{L}(b_{L}) = W_{b_{\bar{L}}}(b_{\bar{L}}).
     \end{equation}
\end{corollary}

\begin{proof}
    Apply Proposition \ref{equivobstr} with $y = \mathbf{1}_\infty$.
\end{proof}

\begin{remark}
    Under the following Künneth formula of weak Maurer-Cartan spaces due to Amorim \cite{Amorim} (see also \cite{Fukaya-corr}),
    $$MC_{weak}(L) \times MC_{weak}(L^\pi) \cong MC_{weak}(L \times L^\pi),$$
$(h_{L^\pi}+h_{L})^{-1}(0)$ can be identified with $(h_{L \times L^\pi})^{-1}(0)$, hence is non-empty.  
\end{remark}

\subsection{Application to a conjecture of Teleman}
\label{subsec-telconj}

In this section, we prove a conjecture of Teleman in \cite{Teleman} using equivariant Lagrangian Floer theory. For simplicity, we restrict ourselves to the abelian case $G = T = U(1)^k$.

\begin{conjecture} \label{telconj}
    Given a Hamiltonian $T$ space $((Y, \omega_Y), T, \mu_Y)$, there exists a  ``mirror holomorphic fibration" $F: \check{Y}\rightarrow \check{T}_\C$, where $\check{Y}$ is a mirror of $Y$,
    such that for each $c\in\mathfrak{t}^*$ with $T$ acting freely on $\mu^{-1}_Y(c)$ with smooth symplectic quotient $X$, there exists $q = q(c) \in \check{T}_\C$ such that $\check{X} \coloneqq F^{-1}(q)$ is a mirror of $X$.
    
    Also, under the Landau-Ginzburg (LG) Mirror Symmetry, if $(\check{Y}, W_Y)$ is an LG mirror of $Y$, then $(\check{X}, W_X) \coloneqq (F^{-1}(q), W_Y|_{F^{-1}(q)})$ is an LG mirror of $X$.
\end{conjecture}

We proves a localised version of Conjecture \ref{telconj} as follows:

\begin{theorem} \label{telthm}
    In the context of Conjecture \ref{telconj} and the Setup \ref{hamlagsmred}, assuming further that $L$ has nonnegative minimal Maslov index, we define the following:

\begin{itemize}
    \item $c = 0$ and $(\check{X}, W_X) = (MC_{weak}(\bar{L}), W_{\bar{L}})$ is the ``localised mirror space of $\bar{L}$" with disc potential $W_{\bar{L}}$.
    \item $(\check{Y}, W_Y) = (MC^{Morse}_{T}(L), W^{Morse}_{L, T})$ is the ``localised $T$-equivariant mirror space of $L$" with equivariant disc potential $W^{Morse}_{L, T}$.
    \item $F_L = exp(h_L): MC^{Morse}_{T}(L) \rightarrow (\Lambda^\times)^k$, where 
    $$h_L = (h^i_L)^k_{i=1}: MC^{Morse}_{T}(L) \rightarrow H^2_T(pt; \Lambda_0) \cong \bigoplus_{i=1}^{k}\Lambda_0 \lambda_i \cong \Lambda^k_0$$ 
    is the equivariant part of $W^{Morse}_{L, T}$, i.e. 
    $$W^{Morse}_{L, T}(b) = W_L(b) + \displaystyle{\sum_{i=1}^k}\lambda_i h^i_L(b).$$
    %\item $(\check{Y}, W) = (MC_{weak}(L), W_L)$ is the ``localised mirror space of $L$" 
    %\item $F_L = exp(h_L): MC_{weak}(L) \rightarrow (\Lambda^\times)^r$, where $$h_L: MC_{weak}(L) \subseteq MC_{weak}(CF_{T, Morse}(L)) \rightarrow H^2_T(pt; \Lambda_0) \cong \Lambda^r_0$$ is the equivariant part of the equivariant disc potential $W^{Morse}_{L, T}$ defined on $MC_{weak}(CF_{T, Morse}(L))$ restricted to $MC_{weak}(L)$, i.e. 
    %$$W^{Morse}_{L, T}(b) = W_L(b) + \displaystyle{\sum_{i=1}^r}\lambda_i h^i_L(b)$$
    \item Replacing $q = 1 \in (\Lambda^\times)^k$ by $\log q\coloneqq 0 \in \Lambda^k_0$, and $F^{-1}_L(1)$ by $h^{-1}_L(0)$.
\end{itemize}
Assume $MC^T(L^\pi) \neq \phi$, then Conjecture \ref{telconj} holds in the following sense:
$$(h^{-1}_L(0), W^{Morse}_{L, T}|_{h^{-1}_L(0)}) \cong (\check{X}, W_X).$$
\end{theorem}

\begin{remark}
    The adjective ``localised" refers to the situation that we are considering   ``localised mirror spaces", studied in great details by Cho, Hong and the first-named author in a series of works \cite{CHL, CHL-nc, CHL-toric, HL}. In a forth-coming work, we will ``globalised" Theorem \ref{telthm} after gluing the corresponding localised mirrors and fibrations of various $L$, using the techniques in \cite{CHL3}.
\end{remark}

\begin{proof}
Note that by definition of $h_L$, we have 
$$(h^{-1}_L(0), W^{Morse}_{L, T}|_{h^{-1}_L(0)}) \cong  (MC_{weak}(CF_{T, Morse}(L)) \times_{H_T(pt; \Lambda_0)} \Lambda_0, pr_2),$$

where by Corollary \ref{equivcfmcisom}, 
$$(MC_{weak}(CF_{T, Morse}(L)) \times_{H_T(pt)} \Lambda_0, pr_2) \cong (MC_{weak}(H^\bullet_T(L; \Lambda_0); \Lambda_+), W^T_L).$$

Finally, since $MC^T(L^\pi) \neq \phi$, for any $b_{L^\pi}\in MC^T(L^\pi)$, (\ref{updownmcisom}) implies 
$$(MC_{weak}(H^\bullet_T(L; \Lambda_0(\R)); \Lambda_+(\R)), W^T_L) \cong (MC_{weak}(\bar{L}), W_{\bar{L}}),$$
where $(MC_{weak}(\bar{L}), W_{\bar{L}}) = (\check{X}, W_X)$ by definition.
\end{proof}

 From the proof, we actually have refined (the localised version of) Conjecture \ref{telconj} with the same proof: replacing the last assumption $MC^T(L^\pi) \neq \phi$ by $MC^T_{weak}(L^\pi) \neq \phi$, then for each $b_{L^\pi}\in MC^T_{weak}(L^\pi)$ we still have $h^{-1}_L(0)\cong \check{X}$, but (\ref{www}) implies that their potentials differ by the potential of $L^\pi$ in general, i.e. for any pair $(b_L, b_{\bar{L}})$ under the bijection  $h^{-1}_L(0)\cong \check{X}$, we have
 \begin{equation} \label{eq:W+W=W}
     W^{Morse}_{L, T}(b_L) + W^T_{L^\pi}(b_{L^\pi}) = W_{\bar{L}}(b_{\bar{L}}).
 \end{equation}

Even more generally, for each $b_{L^\pi} \in MC_{weak}(L^\pi)$, we replace $h^{-1}_L(0)$ above by $h^{-1}_L(-h_{L^\pi}(b_{L^\pi}))$. By Corollary \ref{equivmorseud}, the same conclusion holds.

%\begin{remark}
%    Careful reader may have realised that in the above argument, there are two $A_\infty$ algebra structure on $H_G(L; \Lambda_0)$ involved: one comes from a canonical model of the equivariant de Rham model $ \varprojlim  \Omega(L_{l}; \Lambda_0)$; another one is obtained from a canonical model of the equivariant Morse model $C(f)\otimes_\R H^\bullet_G(pt; \Lambda_0)$, whose $A_\infty$ structure comes from the singular cochain model (See e.g. \cite{KLZ} for details). \\
    
%    We expect that these two $A_\infty$ algebra structure on $H_G(L; \Lambda_0)$ are $A_\infty$ homotopic, which will in turn justify the above argument. A possible solution is to ``unify" the chain models in the beginning: for instance, one can use de Rham model and perform a de Rham-to-Morse Homological Perturbation (using Witten-Morse contractions defined in Corollary \ref{wittenmorsecontr}), or use singular cochain models in constructing (equivariant) Correspondence tri-modules. 
%\end{remark}

\subsection{A discussion about singular moment levels}
\label{singlevel}
In this informal subsection, we discuss potential generalisations to singular symplectic quotients.

In \cite{LS}, Lekili and Segal conjectured about an equivalence between the wrapped Fukaya category of a symplectic $T$-quotient $X$ at a singular level and the wrapped Fukaya category of a spectral component of the $T$-equivariant Fukaya category $Y$ for a torus $T$.  (We have switched $X$ and $Y$ to match the notations of this paper.) In this formulation, the spectral component was defined using $\mathcal{CO}^0(s)$, where $s$ is the Seidel element \cite{Ritter} in the symplectic cohomology associated to a torus action, and $\mathcal{CO}^0$ denotes the zeroth order part of the closed-open map associated to a Lagrangian.  In this subsection, we discuss this situation in the perspective of SYZ and equivariant disc potentials.  

In previous sections, we assumed a regular central value $c \in \mathfrak{g}^*$ of the moment map $\mu$ such that $\mu^{-1}\{c\}$ gives a smooth Lagrangian correspondence $L^\pi$.  Moreover, we assume that the Lagrangian correspondence to be weakly unobstructed, so that it has a well-defined equivariant disc potential $W_{L^\pi} + \lambda \cdot h_{L^\pi}$.  Then we consider weakly unobstructed $G$-Lagrangians $L \subset \mu^{-1}\{c\} \subset Y$ whose equivariant part of the disc potential satisfies $h_L = -h_{L^\pi}$.  By Corollary \ref{equivmorseud}, 
the equivariant Lagrangian correspondence sends such $L$ with potential value $W_L$ to its quotient $\bar{L}$ (with boundary deformations $b_{\bar{L}}$ on $\bar{L}$) that has potential value $W_{\bar{L}} - W_{L^\pi}$, and it induces isomorphisms on their Floer cohomologies (Corollary \ref{udisom}).  Thus, the localized mirror of the quotient $\bar{L}$ is given by a fiber of $h_L$ on the localized mirror of $L$.

Similarly, for a pair of such Lagrangians $L_1$ and $L_2$ with the same potential value $W$, we can apply the theory to $L_1 \cup L_2$ to obtain the corresponding boundary deformed quotient objects $\bar{L}_1,\bar{L}_2$ with potential value $W - W_{L^\pi}$, and the Floer cohomology for $(L_1,L_2)$ is isomorphic to that for $(\bar{L}_1,\bar{L}_2)$.

When $c$ is a singular value, we can still take weakly unobstructed $G$ -Lagrangians $L \subset \mu^{-1}\{c\} \subset Y$ which have a well-defined equivariant disc potential.  However, $\mu^{-1}\{c\}$ and the symplectic quotient are singular.  

Let $Y^\circ$ be the complement of the set of singular points of $\mu$ in $Y$, and $X^\circ$ the complement of the set of singular points in $X$.  Then we have the non-compact submanifolds $(L^\pi)^\circ = L^\pi \cap (Y^\circ \times X^\circ)$ that serves as Lagrangian correspondence.  Moreover, for weakly unobstructed $G$-Lagrangians $L \subset \mu^{-1}\{c\} \subset Y$, we consider $\bar{L}^\circ = \bar{L} \cap X^\circ$ for its quotient $\bar{L}$.

Assuming that $(L^\pi)^\circ$ and $\bar{L}^\circ$ can be defined as objects in certain wrapped Fukaya categories for $Y^\circ \times X^\circ$ and $X^\circ$ respectively, one can consider the correspondence tri-module for $(L, (L^\pi)^\circ, \bar{L}^\circ)$ using the work of \cite{gao}.  In our formulation, we take the corresponding Borel spaces and consider their equivariant Floer theories.

The algebraic structures are similar.  We expect that $(L^\pi)^\circ$ and $\bar{L}^\circ$ still have well-defined equivariant disc potentials, and Corollary \ref{udisom} and \ref{equivmorseud} on the relation of their equivariant disc potentials and Floer cohomologies still holds.  Then the localized mirror of the quotient $\bar{L}$ is again given by a fiber of $h_L$ on the localized mirror of $L$.

The following example provides evidence for the expected statement on localized mirrors.

\begin{example}
    Consider the $\bS^1$-action on $Y = \C^2 - \{ab = 1\}$ by $\zeta \cdot (a,b) \mapsto (\zeta a, \zeta^{-1} b)$.  It has the moment map $\mu = |a|^2 - |b|^2$, and the level at $0$ is singular.  The quotient at level $0$ can be identified with $X = \C - \{1\}$ by the invariant function $ab$, whose reduced symplectic structure is singular at $0 \in \C$.  We consider the pair-of-pants $X^\circ = \C - \{0,1\}$.

    Consider the immersed Lagrangian sphere $L \subset \mu^{-1}\{0\}$ which is the preimage of the unit circle $\bar{L} \subset \C-\{1\}$ centered at $1$.  Note that $\bar{L}$ passes through the singular point $0$.  $\bar{L}^\circ = \bar{L} - \{0\} \subset X^\circ$ is considered as an object in the wrapped Fukaya category.
    
    By \cite[Theorem 5.8]{KLZ}, the equivariant disc potential of $L$ equals $h_L = \log (1-uv)$.  Then
    $h_L^{-1}\{0\} = \{uv = 0\}$
    is a singular conic, which is mirror to the pair-of-pants $X^\circ$ by \cite{Jeffs}.
\end{example}

\section{Obstructions in toric Lagrangian correspondence} \label{sec:toric}

Even in simple toric examples,  Floer-theoretical obstructions can arise for equivariant Lagrangian correspondence. 
In general, one approach to remove these obstructions is to apply bulk deformations on the ambient space $Y_G \times X$. We will further discuss this approach in a future work.

In this section, we consider some vanishing conditions of the obstruction. For instance, if we assume that $X$ and $Y$ are monotone (with the same monotonicity constants), and that the Lagrangian correspondence $L^\pi \subset Y^-\times X$ is also monotone, then the nonzero-degree terms of the obstruction vanish automatically, since the only discs with non-positive Maslov index are constant.
%However, it is common that a monotone moment-map level is singular (namely $L^\pi$ does not receive a free action and the quotient space $X$ is singular).

The following proposition provides another instance of vanishing conditions.

\begin{prop} \label{prop:unobs}
    Under the Setup \ref{hamsmred}, with the following additional assumptions:
    \begin{enumerate}
        \item $Y$ and $X$ are equipped with Fano almost complex structures, in the sense that $c_1(\alpha)>0$ for every nonzero effective curve class $\alpha$. 
        \item The quotient correspondence $L^\pi$ (defined in Setup \ref{hamlagsmred}) satisfies $H_1(L^\pi)=0$, so that $\phi: H_2(Y^-\times X) \rightarrow H_2(Y^-\times X, L^\pi)$ is surjective. 
        \item $\phi$ restricts to a surjection $\phi: H_2^\eff(Y^-\times X) \rightarrow H_2^\eff(Y^-\times X, L^\pi)$, where $H_2^\eff(Y^-\times X)$ (resp. $H_2^\eff(Y^-\times X, L^\pi)$) is the cone generated by classes of rational curves (resp. nodal unions of holomorphic discs with holomorphic spheres).
    \end{enumerate}
    Then $L^\pi$ is weakly unobstructed.

% Moreover, assume that all effective disc classes of $L^\pi$ come from effective curve classes, namely, $H_2^\eff(Y^-\times X) \subset H_2(Y^-\times X)$ (which consists of classes represented by rational curves) is mapped onto $H_2^\eff(Y^-\times X, L^\pi) \subset H_2(Y^-\times X, L^\pi)$ (which consists of classes represented by nodal unions of holomorphic discs with holomorphic spheres) under the surjection.

\end{prop}
\begin{proof}
    By assumption (3), every effective disc class bounded by $L^\pi$ is an effective sphere class.  By assumption (1), the Maslov index of a non-constant stable disc is always positive.  By taking a canonical model and the fact that $H^0(L^\pi)$ is one-dimensional, it follows that $m_0^b(L^\pi)$ is proportional to the unit for any degree-one boundary deformations $b$.
\end{proof}

In this section, we use toric methods to construct geometries in which the obstruction of the equivariant Lagrangian correspondence $L^\pi_T$ vanishes.  For instance,  we shall show that when both $Y$ and $X$ are Fano, and when the moment level set hit all the toric divisors of $Y$, vanishing of $m_{0}^G(L^\pi)$ holds.  Moreover, we will demonstrate that even in basic toric Fano cases of $Y$, obstructions are present.  Furthermore, we will deduce a relation between equivariant obstructions and mirror maps for compact toric semi-Fano manifolds.

Let's quickly recall the toric setup.  A toric variety $Y$ is a symplectic quotient of $\C^N$ by a torus $T$, which embeds in $T^N$ that acts on $\C^N$ by coordinate-wise multiplication.  Let $n$ be the complex dimension of $Y$.  Then $Y$ has a residual Hamiltonian $T^n$ action, whose moment map image is a polytope $P$.  The (closure of) inverse images of the codimension-one boundaries of $P$ are called toric divisors in $Y$.  Let $L \subset Y$ be a regular Lagrangian torus fiber of $\mu$. Throughout the whole section, $L$ is endowed with the standard $T^n$-equivariant spin structure induced from the left trivialisation $TL \cong L \times \mathfrak{t^n}$ (as in \cite{CO}), hence the background datum $V$ of relative spin structure is trivial. 

The foundational work of Cho-Oh \cite{CO} classified all holomorphic discs bounded by a regular toric fiber in $Y$.  In particular, the disc classes of $(Y,L)$ are generated by the basic disc classes $\beta_i$ emanated from the toric divisors $D_i$ for $i=1,\ldots, m$.  We denote by $v_i$ the corresponding primitive integer vectors in $\mathfrak{t}^n$.

A toric variety has a meromorphic volume form $\Omega$ which has simple poles along the toric divisors.  The regular toric fibers $L$ are special with respect to $\Omega$, which means that $\iota^*_L \mathrm{Im}(\Omega) = 0$.  Since the basic disc classes $\beta_i$ intersects the toric divisors (which are simple poles of $\Omega$) exactly once, each $\beta_i$ has Maslov index two.  We refer to \cite{CO,auroux07} for details.

We consider a subtorus $T^k \subset T^n$ acting on $Y$.  Let's denote the corresponding moment map by $\mu:Y \to \R^k$.  Then, we have a toric quotient $X = Y\sslash_c T^k = \mu^{-1}\{c\}/T^k$, which is assumed to be smooth.  The moment level set $\mu^{-1}\{c\} \subset Y$ induces a Lagrangian correspondence $L^\pi \subset Y^- \times X$ for the symplectic quotient.  (Recall that $Y^-$ denotes the symplectic manifold $(Y,-\omega)$.  $L^\pi := \{(y,[y]) \in Y^-\times X: y\in \mu^{-1}\{c\}\}$.) 

The following short exact sequence is useful in describing the topology of $(X,\bar{L})$ for a regular toric fiber $\bar{L} \subset X$:
\begin{equation} \label{toric-exact}
    0 \to K \to \Z^n \to N \to 0
\end{equation}
where $K = H_2(X,\Z)$, $N = H_1(\bar{L},\Z)$ and $H_2(X,\bar{L}) \cong \Z^n$.  $N$ is the lattice whose induced vector space supports the fan picture of $X$.

Let's make some topological preparations by describing disc classes in $\pi_2(Y^-\times X, L^\pi)$.

\begin{lemma} \label{lem:disc-moment}
    Let $Y$ be a toric manifold and $X$ a symplectic quotient with respect to a subtorus action.
    For the level set $\mu^{-1}\{c\}\subset Y$, let $I$ be the subset of indices $i\in 1,\ldots,m$ that satisfies $D_i \cap \mu^{-1}\{c\} = \emptyset$, where $D_i$ denote the toric prime divisors of $Y$.
    Let $\beta_j$, $j=1,\ldots,m$, be the basic disc classes of a regular toric fiber of $Y$.  Then $\pi_2(Y,\mu^{-1}\{c\})$ is generated by $\{\beta_i: i \in I\}$.  %Moreover, every holomorphic disc class is a non-negative integer combination of $\beta_i$ for $i \in I$.
\end{lemma}
\begin{proof}
    A disc $(\Delta,\partial\Delta) \to (Y,\mu^{-1}\{c\})$ is homotopic to one whose boundary lies in a regular toric fiber of $Y$.  Thus, its homotopy class is an integer combination of $\beta_i$.  Moreover, $\beta_i$ is zero as a class in $\pi_2(Y,\mu^{-1}\{c\})$ if and only if $D_i \cap \mu^{-1}\{c\} \not= \emptyset$.  Thus $\pi_2(Y,\mu^{-1}\{c\})$ is generated by $\{\beta_i: i \in I\}$.
\end{proof}

%For $i \in I$, since $D_i \cap \mu^{-1}\{c\} = \emptyset$, the intersection number of $D_i$ with a disc class in $\pi_2(Y,\mu^{-1}\{c\})$ is well-defined.  
    %For a holomorphic disc class, the intersection number with the holomorphic divisors $D_i$ must be non-negative for all $i\in I$.  Thus it is a non-negative integer combination of $\beta_i$ for $i \in I$.

\begin{lemma} \label{lem:dics_class}
    Let $Y$ be a toric manifold and $X$ a compact symplectic quotient with respect to a subtorus action.
    We have the short exact sequence 
    $$ 0 \to \pi_2(X) \to \pi_2(Y^- \times X,L^\pi) \to \pi_2(Y,\mu^{-1}\{c\}) \to 0.$$
\end{lemma}

\begin{proof}
    We have the projection map $\pi_2(Y^- \times X,L^\pi) \to \pi_2(Y^-,\mu^{-1}\{c\})$.  Let's consider the kernel of this map.  Suppose an element in $\pi_2(Y^- \times X,L^\pi)$ is projected to the zero class in $\pi_2(Y^-,\mu^{-1}\{c\})$.  In particular, the boundary loop is homotopic to zero in $\mu^{-1}\{c\} \cong L^\pi$.  Thus the corresponding class in $\pi_2(Y^- \times X,L^\pi)$ is homotopic to a sphere class of $Y \times X$.  Since it projects to zero class in $Y$, it is a sphere class in $X$.

    By the previous lemma, a disc class in $\pi_2(Y^-,\mu^{-1}\{c\})$ is a linear combination of the basic disc classes $\beta_i$ for $i \in I$.  Consider the boundary of a basic disc class $\beta_i$ and $D_i \cap \mu^{-1}\{c\} = \emptyset$.  Since $X$ is compact toric, $\pi_1(X)=0$, and so the image of $\partial \beta_i$ in $X$ bounds a disc class in $X$.  Thus, we can lift any such $\beta_i$ to a disc class in $\pi_2(Y^- \times X,L^\pi)$.  Thus, the stated short exact sequence holds.
\end{proof}

\begin{theorem}[Vanishing of obstructions in Lagrangian correspondence] \label{thm:m0=0}
    Let $Y=\C^n$ and $X$ be a toric Fano symplectic quotient of $Y$.  Suppose that the moment-map level set $\mu^{-1}\{c\}$ for $X$ intersects all the toric divisors of $Y$.  Then the equivariant disc potential of the quotient Lagrangian correspondence $L^\pi \cong \mu^{-1}\{c\}$ vanishes.
\end{theorem}
\begin{proof}
    By the condition that the moment level set intersects all the toric divisors of $Y$, $H^1(\mu^{-1}\{c\})=0$.
    We want to show that the Maslov indices of all nonconstant stable discs bounded by $L^\pi$ are positive. This implies that $L^\pi$ is weakly unobstructed.  Also, the equivariant part which is contributed by Maslov-zero disc classes also vanishes.

    By Proposition \ref{prop:unobs}, it suffices to show that assumption (3) holds, i.e. the effective disc classes $\beta$ bounded by $L^\pi$ coincide with the effective curve classes of $Y^-\times X$, or equivalently those of $X$ as $Y = \C^n$. Suppose $\beta\not=0$ corresponds to a curve class $\alpha$ of $X$ which is not effective. Then there exists a toric K\"ahler form $\omega$ of $X$ such that $([\omega],\alpha) \leq 0$. Such $\omega$ can be lifted to a toric K\"ahler form of $Y$ whose symplectic quotient gives $(X,\omega)$. For the corresponding symplectic form $\hat{\omega}$ on $Y^-\times X$, we have $([\hat{\omega}],\beta) \leq 0$. Thus $\beta$ cannot be an effective disc class.
    
    In the situation that the moment level set $\mu^{-1}\{c\}$ intersects all the toric divisors, by the classification of holomorphic discs of Cho-Oh \cite{CO}, there is a one-to-one correspondence between holomorphic discs of Maslov index two of $(Y,L)$ and $(X,\bar{L})$.  Thus, the disc potentials $W_L(b_L)$ and $W_{\bar{L}}(b_{\bar{L}})$ exactly coincide under the restriction map in Theorem \ref{telthm}.  By Equation \eqref{eq:W+W=W}, the non-equivariant part of the potential of $L^\pi$ vanishes.
\end{proof}

\begin{example}
    Consider $Y=\C^n$ and suppose $X=Y\sslash_c T^k$ is a compact Fano toric manifold, where $c$ is a generic value in the moment map image and $Y$ comes from the middle term of the toric exact sequence \eqref{toric-exact} of $X$ as $\C^n = \Z^n \otimes \C$.  Let $l_j=c_j$ for $j=1,\ldots,k$ be the defining affine linear equations of the moment level set, where $l_j$ correspond to a basis of the subtorus $T^k$ and $c=(c_1,\ldots,c_k)\in \R^k$ is a constant vector.  Let $L$ be a regular toric fiber of $\C^n$ whose quotient is a regular toric fiber $\bar{L} \subset X$.

    By \cite{KLZ}, the equivariant disc potential of $\C^n$ equals
    $$\textstyle \bT^{A_1} e^{x_1} + \ldots + \bT^{A_n} e^{x_n} + \sum_j \lambda_j l_j(x_1,\ldots,x_n) $$
    where $A_i$ are the symplectic areas of the basic disc classes $\beta_i$ bounded by the toric fiber $L\subset \C^n$.
    By Theorem \ref{telthm}, the disc potential of $\bar{L}$ equals the restriction of $\bT^{A_1} e^{x_1} + \ldots + \bT^{A_n} e^{x_n}$ on $\{l_j(x_1,\ldots,x_n)=0 \textrm{ for all } j=1,\ldots,k\}$, where $x_i$ are the (complex) boundary deformation parameters in $H^1(L)$ and $A_i$ are some positive real numbers.
\end{example}

Next, we consider the semi-Fano situation $c_1 \geq 0$ for all curve classes.

\begin{prop} \label{prop:semi-Fano}
    Let $Y=\C^n$ and $X$ be a toric semi-Fano symplectic quotient of $Y$ (meaning that all curve classes have non-negative first Chern number $c_1$).  Suppose that the moment-map level set $\mu^{-1}\{c\}$ for $X$ intersects all the toric divisors of $Y$. Then the (non-equivariant) obstruction $m_0^{L^\pi}$ vanishes.    
\end{prop}
\begin{proof}
    As in the proof of Theorem \ref{thm:m0=0}, holomorphic disc classes bounded by $L^\pi$ coincide with rational curve classes of $X$.  By the semi-Fano condition, all disc classes have Maslov indices $\geq 0$.  Thus $m_0(L^\pi)$ has degree $\leq 2$.
    Moreover, $H^2(L^\pi)=0$: By using the $(\C^\times)^n$-action on $\C^n$, any two-cycle of $\mu^{-1}\{c\}$ is homologous to a two-cycle supported in the intersection of $\mu^{-1}\{c\}$ with a coordinate plane $\C^2 \subset\C^n$, which does not support $H^2$.  Hence the two-cycle must be homologous to zero.
    
More generally, one can show that $\pi_1(L^\pi) = 0 = \pi_2(L^\pi)$ as follows: since $\mu^{-1}\{c\}$ is homotopically equivalent to the stable locus $U$ of the $(\C^\times)^k$-action on $\C^n$, it suffices to show that $\pi_1(U) = 0 = \pi_2(U)$. The former equality follows from $\mathrm{codim}_\C(U\subseteq \C^n)\geq 2$; the latter equality follows from the long exact sequence of homotopy groups associated to $U\rightarrow X$ as a principal $(\C^\times)^k$-bundle.
    
    Thus, there is no degree two element in the canonical model.  Hence, the non-equivariant $m_0(L^\pi)$ vanishes.
    %By using induced action on $L^\pi$ from the $(\C^\times)^n$-action on $\C^n$, any two-cycle is homologous to a two-cycle supported in the intersection of $L^\pi$ with a coordinate plane $\C^2 \subset\C^n$
\end{proof}

In the above proposition, even though $L^\pi$ has zero non-equivariant disc potential, holomorphic curves with Chern number zero can still contribute to the equivariant disc potential of $L^\pi$ and lead to quantum corrections for the comparison between $L\subset Y$ and $\bar{L} \subset X$.  Indeed, these are crucial for the disc potentials of semi-Fano toric manifolds.

In the joint work of the first and second named author with Chan and Tseng \cite{CLLT}, the disc potential of a regular toric fiber of a compact semi-Fano toric manifold was computed and expressed in terms of the (inverse) mirror map.  The mirror map is given by hypergeometric functions that are solutions to a certain Picard-Fuchs system of differential equations. 

\begin{theorem}[\cite{CLLT}]
    The disc potential of a regular toric fiber of a compact semi-Fano toric manifold equals
    $W_{\bar{L}} = \sum_{l=1}^{n} \exp(g_l(\check{q}(q))) \, Z_l,$
    where
    \begin{equation*} \label{Eqn Z}
    Z_l = \left\{ \begin{array}{ll}
    z_l & \text{ when } l=1,\ldots,d;\\
    q_{l-d} z^{v_l}:=q_{l-d}\prod_{i=1}^d z_i^{(\nu_i, v_l)} & \text{ when } l=d+1,\ldots,n,
    \end{array}
    \right.
    \end{equation*}
    \begin{equation}\label{eqn:funcn_g}
    g_l(\check{q}):=\sum_{c}\frac{(-1)^{(D_l\cdot c)}(-(D_l\cdot c)-1)!}{\prod_{p\neq l} (D_p\cdot c)!}\check{q}^c
    \end{equation}
    and the summation is over all effective curve classes $c\in H_2^\text{eff}(X)$ satisfying
    $$-K_X\cdot c=0, D_l\cdot c<0 \text{ and } D_p\cdot c \geq 0 \text{ for all } p\neq l$$
    and $\check{q}=\check{q}(q)$ is the inverse of the mirror map $q=q(\check{q})$.
\end{theorem}

In the above theorem, we have fixed a maximal cone of the fan of $X$ spanned by a basis$\{v_1,\ldots,v_d\}$, whose dual basis is denoted by $\{\nu_1,\ldots,\nu_d\}$.  Then other vectors in the fan are expressed in terms of this basis: $v_l = \sum_{i=1}^d (\nu_i,v_l) v_i$.  We have the curve classes $\Psi_j \in H_2(X)$ for $j=1,\ldots,n-d$, which is the linear combination of basic disc classes $\beta_{d+j} - \sum_{i=1}^d (\nu_i,v_l) \beta_i$.  Their corresponding K\"ahler parameters are denoted by $q_j = T^{\omega \cdot \Psi_j}$.  Moreover, we have absorbed the Novikov coefficients into the variables $z_l$, namely, $z_l = T^{\omega\cdot\beta_l} e^{x_l}$.

In this case, even though the disc potential of $\C^n$ is simple, the disc potential of its toric quotient is highly non-trivial.  Using Proposition \ref{prop:semi-Fano}, we can now explain that the coefficients $\exp(g_l(\check{q}(q)))$ in terms of the equivariant Lagrangian correspondence $L^\pi_{T^k}$.

\begin{theorem} \label{thm:equiv-semi-Fano}
    Let $Y=\C^n$ and $X=Y\sslash_c T^k$ be a compact semi-Fano toric manifold, where $\C^n$ comes from the middle term of the toric exact sequence \eqref{toric-exact} of $X$.
    The equivariant disc potential of the Lagrangian correspondence $L^\pi$ equals
    $$\textstyle W^{Morse}_{L^\pi, T} = \sum_{j=1}^{n-d} \lambda_j (\log q_j - \log \check{q}_j(q))$$
    where $\check{q}_j(q)$ denotes the inverse mirror map for $X$.
\end{theorem}
\begin{proof}
    As in the proof of Proposition \ref{prop:semi-Fano}, $H^2(L^\pi)=0$. This implies the vanishing of the non-equivariant part of $W^{Morse}_{L^\pi, T}$. 
    Thus, $W^{Morse}_{L^\pi, T} = \sum_{j=1}^{n-d} \lambda_j h_j(q)$ for some $h_j$.  
    
    The $T^k$-action is in the direction of $E_l + \sum_{i=1}^d (\nu_i,v_l) E_i$ for $l=d+1,\ldots,n$, where $\{E_k:k=1\ldots,d\}$ denotes the standard basis.  By Equation \eqref{eq:WWW}, $W_L = \tilde{z}_1 + \ldots + \tilde{z}_n$ of $\C^n$ equals $W_{\bar{L}} = \sum_{l=1}^{n} \exp(g_l(\check{q}(q))) \, Z_l$ by restricting to $\log \tilde{z}_l + \sum_{i=1}^d (\nu_i, v_l) \log \tilde{z}_i + h_j = 0$ for $l=d+1,\ldots,n$.  Then the equality follows from the fact that the inverse mirror map is given by
    $ \check{q}_j(q) = q_j \prod_{l=1}^n (\exp g_l(\check{q}(q)))^{-D_l \cdot \Psi_j}, $
    where $\Psi_j$ is the curve class for the K\"ahler parameter $q_j$.
\end{proof}

\begin{example} \label{ex:F2}
    Let's consider the Hirzebruch surface $X=\bF_2$, which is obtained as a symplectic quotient of $Y=\C^4$ by the $T^2$-action generated by $(0,1,0,1)$ and $(1,2,1,0)$.  It is well known that $X$ is semi-Fano.  There are two generating curve classes, namely the fiber class $f$ and the exceptional curve class $e$ (that has self-intersection $(-2)$.  Let's denote their K\"ahler parameters by $q^f = T^{\omega\cdot f}$ and $q^e = T^{\omega\cdot e}$.  We have $\Psi_1=e+2f$ and $\Psi_2=f$, and so $q_1 = q^{e+2f}$ and $q_2 = q^f$.  We also have the corresponding mirror complex parameters $\check{q}_1=\check{q}^{e+2f}$ and $\check{q}_2=\check{q}^f$.

    The above $g_l$ is non-zero only when $l=4$:
     $g_4(\check{q}^e) = \sum_{j>0} \frac{(2j-1)!}{(j!)^2} \check{q}^j. $
    The mirror map is given by $q_1=\check{q}_1$ and 
    $$q_2 = \check{q}_2 \exp (-g_4(\check{q}^e)) = \check{q}_2 \exp (-g_4(\check{q}_1\check{q}_2^{-2})).$$
    Its inverse is given by $\check{q}_1=q_1,\,\check{q}_2 = q_2(1+q^e)$.  This gives the disc potential of a regular toric fiber of $X=\bF_2$ to be
    $$\textstyle z_1 + z_2 + \frac{q_1}{z_1z_2^2} + \frac{q_2(1+q^e)}{z_2}. $$
    (We have absorb some Novikov coefficients $T^A$ into $z_1,z_2$.)

    The equivariant disc potential of a regular toric fiber $L \subset \C^4$ equals $T^a \tilde{z}_1 + T^b \tilde{z}_2+ T^c \tilde{z}_3+ T^d \tilde{z}_4 + \lambda_1 \log \left(\tilde{z}_1\tilde{z}_2^2\tilde{z}_3\right) + \lambda_2 \log \left( \tilde{z}_2\tilde{z}_4\right)$ for some $a,b,c,d>0$, where $a+2b+c = \omega \cdot (e+2f)$ and $b+d = \omega \cdot f$ give the symplectic areas of the curve classes downstairs.
    Let $W^{Morse}_{L^\pi, T} = \lambda_1 h_1(q_1,q_2) + \lambda_2 h_2(q_1,q_2)$ be the equivariant disc potential of $L^\pi$.  We require vanishing of the equivariant part of $W^{Morse}_{L, T} + W^{Morse}_{L^\pi, T}$ to obtain the disc potential of $W_{\bar{L}}$:
    equals 
    $$\textstyle \lambda_1 \log \left(\tilde{z}_1\tilde{z}_2^2\tilde{z}_3\right) + \lambda_2 \left(\log  (\tilde{z}_2\tilde{z}_4) + h_2(q_1,q_2)\right) = 0.$$
    This gives $T^a \tilde{z}_1 + T^b \tilde{z}_2+ \frac{T^c}{\tilde{z}_1\tilde{z}_2^2\exp h_1(q_1,q_2)} + \frac{T^d}{\tilde{z}_2 \exp h_2(q_1,q_2)}$ as the potential of $\bar{L}$.  This equals the above expression as Laurent polynomials in $\tilde{z}$ by some change of coordinates of the form $z_1 = T^a \tilde{z}_1(1+o_1(T))$ and $z_2 = T^b \tilde{z}_2(1+o_2(T))$ for some elements $o_1(T),o_2(T)$ in the Novikov ring $\Lambda_+$.  This forces $h_1=0$ and $h_2 = -\log (1+q^e)$.
\end{example}

In general, when $Y$ is not $\C^n$, even in very simple Fano situations, the Lagrangian correspondence $L^\pi$ can be obstructed.  The general theory of Fukaya-Oh-Ohta-Ono \cite[Theorem 3.8.41, Corollary 3.8.43]{FOOO} will be useful to kill the obstructions.  Namely, if the obstruction $m_0(L^\pi)$ (other than the unit term) lies in the image of $H^*(Y^-\times X) \to H^*(L^\pi)$, then one can take bulk deformation by an element in $H^*(Y\times X)$ to achieve weakly unobstructedness.  Intuitively, by using a bulk deformation, there are constant discs that have an interior marked point passing through the obstruction cochain.  Since these discs are constant, they evaluate at the output boundary marked point to the obstruction cochain itself, which contributes to kill the obstruction.

Once we achieve weakly unobstructedness by bulk deformation, $L^\pi$ has a well-defined equivariant disc potential.  We need this as quantum corrections in comparing the disc potentials before and after quotient via Equation \ref{eq:W+W=W}.

To understand the obstruction terms, let's take a closer look at Lemma \ref{lem:dics_class}, which states that $\pi_2(Y^- \times X,L^\pi)$ is generated by some lifting of the disc classes $\beta_i$ for $i\in I$ given in Lemma \ref{lem:disc-moment} and sphere classes of $Y$ and $X$.  However, note that there are different liftings of $\beta_i$ to $\pi_2(Y^- \times X,L^\pi)$ in general.  Indeed, different liftings can have different Maslov indices.  

To fix this ambiguity, we stratify the moment level set $\mu^{-1}\{c\}$ by its intersections with the (open) toric strata of $Y$, namely, $\mu^{-1}\{c\} = \coprod_B \mu^{-1}\{c\} \cap B$ where $B$ runs over the toric strata of $Y$ such that $\mu^{-1}\{c\} \cap B \not= \emptyset$.  When the closure $\bar{B}$ (which is itself a toric manifold) contains $\beta_i$ as a basic disc class bounded by a regular toric fiber $F$ of $B$ and $\partial\beta_i\in\pi_1(F)$ descends to a constant point in the quotient $F/T \subset X$, we denote by $\beta_i^B$ to be the corresponding disc class of $\pi_2(Y^- \times X,L^\pi)$ which is constant in the $X$ factor.

The existence of obstruction is due to the fact that these disc classes $\beta_i^B$ can have non-positive Maslov indices.  The work of Cho-Kim \cite{Cho-Kim} is useful to find the Maslov indices of these disc classes.

\begin{theorem}[Theorem 3.7 of \cite{Cho-Kim}]
Let $X$ be a symplectic manifold with a Hamiltonian $\bS^1$-action, and let $H:X\to\R$ be the Hamiltonian.  Suppose $L$ is an $\bS^1$-invariant Lagrangian submanifold of $M$ contained in a level set of $H$.  For any gradient holomorphic disc $u$, its Maslov index equals $-2 n_z$, where $n_z$ is the sum of weights at the unique fixed point $z$ in the image of $u$.
\end{theorem}

In general, a toric quotient is an orbifold.  For simplicity, we restrict to the situation that the quotient is a smooth manifold, in which case the Maslov indices are integers.

\begin{prop} \label{prop:Maslov}
   Let $Y$ be a toric manifold and $X$ be a symplectic quotient with respect to a subtorus $T^k$ action.  Assume that the action is free so that $X$ is smooth.  For the disc class $\beta_i^B$ described above, its Maslov index is given by
    $\mu(\beta_i^B) = 2 \sum_j n_j $
   where $v = \sum_j n_j v_j$ is a primitive vector that has $n_i>0$ and lies in the intersection of $\mathfrak{t}^k$ and the normal space of the moment map image of $B \cap D_i$.
\end{prop}

\begin{proof}
    By the assumption that $\partial \beta_i^B$ descends to a point in the quotient in the definition of $\beta_i^B$, there must be a vector in $\mathfrak{t}^k$ that preserves and acts non-trivially on $\partial \beta_i^B$.  Such a vector is contained in the normal space of the moment map image of $B \cap D_i$ and has a non-zero coefficient in $v_i$.  Thus the vector $v$ stated above exists.

    Since $v$ lies in $\mathfrak{t}^k$, it preserves the moment map level set $\mu^{-1}\{c\}$.  Then $(v,0)$ gives a Hamiltonian action on $Y^-\times X$ that preserves $L^\pi$.  Moreover, since it lies in the normal space of the moment map image of $B \cap D_i$, it preserves a basic holomorphic disc in $B$ that represent the class $\beta_i^B$.  As $n_i>0$, it acts on this basic disc non-trivially and has exactly one fixed point, which is the intersection point of the disc with $B \cap D_i$.  Thus the above theorem of \cite{Cho-Kim} can be applied.  The weights at the fixed points are negative of the coefficients $n_j$ of the expression of $v$.
\end{proof}

Below, we give some examples for $\beta_i^B$, their Maslov indices, and explain how they affect the obstructions and disc potentials.  

\begin{example} \label{ex:P1P1}
    Let $Y = \bP^1 \times \bP^1$ and $\bS^1 = \langle (1,1) \rangle \subset T^2$.  Let $p,q$ be the symplectic areas of the two factors of $\bP^1 \subset Y$ respectively.
    Some moment level sets for different choices of $c$ are depicted in Figure \ref{fig:P1TimesP1-moment}.

    It can be computed by Proposition \ref{prop:Maslov} that for the disc classes bounded by $L^\pi \subset Y^-\times X$ shown in Figure \ref{fig:P1TimesP1-moment}, 
        $\mu(\beta_1^{D_4})=0$, $\mu(\beta_3^{D_4}) = 4$, $\mu(\beta_1^{D_2}) = 4$,
        $\mu(\beta_3^{D_2}) = 0$,
        $\mu(\beta_2^{D_1})=0$, $\mu(\beta_4^{D_1})=4$.

    Note that the effective disc class $\beta_3^{D_2} \in H_2(Y^-\times X,L^\pi)$ does not correspond to an effective curve class of $H_2(Y^-\times X)$. Consider the curve class $[D_2]$ in $Y$. Let $l$ be the fundamental class of the quotient $X \cong \bP^1$. Indeed $\beta_3^{D_2} = [D_2] - l$, which is not an effective curve class. Similarly, $\beta_2^{D_1} = [D_1] - l$ is not an effective curve class.
    
    The discs with Maslov index zero in class $\beta_3^{D_2}$ and $\beta_2^{D_1}$ will contribute to obstructions of $L^\pi$.  This is also manifested in comparing the disc potentials of $Y$ and $X$.
    Let's first focus on the case on the left.  For $L \subset Y$, its equivariant disc potential is
    $$ \textstyle W_L = T^a x + T^b y + \frac{T^c}{x} + \frac{T^d}{y} + \lambda \log xy$$
    where $a+c=p$ and $b+d=q$. Let's write it as
    $ W_L = T^a\left(x + \frac{T^{d-a}}{y}\right) + T^b\left(y+\frac{T^{c-b}}{x}\right)+ \lambda \log xy. $
    The equivariant part $\lambda \cdot \log xy$ tells us that setting $xy=1$ should be related to the potential of the quotient $\bar{L}\subset X$.  This will give the expression
    $ T^a x (1+T^{d-a}) + \frac{T^b}{x}(1+T^{c-b}) $
    which is not really the potential of $\bar{L}\subset X$.  
    
    Indeed, we need to take a bulk deformation by divisor classes $\pi_2^*([p_0])$ and $\pi_2^*([p_\infty])$ in $H^2(Y\times X)$ respectively, where $p_0,p_\infty$ are the toric divisors of $X=\bP^1$, to kill the (degree-two) obstruction in $L^\pi$ contributed by $\beta_2^{D_1},\beta_3^{D_2}$ that have areas $(d-a)$ and $(c-b)$ respectively.  Correspondingly, $X$ is also bulk-deformed by the divisor classes $p_0,p_\infty$, which accounts for the factors $(1+T^{d-a})$ and $(1+T^{c-b})$.  The analysis for the case on the right of Figure \ref{fig:P1TimesP1-moment} is similar and is left to the reader.
\end{example}

\begin{figure}[h]
\begin{center}
\includegraphics[scale=0.4]{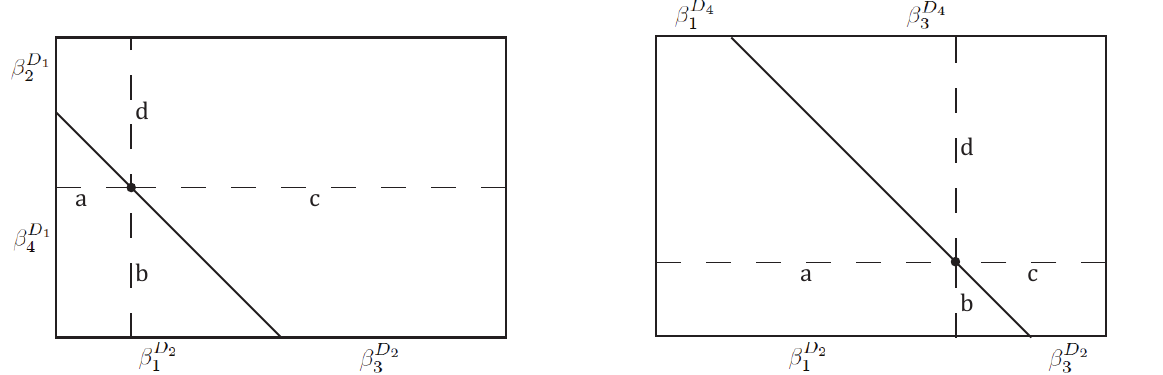}
\caption{\small{Symplectic quotients of $\bP^1 \times \bP^1$.}}
\label{fig:P1TimesP1-moment}
\end{center}
\end{figure}

\begin{example} \label{ex:P1P1P1-2}
    We go back to Example \ref{ex:P1P1P1} in the Introduction.  The equivariant disc potential of $L \subset Y = (\bP^1)^3$ is
    $$ \textstyle W_L = T^{a_1} z_1 + T^{a_2} z_2 + T^{a_3} z_3 + \frac{T^{b_1}}{z_1} + \frac{T^{b_2}}{z_2} + \frac{T^{b_3}}{z_3} + \lambda \log z_1z_2z_3$$
    where $(a_1,a_2,a_3)$ are the base coordinates in the moment-map cube of $L$, $(a_1+b_1),(a_2+b_2),(a_3+b_3)$ are the symplectic areas of the three coordinate lines of $Y=(\bP^1)^3$, and $a_1+a_2+a_3=c$ is the moment-map level.  The equivariant part tells us to set $z_1z_2z_3=1$, and we get
    $$\textstyle T^{a_1} z_1 + T^{a_2} z_2 + \frac{T^{a_3}}{z_1z_2} + \frac{T^{b_1}}{z_1} + \frac{T^{b_2}}{z_2} + T^{b_3}z_1z_2.$$
    On the other hand, the disc potential of the quotient $\bar{L} \subset X$ is
     $T^{a_1} z_1 + T^{a_2} z_2 + \frac{T^{c-a_1-a_2}}{z_1z_2}. $

    In this case, $L^\pi$ bounds holomorphic discs of Maslov index $(-2)$ drawn in Figure \ref{fig:P1P1P1}.  These contribute to degree-four obstructions in $m_0^{L^\pi}$.  In order to kill these obstructions of $L^\pi$, we can take a bulk deformation by $\pi_2^*([p_1]),\pi_2^*([p_2]),\pi_2^*([p_3]) \in H^4(Y\times X)$, where $p_1=[1:0:0],p_2=[0:1:0],p_3=[0:0:1]$.  Then constant discs with interior marked points mapped to $(\pi^{-1}\{p_i\},p_i) \subset L^\pi$ contribute to cancel the obstructions $[\pi^{-1}\{p_i\},p_i] \in H^4(L^\pi)$, where $\pi:\mu^{-1}\{c\} \to X$ denotes the quotient map.  Correspondingly, $X$ also needs to be bulk-deformed by $[p_1],[p_2],[p_3]\in H^4(X)$, which lead to contributions from Maslov-four disc classes in $X=\bP^2$ that produce the extra terms $z_1^{-1}$,$z_2^{-1}$ and $z_1z_2$.

    Alternatively, we can take bulk deformation by $\pi_1^*([A_1]),\pi_1^*([A_2]),\pi_1^*([A_3]) \in H^4(Y\times X)$, where $[A_i] \in H^4(Y)$ are the three coordinate axes of $Y = (\bP^1)^3$ that contain the discs of Maslov index $(-2)$ for $L^\pi$.  This has the same effect of turning on constant discs with interior marked points mapped to $(\pi^{-1}\{p_i\},p_i) \subset L^\pi$ to cancel the obstructions.  Correspondingly, $Y$ is bulk-deformed by $[A_1],[A_2],[A_3]\in H^4(Y)$, which leads to contributions from Maslov-four disc classes in $Y$ that produce extra terms $z_2z_3,z_1z_3,z_1z_2$ to $W_{Y}$.
    
\end{example}

Since $L^\pi$ is non-toric, explicit expressions of general holomorphic discs is not available.  On the other hand, some representatives are easy to write down.  For instance, we have used basic disc classes (bounded by possibly degenerate toric fibers) to represent the classes $\beta_i^B$.  In above, we have considered discs in $Y$ whose boundary becomes a point under quotient.  More generally, we can make use of symplectic involution to have a general form of a holomorphic disc in $Y^- \times X$.  The Lemma below is easy to prove.

\begin{lemma} \label{lem:ubar}
  Let $Y$ be a symplectic quotient of $\C^n$ by a linear subtorus action $\rho: T^k \times \C^n \to \C^n$.  let $\sigma: \C^n \to \C^n$ be defined by $\sigma(z_1,\ldots,z_n) = (\bar{z}_1,\ldots,\bar{z}_n)$.
  \begin{enumerate}
    \item $\sigma$ is an anti-symplectic involution on $(\C^n,\omega_{\C^n})$ for the standard symplectic form $\omega_{\C^n}$, namely, $\sigma^* \omega_{\C^n} = -\omega_{\C^n}$.  In other words, $\sigma$ is a symplectomorphism $(\C^n,\omega_{\C^n}) \cong (\C^n,-\omega_{\C^n})$.
    \item $\sigma \circ \rho_t = \rho_{-t} \circ \sigma$, where $\rho_t(x):=\rho(x,t)$.  In particular, $\sigma$ descends to a diffeomorphism on $Y$, which is still denoted by $\sigma$.
    \item $\sigma$ is an anti-symplectic involution on $(Y,\omega_Y)$.  It maps every toric fiber of $Y$ back to itself.
    \item Denote the symplectic manifold $(Y,-\omega)$ by $Y^-$.  Let $F \subset Y$ be a toric fiber.  There is a one-to-one correspondence between holomorphic disc $u:(\Delta,\partial \Delta) \to (Y,F)$ and that of $(Y^-,F)$ by $u \mapsto \bar{u} := \sigma \circ u$.
  \end{enumerate}
\end{lemma}

Using the above symplectic involution, a general form of a holomorphic disc can be written as follows.
\begin{prop}
    A holomorphic disc bounded by the Lagrangian correspondence $L^\pi \subset Y^- \times X$ is of the form $(\bar{u}_Y,u_X)$, where $u_Y: \Delta \to Y$ is a holomorphic disc bounded by the moment level set $\mu^{-1}\{c\}$ of $Y$ and $u_X: \Delta \to X$ is a holomorphic disc such that $u_X|_{\partial\Delta}$ agrees with the composition of the quotient map and $\bar{u}_Y|_{\partial\Delta}$.    
\end{prop} 

\begin{proof}
    Given a holomorphic disc $u: (\Delta,\partial\Delta) \to (Y^- \times X,L^\pi)$, its projections to $Y^-$ and to $X$ are holomorphic.  By Lemma \ref{lem:ubar}, the projection to $Y^-$ must be $\bar{u}_Y$ for some holomorphic disc of $Y$.  The projection to $X$ is denoted by $u_X$.  By the boundary condition $u(\partial\Delta) \subset L^\pi$, it follows that $\bar{u}_Y(\partial\Delta)\subset \mu^{-1}\{c\}$ and $u_X|_{\partial\Delta} = [\bar{u}_Y|_{\partial\Delta}]$.  Moreover, since the involution $\sigma:Y\to Y$ preserves all toric fibers, $\sigma(\mu^{-1}\{c\}) = \mu^{-1}\{c\}$, and hence $u_Y(\partial\Delta)\subset \mu^{-1}\{c\}$.
\end{proof}

\begin{example}
    In Example \ref{ex:F2}, we have seen that the exceptional curve $e$ in $X=\bF_2$ contributes to the equivariant disc potential of $L^\pi$.  We can depict such discs in the above form $(\bar{u}_Y,u_X)$.  Namely, $\bar{u}_Y$ is taken as the conjugation of a basic holomorphic disc bounded by a degenerate toric fiber $T^3$ in the coordinate hyperplane $\{x_2=0\}$ of $\C^4$; $u_X$ is taken as a basic holomorphic disc bounded by a degenerate toric fiber $T^1$ contained in the exceptional curve $e$ of $\bF_2$, whose boundary is negative of the quotient image of the boundary of $\bar{u}_Y$.  This gives a Maslov-zero holomorphic disc that contributes to the term $q^e$ of the equivariant disc potential of $L^\pi$.  See Figure \ref{fig:F2}.
\end{example}

\begin{figure}[h]
\begin{center}
\includegraphics[scale=0.25]{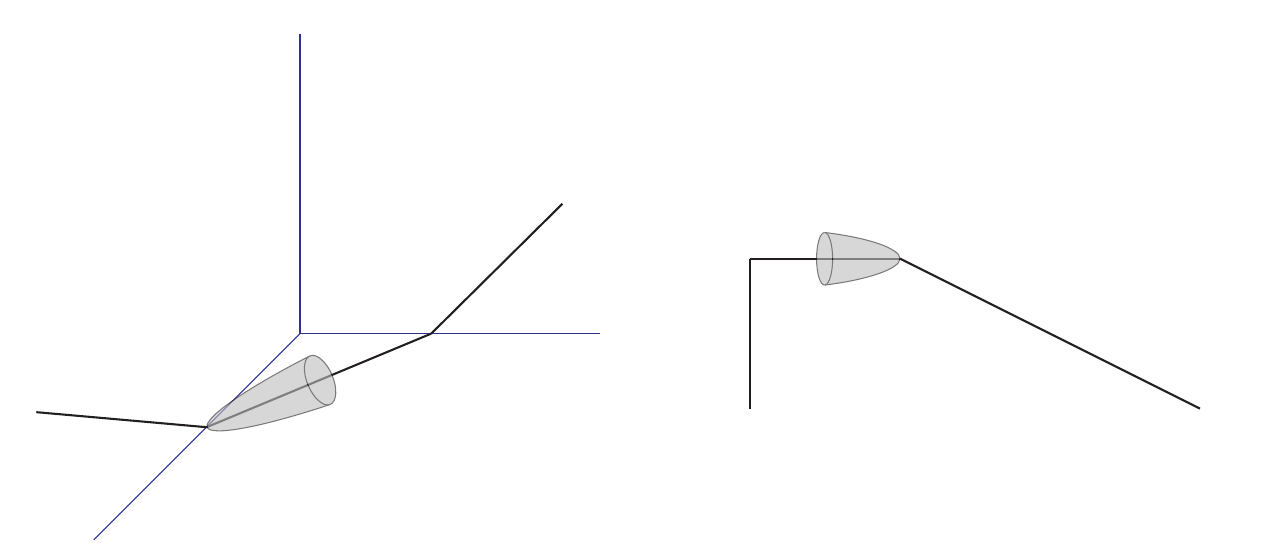}
\caption{\small{Holomorphic discs in $(\C^4 \times \bF_2,L^\pi)$ that contribute to the equivariant disc potential.  The infinite divisor has been taken away in the picture to reduce from $\C^4$ to $\C^3$.}}
\label{fig:F2}
\end{center}
\end{figure}

\bibliographystyle{plain}
\bibliography{geometry}
	
\end{document}